\documentclass[12pt]{amsart}
\usepackage{amssymb,amsfonts,amsmath,amsopn,amstext,amscd,color,latexsym,leqno,mathrsfs,url,xy}
\usepackage{extraipa}
\usepackage{tipa}
\usepackage{undertilde}
\usepackage{ushort}
\xyoption{all}

\theoremstyle{remark}

\newtheorem{para}{\bf}[subsection]

\theoremstyle{definition}

\newtheorem{dfn}[para]{Definition}

\theoremstyle{plain}

\newtheorem{thm}[para]{Theorem}
\newtheorem{lemma}[para]{Lemma}

\newtheorem{cor}[para]{Corollary}
\newtheorem{prop}[para]{Proposition}

\newenvironment{numequation}
{\addtocounter{para}{1}\begin{equation}}{\end{equation}}

\newcommand{\al}{{\alpha}}

\newcommand{\bbA}{{\mathbb A}}

\newcommand{\bbC}{{\mathbb C}}

\newcommand{\bbG}{{\mathbb G}}

\newcommand{\bbN}{{\mathbb N}}

\newcommand{\bbQ}{{\mathbb Q}}
\newcommand{\bbR}{{\mathbb R}}

\newcommand{\bbZ}{{\mathbb Z}}

\newcommand{\bB}{{\bf B}}
\newcommand{\bC}{{\bf C}}

\newcommand{\bG}{{\bf G}}

\newcommand{\bN}{{\bf N}}

\newcommand{\bT}{{\bf T}}
\newcommand{\bU}{{\bf U}}

\newcommand{\bb}{{\bf b}}

\newcommand{\frb}{{\mathfrak b}}

\newcommand{\frg}{{\mathfrak g}}
\newcommand{\frh}{{\mathfrak h}}

\newcommand{\frn}{{\mathfrak n}}

\newcommand{\frt}{{\mathfrak t}}

\newcommand{\frx}{{\mathfrak x}}

\newcommand{\frB}{{\mathfrak B}}

\newcommand{\frG}{{\mathfrak G}}

\newcommand{\frT}{{\mathfrak T}}

\newcommand{\frX}{{\mathfrak X}}
\newcommand{\frY}{{\mathfrak Y}}

\newcommand{\cA}{{\mathcal A}}

\newcommand{\cC}{{\mathcal C}}
\newcommand{\cD}{{\mathcal D}}

\newcommand{\cF}{{\mathcal F}}

\newcommand{\cI}{{\mathcal I}}

\newcommand{\cL}{{\mathcal L}}
\newcommand{\cM}{{\mathcal M}}
\newcommand{\cN}{{\mathcal N}}
\newcommand{\cO}{{\mathcal O}}

\newcommand{\cT}{{\mathcal T}}

\newcommand{\sB}{{\mathscr B}}
\newcommand{\sC}{{\mathscr C}}
\newcommand{\sD}{{\mathscr D}}

\newcommand{\sF}{\vartheta^{-1}_{\bf B}}

\newcommand{\sI}{{\mathscr I}}

\newcommand{\sK}{{\mathscr K}}

\newcommand{\sM}{{\mathscr M}}

\newcommand{\End}{{\rm End}}

\newcommand{\Hom}{{\rm Hom}}

\newcommand{\lra}{\longrightarrow}

\newcommand{\ra}{\rightarrow}

\newcommand{\smprod}{\#}
\newcommand{\cotimes}{\hat{\otimes}}
\newcommand{\ci}{\circ}
\newcommand{\car}{\stackrel{\cong}{\longrightarrow}}

\newcommand{\UFe}{U_F^{(e)}}
\newcommand{\UFbe}{U_{F'}^{(e)}}
\newcommand{\Uxne}{U_{x_0}^{(e)}}
\newcommand{\Uze}{U_z^{(e)}}

\newcommand{\comFe}{\mathring{\omega}_F^{(e)}}

\newcommand{\Gan}{\bG^{an}}

\newcommand{\Xan}{X^{an}}

\newcommand{\sta}{\cO_{\sB,z}}

\newcommand{\shf}{(\cO_{\sB}\smprod\tiD)}\newcommand{\shfo}{\cO_{\sB}\smprod\tiD}
\newcommand{\shfu}{\cO_{\sB}\smprod U(\frg)_K}

\newcommand{\sDc}{\sD_{r,\chi}}\newcommand{\sDcz}{\sD_{r,\chi,z}}

\newcommand{\tiD}{\underline{D}_r}
\newcommand{\tiDt}{\underline{D}_{r,\theta}}
\newcommand{\tiM}{\underline{M}_r}

\newcommand{\func}{\mathscr{L}_{r,\chi}}

\def\Vsim{\raise1pt\hbox{$\mathop{V}\limits_{\approx}$}} 
\def\cVsim{\raise1pt\hbox{$\mathop{\check{V}}\limits_{\approx}$}} 
\def\cVrsim{\raise1pt\hbox{$\mathop{\check{V}}\limits_{\approx,r}$}} 
\newcommand{\tiV}{\Vsim}

\begin{document}

\title{Locally analytic representations
and sheaves on the Bruhat-Tits building}

\author{Deepam Patel}
\address{Department of Mathematics, VU University Amsterdam - Faculty of Sciences,
De Boelelaan 1081a, 1081 HV Amsterdam, The Netherlands}
\email{deeppatel1981@gmail.com}
\author{Tobias Schmidt}
\address{Mathematisches Institut, Westf\"alische Wilhelms-Universit\"at M\"unster, Einsteinstr.
62, D-48149 M\"unster, Germany}
\email{toschmid@math.uni-muenster.de}
\author{Matthias Strauch}
\address{Indiana University, Department of Mathematics, Rawles Hall, Bloomington, IN 47405, U.S.A.}
\email{mstrauch@indiana.edu}
\thanks{M. S. would like to acknowledge the support of the National Science Foundation (award numbers DMS-0902103 and DMS-1202303).}
\maketitle

\normalsize

\begin{abstract}
Let $L$ be a finite field extension of
$\bbQ_p$ and let $G$ be the group of $L$-rational points of a split
connected reductive group over $L$. We view $G$ as a locally
$L$-analytic group with Lie algebra $\frg$. The purpose of this
work is to propose a construction which extends the localization
of smooth $G$-representations of P. Schneider and U. Stuhler
(\cite{SchSt97}) to the case of locally analytic
$G$-representations. We define a functor from admissible locally analytic
$G$-representations with prescribed infinitesimal character to a
category of equivariant sheaves on the Bruhat-Tits building of
$G$. For smooth representations, the corresponding sheaves are closely related to
the sheaves of Schneider and Stuhler. The functor is also compatible, in a certain sense, with the localization of $\frg$-modules on the flag variety by A.
Beilinson and J. Bernstein (\cite{BB81}).

\end{abstract}

\tableofcontents

\vskip30pt

\section{Introduction}

Let $L$ be a finite field extension of
the field $\bbQ_p$ of $p$-adic numbers. Let $\bG$ be a connected
split reductive group over $L$ and $\bB\subset\bG$ a Borel
subgroup defined over $L$. Let $\bT\subset\bG$ be a maximal torus
contained in $\bB$. Let $G:=\bG(L)$, $T:=\bT(L)$ denote the groups
of rational points, viewed as a locally $L$-analytic groups. Let
$\frg$ and $\frt$ be the corresponding Lie algebras.

\vskip8pt

The purpose of this work is to propose a construction which
extends the localization theory for smooth $G$-representations of P.
Schneider and U. Stuhler (\cite{SchSt97}) to the case of locally
analytic $G$-representations. In more concrete terms, we define
an exact functor from admissible locally analytic
$G$-representations with prescribed infinitesimal character to a
category of equivariant sheaves on the Bruhat-Tits building of
$G$. The functor is also compatible, in a certain sense, with
the localization theory for $\frg$-modules on the flag variety of $\bG$ by
A. Beilinson and J. Bernstein (\cite{BB81}),
and J.-L. Brylinski and M. Kashiwara (\cite{BK80}, \cite{BK81}).

\vskip8pt

To give more details, let $\sB$ be the (semisimple) Bruhat-Tits
building of $G$. The torus $\bT$ determines an apartment $A$ in
$\sB$. We fix a fundamental chamber $\sC\subset A$ and a special
vertex $x_0\in \overline{\sC}$ which will be used as an origin for
the affine space $A$. In \cite{SchSt97} the authors consider, for
any point $z\in\sB$, a well-behaved filtration
$$P_z \supset U_z^{(0)} \supset U_z^{(1)}\supset...$$

of the pointwise stabilizer $P_z$ of $z$ in $G$ by open pro-$p$
subgroups $\Uze$. It forms a fundamental system of neighborhoods
of $1\in P_z$. The group $\Uze$ does not change when $z$ varies in a facet
of the building. Let from now on $e\geq 0$ be a fixed number (called
a {\it level} in loc.cit.).  


\vskip8pt

Using the groups $\Uze$, Schneider and Stuhler define in \cite[sec. IV]{SchSt97}
an exact functor
$$V\mapsto\tiV$$ from smooth complex
$G$-representations to sheaves of complex vector spaces on $\sB$.
The stalk of the sheaf $\tiV$ at a point $z$ is given by the
coinvariants $V_{\Uze}$ and the restriction of $\tiV$ to a facet
$F\subset\sB$ equals the constant sheaf with fibre $V_{\UFe}$. The
functor $V\mapsto\tiV$ has particularly good properties when
restricted to the subcategory of representations generated by
their $\Uxne$-fixed vectors. It is a major tool in the proof of
the Zelevinsky conjecture (loc.cit.).

\vskip8pt

From now on we fix a complete discretely valued field extension $K$ of $L$. The
functor $V\mapsto\tiV$ can be defined in exactly the same way for
smooth $G$-representations on $K$-vector spaces, and produces
sheaves of $K$-vector spaces on $\sB$. The naive extension of the
functor $V\mapsto\tiV$ to locally analytic representations, by
taking coinvariants as above, does not have good properties. For
instance, applying this procedure to an irreducible
finite-dimensional algebraic representation, which is not the
trivial representation, produces the zero sheaf. Moreover, if we
aim at a picture which is related to the localization theory of
$\frg$-modules, then localizing an irreducible algebraic
representation should give a line bundle.

\vskip8pt

We consider the variety of Borel subgroups
$$X=\bG/\bB$$ of $\bG$. We let $\cO_X$ be its structure sheaf and $\cD_X$
be its sheaf of differential operators. Deriving the left regular
action of $\bG$ on $X$ yields an algebra homomorphism $$\alpha:
\underline{U}(\frg)\longrightarrow\cD_X$$ where the source refers
to the constant sheaf on $X$ with fibre equal to the universal
enveloping algebra $U(\frg)$ of $\frg$. Let $Z(\frg)$ be the
center of the ring $U(\frg)$.

\vskip8pt

The torus $\bT$ determines a root system. Let $\rho$ be half the
sum over the positive roots with respect to $\bB$. For any
algebraic character $\chi-\rho$ of the torus $\bT$ we have the
sheaf $\cD_\chi$ of differential endomorphisms of the line bundle
on $X$ associated with $\chi-\rho$. Any trivialization of the line
bundle induces a local isomorphism between $\cD_\chi$ and $\cD_X$,
and we have $\cD_\rho = \cD_X$. More generally, if $\chi-\rho$ is
an arbitrary character of $\frt$ there is a sheaf of so-called
{\it twisted} differential operators $\cD_\chi$ on $X$. As in the
former case it comes equipped with a morphism
$\cO_X\hookrightarrow\cD_\chi$ which is locally isomorphic to the
canonical morphism $\cO_X\hookrightarrow\cD_X$. Moreover, there is
an algebra homomorphism $\underline{U}(\frg)\rightarrow\cD_\chi$
locally isomorphic to $\alpha$. The sheaf
$\cD_\chi$ was first introduced in
\cite{BB81} as a certain quotient sheaf of the skew tensor
product algebra $\cO_X\smprod U(\frg)$. We use this notation
('$\smprod$') to indicate that the multiplication on the tensor
product $\cO_X \otimes U(\frg)$ involves the action of $U(\frg)$
on $\cO_X$.

\vskip8pt

Let $\chi$ be a character of $\frt$. Let $\theta$ be the character
of $Z(\frg)$ associated with $\chi$ via the classical
Harish-Chandra homomorphism. The above map factors via a
homomorphism
$$\underline{U}(\frg)_\theta\longrightarrow\cD_\chi$$ where
$U(\frg)_\theta=U(\frg)\otimes_{Z(\frg),\theta} L$. If $\chi$ is
dominant and regular, a version of the localization theorem
(\cite{BB81}) asserts that the functor
$$ \Delta_\chi: M\mapsto\cD_\chi \otimes_{\underline{U}(\frg)_\theta} \underline{M}$$
is an equivalence of categories between the (left)
$U(\frg)_\theta$-modules and the (left) $\cD_\chi$-modules which
are quasi-coherent as $\cO_X$-modules. The underlined objects
refer to the associated constant sheaves on $X$. We remark that a
seminal application of this theorem (or rather its complex
version) leads to a proof of the Kazhdan-Lusztig multiplicity conjecture
(cf. \cite{BB81}, \cite{BK80}, \cite{BK81}).

\vskip8pt

The starting point of our work is a result of V. Berkovich
(\cite{BerkovichBook}, \cite{RemyThuillierWerner10}) according to
which the building $\sB$ may be viewed as a locally closed
subspace$$\sB\hookrightarrow\Xan$$ of the Berkovich
analytification $\Xan$ of $X$. This makes it possible to 'compare'
sheaves on $\sB$ and $\Xan$ in various ways. Most of what has been
said above about the scheme $X$ extends to the analytic space
$\Xan$. In particular, there is an analytic version
$\cD^{an}_\chi$ of $\cD_\chi$ and an analytic version
$\Delta_\chi(\cdot)^{an}$ of the functor $\Delta_\chi$ (sec. 6).

\vskip8pt For technical reasons we have to assume at some point in
this paper that $L=\bbQ_p$, with $p>2$ an odd prime. (However, we have no doubts that our
results eventually extend to general $L$ and $p$). To describe our
proposed locally analytic 'localization functor' under this
assumption we let $D(G)$ be the algebra of $K$-valued locally
analytic distributions on $G$. It naturally contains $U(\frg)$.
Recall that the category of admissible locally analytic
$G$-representations over $K$ (in the sense of P. Schneider and J.
Teitelbaum, cf. \cite{ST5}) is anti-equivalent to a full abelian
subcategory of the (left) $D(G)$-modules, the so-called
coadmissible modules. A similar result holds over any compact open
subgroup $\Uze$.

\vskip8pt From now on we fix a central character
$$\theta: Z(\frg_K)\longrightarrow K$$ and a toral character $\chi\in\frt^*_K$ associated to $\theta$ via
the classical Harish-Chandra homomorphism. To give an example, the
usual augmentation map $K[T]\rightarrow K$ of the group ring
$K[T]$ induces a character $\lambda_0$ of $D(T)$ such that
$\chi=\rho$ and with $\theta_0$ equal to the trivial infinitesimal
character, i.e., $\ker\theta_0= Z(\frg_K)\cap U(\frg_K)\frg_K$.
The ring $Z(\frg_K)$ lies in the center of the ring $D(G)$ (\cite{ST4})
so that we may consider the central reduction

$$D(G)_\theta := D(G)\otimes_{Z(\frg_K),\theta} K.$$
We propose to study the abelian category of (left)
$D(G)_\theta$-modules which are coadmissible over $D(G)$. As
remarked above it is anti-equivalent to the category of admissible
locally analytic $G$-representations over $K$ with infinitesimal
character $\theta$. We emphasize that {\it any} topologically
irreducible admissible locally analytic $G$-representation admits,
up to a finite extension of $K$, an infinitesimal character
(\cite{DospSchraen}).

\vskip8pt

To start with, consider a point $z\in\sB$. The group $\Uze$ carries
a natural $p$-valuation in the sense of M. Lazard, cf.
\cite{Lazard65}. According to the general locally analytic theory
(\cite{ST5}), this induces a family of norms $||.||_r$ on the
distribution algebra $D(\Uze)$ for $r\in [r_0,1)$ where
$r_0:=p^{-1}$. We let $D_r(\Uze)$ be the corresponding completion
of $D(\Uze)$ and $D_r(\Uze)_\theta$ its central reduction. In 8.2
we introduce sheaves of distribution algebras $\tiD$ and $\tiDt$
on $\sB$ with stalks
$$(\tiD)_z=D_r(\Uze),~~~~~~(\tiDt)_z=D_r(\Uze)_\theta$$ for all points $z\in\sB$. The inclusions $U(\frg)\subset D_r(\Uze)$ sheafify to a morphism
$\underline{U}(\frg_K)_\theta\ra \tiDt.$ Similarly, for any
coadmissible $D(G)_\theta$-module $M$ we consider a
$\tiDt$-module $\tiM$ on $\sB$ having stalks
$$(\tiM)_z=D_r(\Uze)_\theta\otimes_{D(\Uze)_\theta} M$$ for all points $z\in\sB$. The formation of $\tiM$ is functorial in $M$.
The sheaves $\tiDt,\tiM$ are constructible and will formally
replace the constant sheaves appearing in the definition of the
functors $\Delta_\chi,\Delta_\chi^{an}$. To simplify the
exposition in this introduction we assume from now on that the
level $e\geq 0$ is sufficiently large.

\vskip8pt

Consider the restriction of the structure sheaf of $X^{an}$ to $\sB$, i.e.,
$$\cO_\sB = \cO_{X^{an}}|_\sB \;.$$

We then define a sheaf of
noncommutative rings $\sDc$ on $\sB$ which is also a module over
$\cO_{\sB}$ and which is vaguely reminiscent of a 'sheaf of
twisted differential operators'. It has a natural $G$-equivariant
structure. It depends on the level $e$, but, following the usage
of \cite{SchSt97}, we suppress this in our notation. More
important, it depends on the 'radius' $r$ which is genuine to the
locally analytic situation and is related to a choice of completed
distribution algebra $D_r(\Uze)$ at each point $z\in\sB$.
Completely analogous to constructing $\cD_\chi$ out of the skew
tensor product algebra $\cO_X\smprod U(\frg_K)$ (cf. \cite{BB81})
we obtain the sheaf $\sDc$ out of a skew tensor product
algebra of the form $\cO_{\sB}\smprod \tiD$.

\vskip8pt

To describe the sheaf $\sDc$ we observe first that, for any point
$z\in\sB$, the inclusion $\Uze\subset P_z$ implies that there is a
locally analytic $\Uze$-action on the analytic stalk
$\cO_{\sB,z}$. We therefore have the corresponding skew group ring
$\cO_{\sB,z}\smprod\Uze$ as well as the skew enveloping algebra
$\cO_{\sB,z}\smprod U(\frg)$, familiar objects from
noncommutative ring theory (\cite{MCR}). In sec. 3 and sec.
6.3/4 we explain how the completed tensor product
$$\cO_{\sB,z}\hat{\otimes}_L D_r(\Uze)$$ can be
endowed with a unique structure of a topological $K$-algebra such that
the $\cO_{\sB,z}$-linear maps

\[\begin{array}{clccccl}
 (\ast) \hskip14pt \cO_{\sB,z}\smprod\Uze\ra & \cO_{\sB,z}\hat{\otimes}_L D_r(\Uze) &&{\rm and} & &\cO_{\sB,z}\smprod
U(\frg)\ra& \cO_{\sB,z}\hat{\otimes}_L D_r(\Uze) \;, \\
\end{array}\]
induced by $\Uze\subset D(\Uze)^\times$ and $U(\frg)\subset D(\Uze)$ respectively, become ring homomorphisms. To emphasize
this skew multiplication we denote the target of the two maps in
$(*)$ by $\cO_{\sB,z}\smprod D_r(\Uze)$ keeping in mind that there
is a {\it completed} tensor product involved. This process leads
to a sheaf of $K$-algebras $\shfo$ on $\sB$ with stalks
$$(\shfo)_z=\cO_{\sB,z}\smprod D_r(\Uze)$$
at points $z\in\sB$. It comes equipped with a morphism
$\cO_\sB\smprod U(\frg)\ra\shfo$ giving back the second map in ($\ast$)
at a point $z\in\sB$.

\vskip8pt

To generalize the formalism of {\it twisting} to this new
situation we proceed similarly to \cite{BB81}. Let $\cT_{\Xan}$ be
the tangent sheaf of $\Xan$ and let
$\alpha^{an}:\frg\rightarrow\cT_{\Xan}$ be the analytification of
the map $\alpha|_\frg$. There is the sheaf of $L$-Lie
algebras

$$\frb^{\ci,an} := \ker\;(\cO_{\Xan} \otimes_L \frg \stackrel{\alpha^{an}}{\lra}\cT_{X^{an}}).$$
 \vskip8pt

The inclusion $\bT\subset\bB$ induces an isomorphism of Lie
algebras
$$\cO_{\Xan}\otimes_L \frt\car\frb^{\ci,an}/[\frb^{\ci,an},\frb^{\ci,an}].$$ We have
thus an obvious $\cO_{\Xan}$-linear extension of the character
$\chi-\rho$ of $\frt_K$ to $\frb^{\ci,an}\otimes_L K$. Its
kernel, restricted to the building $\sB$, generates a two-sided
ideal $\sI^{an}_\chi$ in $\cO_\sB \smprod\tiD$ and we set

$$\sD_{r,\chi} := \left(\cO_\sB \smprod\tiD)\right/\sI^{an}_\chi \;.$$
Let $\cD^{an}_{\sB,\chi}$ denote the restriction of
$\cD^{an}_{\chi}$ to the building $\sB$. The sheaf $\sD_{r,\chi}$
comes with an algebra homomorphism
$\cD^{an}_{\sB,\chi}\ra\sD_{r,\chi}$ induced from the inclusion
$\cO_{\sB}\smprod U(\frg_K)\rightarrow\cO_{\sB}\smprod\tiD$. Most
importantly, the canonical morphism
$\tiD\rightarrow\cO_{\sB}\smprod \tiD$ induces a canonical
morphism $\tiDt\ra\sDc$ making the diagram

\[\xymatrix{
\underline{U}(\frg_K)_\theta \ar[d]\ar[r] &  \cD^{an}_{\sB,\chi}\ar[d] \\
 \tiDt  \ar[r]& \sDc }
\]
commutative.

\vskip8pt In this situation we prove that
$$M\mapsto \mathscr{L}_{r,\chi}(M):=\sDc\otimes_{\tiDt}\tiM$$
is an exact covariant functor from coadmissible $D(G)_\theta$
modules into $G$-equivariant (left) $\sDc$-modules. The stalk of
the sheaf $\mathscr{L}_{r,\chi}(M)$ at a point $z\in\sB$ with
residue field $\kappa(z)$ equals the ($\chi-\rho$)-coinvariants of
the $\frt_K$-module
$$(\kappa(z)\hat{\otimes}_L \underline{M}_{r,z})/\frn_{\pi(z)}(\kappa(z)\hat{\otimes}_L
\underline{M}_{r,z})$$ as it should be (\cite{BB81}). Here,
$\frn_{\pi(z)}$ equals the nilpotent radical of the Borel
subalgebra of $\kappa(z)\otimes_L \frg$ defined by the point
$\pi(z)\in X$ where $\pi:\Xan\rightarrow X$ is the canonical map.

We tentatively call $\func$ a locally analytic 'localization
functor'. We suppress the dependence of $\mathscr{L}_{r,\chi}$ on
the level $e$ in our notation.

\vskip8pt

We prove the following compatibilities with the Schneider-Stuhler
and the Beilinson-Bernstein localizations. Suppose first that the
coadmissible module $M$ is associated to a {\it smooth}
$G$-representation $V$. Since $\frg M=0$ it has infinitesimal
character $\theta=\theta_0$ and the natural choice of twisting is
therefore $\chi=\rho$. We establish a canonical isomorphism (Thm. \ref{thm-compSchSt}) of $\cO_{\sB}$-modules
$$\mathscr{L}_{r_0,\rho}(M)\car \cO_{\sB}\otimes_L \check{\tiV}$$
where $\check{V}$ is the smooth dual of $V$ and $\check{\tiV}$ the
sheaf associated to $\check{V}$ by Schneider-Stuhler. The
isomorphism is natural in $M$.

\vskip8pt

Secondly, suppose the coadmissible module $M$ is associated to a
{\it finite dimensional algebraic} $G$-representation. The functor
$\Delta_\chi(\cdot)^{an}$ may be applied to its underlying
$\frg$-module and gives a $\cD^{an}_\chi$-module on $\Xan$ and
then, via restriction, a $\cD^{an}_{\sB,\chi}$-module
$\Delta_{\chi}(M)^{an}_{\sB}$ on $\sB$. We prove (Thm. \ref{thm-compBB}) that there is a number $r(M)\in [r_0,1)$ which
is intrinsic to $M$ and a canonical isomorphism of
$\cD^{an}_{\sB,\chi}$-modules
$$\mathscr{L}_{r,\chi}(M)\car\Delta_\chi(M)^{an}_\sB$$
for all $r\geq r(M)$. The isomorphism is natural in $M$.

\vskip8pt As a class of examples we finally investigate the
localizations of locally analytic representations in the image of
the functor $\cF^G_B$ introduced by S. Orlik and investigated in
\cite{OrlikStrauchJH}. The image of $\cF^G_B$ comprises a wide
class of interesting representations and contains all principal
series representations as well as all locally algebraic
representations (e.g. tensor products of smooth with algebraic
representations).

\vskip8pt

This paper is the first of a series of papers whose aim is to develop
a localization theory for locally analytic representations. Here we only
make a first step in this direction, focussing on the building and merging
the theory of Schneider and Stuhler with the theory of Beilinson and Bernstein,
resp. Brylinski and Kashiwara.
To get a more complete picture one has to extend the construction to a compactification
$\overline{\sB}$ of the building. The compactification which one would take here is,
of course, the closure of $\sB$ in $X^{an}$. Moreover, for intended applications like
functorial resolutions and the computation of Ext groups, one has to develop
a 'homological theory', in analogy to \cite[sec. II]{SchSt97}. However, the sheaves
produced in this way (using a compactification) would still have too many global sections.
For instance, the space of global sections would be a module for the ring of meromorphic functions on $X^{an}$ with poles outside
$\overline{\sB}$, and this is a very large ring. The aim would be to produce sheaves
whose global sections give back the $D(G)$-module one started with.
In the paper (\cite{PatelSchmidtStrauch}) we explore an approach (in the case of $GL(2)$)
which is based on the use of (a family of) semistable formal models $\frX$ of $X^{an}$,
and we replace $\cO_\sB$ by the pull-back of $\cO_\frX \otimes L$
via the specialization map $X^{an} \ra \frX$, and the r\^ole of $\sD_{r,\chi}$ is
played by arithmetic logarithmic differential operators. In this regard we want to
mention related works by C. Noot-Huyghe (\cite{NootHuyghe09}),
and K. Ardakov and S. Wadsley (\cite{AW}). While Noot-Huyghe studies localizations of
arithmetic $\sD$-modules on smooth formal models of $X$, Ardakov and Wadsley define and study localizations of representations of Iwasawa algebras on smooth models. Our present paper is in some sense complementary to these papers, as our focus is on non-compact groups.\vskip8pt

Despite the many aspects (like compactifications, homological theory, relation with formal models) that still have to be
explored, given the many technical details that one has to
take care of we thought it worthwhile to give an account
of the constructions as developed up to this point. \vskip8pt

{\bf Acknowledgments.} We thank Vladimir Berkovich for helpful correspondence on $p$-adic symmetric spaces and buildings.
T. S. gratefully acknowledges travel support by the SFB 878 ''Groups, Geometry $\&$ Actions'' at the University of M\"unster. D.P.
would like to thank Indiana University, Bloomington, for its support and hospitality.

\vskip8pt

{\bf Notations.} Let $p$ be an odd prime. Let $L/\bbQ_p$ be a finite extension
and $K \subseteq \bbC_p$ a complete discretely valued extension of $L$.
The absolute value $|.|$ on $\bbC_p$ is normalized by $|p|=p^{-1}$.
Let $o_L\subset L$ be the ring of integers and $\varpi_L\in o_L$ a
uniformizer. We denote by $v_L$ always the normalized $p$-adic
valuation on $L$, i.e. $v_L(\varpi)=1$. Let $n$ and $e(L/\bbQ_p)$
be the degree and the ramification index of the extension
$L/\bbQ_p$ respectively. Similarly, $o_K\subset K$ denotes the
integers in $K$ and $\varpi_K\in o_K$ a uniformizer. Let
$k:=o_K/(\varpi_K)$ denote the residue field of $K$.

\vskip8pt

The letter $\bG$ always denotes a connected reductive linear
algebraic group over $L$ which is split over $L$ and $G=\bG(L)$
denotes its group of rational points.

\vskip8pt

\section{Distribution algebras and locally analytic
representations}

For notions and notation from non-archimedean functional analysis
we refer to the book \cite{NFA}.

\subsection{Distribution algebras.}\label{subsec-dist} In this section we recall some
definitions and results about algebras of distributions attached
to locally analytic groups (\cite{ST4}, \cite{ST5}). We consider a
locally $L$-analytic group $H$ and denote by $C^{an}(H,K)$ the
locally convex $K$-vector space of locally $L$-analytic functions
on $H$ as defined in \cite{ST4}. The strong dual
$$D(H,K):=C^{an}(H,K)'_b$$ is the algebra of $K$-valued
locally analytic distributions on $H$ where the multiplication is
given by the usual convolution product. This multiplication is
separately continuous. However, if $H$ is compact, then $D(H,K)$
is a $K$-Fr\'echet algebra.

 The algebra $D(H,K)$ comes equipped
with a continuous $K$-algebra homomorphism
\[ \Delta: D(H,K)\lra D(H,K)\cotimes_{K,\iota} D(H,K)\]
which has all the usual properties of a comultiplication
(\cite[\S3 App.]{ST6}). Here, $\iota$ refers to the (complete)
inductive tensor product. Of course,
$\Delta(\delta_h)=\delta_h\otimes\delta_h$ for $h\in H$.

\vskip8pt

The universal enveloping algebra $U(\frh)$ of the Lie algebra
$\frh:=Lie(H)$ of $H$ acts naturally on $C^{an}(H,K)$. On elements
$\frx\in\frh$ this action is given by \[(\frx f)(h) = \frac{d}{dt}
(t\mapsto f(\exp_H(-t\frx)h))|_{t=0}\] where $\exp_H:\frh-->H$
denotes the exponential map of $H$, defined in a small
neighbourhood of $0$ in $\frh$. This gives rise to an embedding of
$U(\frh)_K:= U(\frh)\otimes_L K$ into $D(H,K)$ via

\[ U(\frh)_K\hookrightarrow D(H,K),~~\frx\mapsto
(f\mapsto(\dot{\frx}f)(1)).\]

Here $\frx\mapsto\dot{\frx}$ is the unique anti-automorphism of
the $K$-algebra $U(\frh)_K$ which induces multiplication by $-1$
on $\frh$. The comultiplication $\Delta$ restricted to $U(\frg)_K$
gives the usual comultiplication of the Hopf algebra $U(\frg)_K$,
i.e. $\Delta(\frx)=\frx\otimes 1+1\otimes\frx$ for all
$\frx\in\frh$.

\subsection{Norms and completions of distribution algebras}

\subsubsection{\it $p$-valuations.}
Let $H$ be a compact locally $\bbQ_p$-analytic group. Recall
(\cite{Lazard65}) that a {\it $p$-valuation} $\omega$ on $H$ is a
real valued function $\omega: H\setminus\{1\}\rightarrow
(1/(p-1),\infty)\subset\bbR$ satisfying\vskip8pt
\begin{itemize}
  \item[(i)] $\omega(gh^{-1}) \geq {\rm
  min~}(\omega(g),\omega(h)),$\vskip8pt
  \item[(ii)]
 $ \omega(g^{-1}h^{-1}gh) \geq \omega(g)+\omega(h),$\vskip8pt
  \item[(iii)]
 $ \omega(g^p) =\omega(g)+1$ \\
\end{itemize}
for all $g,h\in H$. As usual one puts $\omega(1):=\infty$ and
interprets the above inequalities in the obvious sense if a term
$\omega(1)$ occurs.

Let $\omega$ be a $p$-valuation on $H$. It follows from loc.cit.,
III.3.1.3/7/9 that the topology on $H$ is defined by $\omega$
(loc.cit., II.1.1.5), in particular, $H$ is a pro-$p$ group.
Moreover, there is a topological generating system $h_1,...,h_d$
of $H$ such that the map
\[\bbZ_p^d\rightarrow H,~~(a_1,...,a_d) \mapsto
h_{1}^{a_1}\cdot\cdot\cdot h_{d}^{a_d}\] is well-defined and a
homeomorphism. Moreover,
\[\omega ( h_{1}^{a_1}\cdot\cdot\cdot h_{d}^{a_d})=\min \{\omega (h_i) + v_p(a_i)
\mid i=1,...,d\}\] where $v_p$ denotes the $p$-adic valuation on
$\bbZ_p$. The sequence $(h_1,...,h_d)$ is called a {\it $p$-basis}
(or an {\it ordered basis}, cf. \cite[\S4]{ST5}) of the $p$-valued
group $(H,\omega).$

\vskip8pt

Finally, a $p$-valued group $(H,\omega)$ is called {\it
$p$-saturated} if any $g\in H$ such that $\omega(g)>p/(p-1)$ is a
$p$-th power in $H$.

\subsubsection{\it The canonical $p$-valuation on uniform pro-$p$ groups.}

We recall some definitions and results about pro-$p$ groups
(\cite[ch. 3,4]{DDMS}) in the case $p\neq 2$. In this subsection
$H$ will be a pro-$p$ group which is equipped with its topology as
a pro-finite group. Then $H$ is called {\it powerful} if $H/H^p$
is abelian. Here, $H^p$ is the closure of the subgroup generated
by the $p$-th powers of its elements. If $H$ is topologically
finitely generated one can show that the subgroup $H^p$ are open
and hence automatically closed. The lower $p$-series
$(P_i(H))_{i\geq 1}$ of an arbitrary pro-$p$ group $H$ is defined
inductively by
\[
P_1(H):=H,~~~~ P_{i+1}(H):=\overline{P_i(H)^p[P_i(H),H]}.
\]

If $H$ is topologically finitely generated, then the groups
$P_i(H)$ are all open in $H$ and form a fundamental system of
neighborhoods of $1$ (loc.cit, Prop. 1.16). A pro-$p$ group $H$ is
called {\it uniform} if it is topologically finitely generated,
powerful and its lower $p$-series satisfies $(H: P_2(H)) =
(P_i(H): P_{i+1}(H))$ for all $i \geq 1$. If $H$ is a
topologically finitely generated powerful pro-$p$ group then
$P_i(H)$ is a uniform pro-$p$ group for all sufficiently large $i$
(loc.cit. 4.2). Moreover, any compact $\bbQ_p$-analytic group
contains an open normal uniform pro-$p$ subgroup (loc.cit. 8.34).
Now let $H$ be a uniform pro-$p$ group. It carries a distinguished
$p$-valuation
 $\omega^{can}$ which is associated to the lower $p$-series and which we call the
{\it canonical $p$-valuation}. For $h\neq 1$, it is defined by
$\omega^{can}(h) = \max \{i\geq 1: h\in P_i(H)\}$.

\subsubsection{\it Norms arising from $p$-valuations.}\label{subsubsec-norms} In this
section we let $H$ be a compact $\bbQ_p$-analytic group endowed
with a $p$-valuation $\omega$ that has rational values.
For convenience of the reader we briefly recall (\cite[\S4]{ST5})
the construction of a suitable family of submultiplicative norms
$||.||_r, r\in [1/p,1)$ on the algebra $D(H,K)$.

\vskip8pt

Let $h_1,...,h_d$ be an ordered basis for $(H,\omega)$. The
homeomorphism $\psi: \bbZ_p^d\simeq H$ given by
$(a_1,...,a_d)\mapsto h_1^{a_1}\cdot\cdot\cdot h_d^{a_d}$ is a
global chart for the $\bbQ_p$-analytic manifold $H$. By
functoriality of $C^{an}(.,K)$ it induces an isomorphism $\psi^*:
C^{an}(H,K)\car C^{an}(\bbZ_p^d,K)$ of topological $K$-vector
spaces. Using Mahler expansions (\cite[III.1.2.4]{Lazard65}) we may
express elements of $C(\bbZ_p^d,K)$, the space of continuous
$K$-valued functions on $\bbZ_p^d$, as series
$f(x)=\sum_{\al\in\bbN_0^d}c_\al {{x}\choose {\al}}$ where
$c_\al\in K$ and ${{x}\choose {\al}}={{x_1}\choose
{\al_1}}\cdot\cdot\cdot {{x_d}\choose {\al_d}}$ for multi-indices
$x=(x_1,...,x_d)$ and $\al=(\al_1,..,\al_d)\in \bbN_0^d$. Further,
we have $|c_\al|\rightarrow 0$ for
$|\al|=\al_1+...+\al_d\rightarrow\infty$. A continuous function
$f\in C(\bbZ_p^d,K)$ is locally analytic if and only if
$|c_\al|r^{|\al|}\rightarrow 0$ for some real number $r>1$
(loc.cit. III.1.3.9).

\vskip8pt

 Put
$b_i:=h_i-1\in\bbZ[H]$ and $\bb^\al:=b_1^{\al_1}\cdot\cdot\cdot
b_d^{\al_d}$ for $\al\in\bbN_0^d$. Identifying group elements with
Dirac distributions induces a $K$-algebra embedding
$K[H]\hookrightarrow D(H,K),~~~~h\mapsto\delta_h.$ In the light of
the dual isomorphism $\psi_*: D(\bbZ_p^d,K)\car D(H,K)$ we see
that any $ \delta\in D(H,K)$ has a unique convergent expansion
$\delta=\sum_{\al\in\bbN_0^d}d_\al\bb^\al$ with $d_\al\in K$ such
that the set $\{|d_\al|r^{|\al|}\}_\al$ is bounded for all
$0<r<1$. Conversely, any such series is convergent in $D(H,K)$. By
construction the value $\delta(f)\in K$ of such a series on a
function $f\in C^{an}(H,K)$ equals $\delta(f)=\sum_\al d_\al
c_\al$ where $c_\al$ are the Mahler coefficients of $\psi^*(f)$.

\vskip8pt

 To take the original $p$-valuation
$\omega$ into account we define $\tau\al:=\sum_i \omega(h_i)\al_i$
for $\al\in\bbN_0^d$. The family of norms $||.||_r,~0<r<1$ on
$D(H,K)$ defined on a series $\delta$ as above via $||\delta
||_r:=\sup_\al |d_\al|r^{\tau\al}$ defines the Fr\'echet topology
on $D(H,K)$. Let $D_r(H,K)$ denote the norm completion of $D(H,K)$
with respect to $||.||_r$. Thus we obtain
\[D_r(H,K)=\{\sum_{\al\in\bbN_0^d}d_\al\bb^\al\in
K[[b_1,...,b_d]]|
\lim_{|\al|\rightarrow\infty}|d_\al|r^{\tau\al}=0.\]

 There is an obvious
norm-decreasing linear map $D_{r'}(H,K)\rightarrow D_r(H,K)$
whenever $r\leq r'$.

\vskip8pt

The norms $||.||_r$ belonging to the subfamily $1/p\leq r<1$ are
submultiplicative (loc.cit., Prop. 4.2) and do not depend on the
choice of ordered basis (loc.cit., before Thm. 4.11). In
particular, each $D_r(H,K)$ is a $K$-Banach algebra in this case.
If we equip the projective limit $\varprojlim_r D_r(H,K)$ with the
projective limit topology the natural map \[D(H,K)\car
\varprojlim_r D_r(H,K)\] is an isomorphism of topological
$K$-algebras.

\vskip8pt

Finally, it is almost obvious that the comultiplication $\Delta$
completes to continuous 'comultiplications' $\Delta_r:
D_r(H,K)\lra D_r(H,K)\cotimes_K D_r(H,K)$ for any $r$ in the above
range.

\vskip8pt

We make two final remarks in case $H$ is a uniform pro-$p$ group
and $\omega$ is its canonical $p$-valuation (2.2.2). In this case
each group $P_m(H), m\geq 0$ is a uniform pro-$p$ group.
\begin{itemize}
    \item[(i)]The resulting $||.||_{1/p}$-norm
topology on $D(H,K)$ coincides with the $p$-adic topology. In
fact, there is a canonical isomorphism between $D_{1/p}(H,\bbQ_p)$
and the $p$-adic completion (with $p$ inverted) of the universal
enveloping algebra of the $\bbZ_p$-Lie algebra $\frac{1}{p}\cL(H)$
(\cite[Thm. 10.4/Remark 10.5 (c)]{AW}).
    \item[(ii)] Let
$$r_m:=\sqrt[p^m]{1/p}$$ for $m\geq 0$. In particular, $r_0=1/p$.
Since $P_{m+1}(H)$ is uniform pro-$p$ we may consider the
corresponding $||.||_{r_0}$- norm on its distribution algebra
$D(P_{m+1}(H))$. In this situtation the ring extension
$D(P_{m+1}(H))\subset D(H)$ completes in the $||.||_{r_m}$-norm
topology on $D(H)$ to a ring extension
\[D_{r_0}(P_{m+1}(H))\subset D_{r_m}(H)\] and $D_{r_m}(H)$ is a finite
and free (left or right) module over $D_{r_0}(P_{m+1}(H)))$ on a
basis any system of coset representatives for the finite group
$H/P_{m+1}(H)$ (\cite[Lem. 5.11]{SchmidtDIM}).
\end{itemize}

\subsection{Coadmissible modules}\label{subsec-coadmissible}
We keep all notations from the preceding section but suppose that
the $p$-valuation $\omega$ on $H$ satisfies additionally \[
\begin{array}{ll} {\rm (HYP)}
 & (H,\omega) {\rm ~is}~p-{\rm saturated~ and~ the~ ordered~ basis}~ h_1,...,h_d {\rm ~of~} H  \\
   & {\rm satisfies}~\omega(h_i)+\omega(h_j)>p/(p-1) {\rm ~for~ any~} 1\leq i\neq j\leq d. \\
\end{array}\]

Remark: This implies that $H$ is a uniform pro-$p$ group.
Conversely, the canonical $p$-valuation on a uniform pro-$p$ group
($p$ arbitrary) satisfies (HYP). For both statements we refer to
\cite[Rem. before Lem. 4.4]{ST5} and \cite[Prop. 2.1]{SchmidtAUS}.

\vskip8pt

Suppose in the following $r\in (p^{-1},1)$ and $r\in p^{\bbQ}$. In
this case the norm $||.||_r$ on $D_r(H,K)$ is multiplicative and
$D_r(H,K)$ is a (left and right) noetherian integral domain
(\cite[Thm. 4.5]{ST5}). For two numbers $r\leq r'$ in the given
range the ring homomorphism
\[D_{r'}(H,K)\rightarrow D_r(H,K)\]
makes the target a flat (left or right) module over the source
(loc.cit., Thm. 4.9). The above isomorphism $D(H,K)\car
\varprojlim_r D_r(H,K)$ realizes therefore a {\it Fr\'echet-Stein
structure} on $D(H,K)$ in the sense of loc.cit. \S3. The latter
allows to define a well-behaved abelian full subcategory $\cC_{H}$
of the (left) $D(H,K)$-modules, the so-called {\it coadmissible
modules}. By definition, an abstract (left) $D(H,K)$-module $M$ is
coadmissible if for all $r$ in the given range
\begin{itemize}
    \item[(i)] $M_r:=D_{r}(H,K)\otimes_{D(H,K)}M$ is finitely
    generated over $D_{r}(H,K)$,
    \item[(ii)] the natural map $M\car\varprojlim_r M_r$ is an
    isomorphism.
\end{itemize}

The projective system $\{M_r\}_r$ is sometimes called the {\it
coherent sheaf} associated to $M$. To give an example, any
finitely presented $D(H,K)$-module is coadmissible.

In general, any compact locally $L$-analytic group has the structure of a
Frechet-Stein algebra (\cite[Thm. 5.1]{ST5}). In particular, we
may define the notion of a coadmissible module over $D(H, K)$ for any
compact $L$-analytic group in a similar manner. For a general locally
$L$-analytic group $G$, a $D(G, K)$-module $M$ is coadmissible if it
is coadmissible as a $D(H,K)$-module for every compact open subgroup
$H \subset G$. It follows from loc. cit. that it is sufficient to check this for
a single compact open subgroup.

\subsection{Locally analytic representations} We recall some facts
of locally analytic representations. A vector space $V$ which
equals a locally convex inductive limit $V=\varinjlim_{n\in\bbN}
V_n$ over a countable system of Banach spaces $V_n$ where the
transition maps $V_n\ra V_{n+1}$ are injective compact linear maps
is called a vector space {\it of compact type}. We recall
(\cite[Thm. 1.1]{ST4}) that such a space is Hausdorff, complete,
bornological and reflexive. Its strong dual $V'_b$ is a nuclear
Fr\'echet space satisfying $V'_b=\varprojlim_n (V_n)'_b$.

\vskip8pt

After this preliminary remark let $H$ be a locally $L$-analytic
group, $V$ a Hausdorff locally convex $K$-vector space and $\rho:
H\ra {\rm GL}(V)$ a homomorphism. Then $V$ (or the pair
$(V,\rho)$) is called a {\it locally analytic representation of
$H$} if the topological $K$-vector space $V$ is barrelled, each
$h\in H$ acts $K$-linearly and continuously on $V$, and the orbit
maps $\rho_v: H\ra V, h\mapsto \rho(h)(v)$ are locally analytic
maps for all $v\in V$ , cf. \cite[sec. 3]{ST4}. If $V$ is of
compact type, then the contragredient $G$-action on its strong
dual $V'_b$ extends to a separately continuous left
$D(H,K)$-module on a nuclear Fr\'echet space.

In this way the functor $V\mapsto V'_b$ induces an equivalence of
categories between locally analytic $H$-representations on
$K$-vector spaces of compact type (with continuous linear $H$-maps
as morphisms) and separately continuous $D(H,K)$-modules on
nuclear Fr\'echet spaces (with continuous $D(H,K)$-module maps as
morphisms).

A locally analytic $H$-representation $V$ is said to be admissible if
 $V'_{b}$ is a coadmissible $D(H,K)$-module. The above functor restricts
to an equivalence between the corresponding categories of admissible locally
analytic representations and coadmissible $D(H,K)$-modules.

\vskip8pt

\section{Completed skew group rings}\label{sec-skew}

In this section we will describe a general method of completing
certain skew group rings.

\subsection{Preliminaries}

Let $H$ be a compact locally $L$-analytic group and let $A$ be a
locally convex $L$-algebra equipped with a locally analytic
$H$-representation $\rho: H\rightarrow {\rm GL}(A).$ The
$H$-action on $A$ extends to a natural structure of a
$D(H,L)$-module (\cite[Thm. 2.2]{ST4}). On the other hand,
$D(H,L)$ is a topological module over itself via left
multiplication. The completion $A\cotimes_L D(H,L)$ is thus a
topological $D(H,L)\cotimes_L D(H,L)$-module. We view it as a
topological $D(H,L)$-module by restricting scalars via the
comultiplication $\Delta$. This allows to define the $L$-bilinear
map
\[(A\otimes_L
D(H,L))\times(A\cotimes_L D(H,L))\lra A\cotimes_L D(H,L)\] given
by $(\sum_i f_i\otimes \delta_i,b)\mapsto\sum_i f_i\cdot
\delta_i(b).$ We consider the product topology on the source. In
view of the continuity of all operations involved together with
\cite[Lem. 17.1]{NFA} this map is continuous. Since the target is
complete it extends in a bilinear and continuous manner to the
completion of the source. In other words, $A\cotimes_L D(H,L)$
becomes a topological $L$-algebra. Of course, $A\cotimes_L D(H,K)$
is then a topological $K$-algebra. To emphasize its skew
multiplication we denote it in the following by
$$A\smprod_L D(H,K)$$
or even by $A\smprod D(H,K)$. This should not cause confusion.
However, one has to keep in mind that there is a {\it completed}
tensor product involved.

\subsection{Skew group rings, skew enveloping
algebras and their completions}

\begin{para}

Using the action $\rho$ we may form the abstract skew group ring
$A\smprod H$ (\cite[1.5.4]{MCR}). We remind the reader that it
equals the free left $A$-module with elements of $H$ as a basis
and with multiplication defined by $(ag)\cdot
(bh):=a(\rho(g)(b))gh$ for any $a,b\in A$ and $g,h\in H$. Each
element of $A\smprod H$ has a unique expression as $\sum_{h\in H}
a_h h$ with $a_h=0$ for all but finitely many $h\in H$. Evidently,
$A\smprod H$ contains $H$ as a subgroup of its group of units and
$A$ as a subring. Furthermore, the inclusion $L[H]\subseteq
D(H,L)$ gives rise to an $A$-linear map

\begin{numequation}\label{equ-skew} A\smprod H=A\otimes_L L[H]\lra A\cotimes_L
D(H,L).\end{numequation}

On the other hand, let $\frh:=Lie(H)$. Differentiating the locally
analytic action $\rho$ gives a homomorphism of $L$-Lie algebras
$\al: \frh\lra {\rm Der}_L(A)$ into the $L$-derivations of the
algebra $A$ making the diagram

\[\xymatrix{
U(\frh) \ar[d]^\subseteq \ar[r]^>>>>>>{\al} &  \End_L(A) \ar[d]^{Id} \\
D(H,K) \ar[r]^\rho & \End_L(A) }
\]

commutative (\cite[3.1]{ST4}). We may therefore form the {\it skew enveloping
algebra}  $A\smprod U(\frh)$ (\cite[1.7.10]{MCR}). We recall that
this is an $L$-algebra whose underlying $L$-vector space equals
the tensor product $A\otimes_L U(\frh)$. The multiplication is
defined by

$$(f_1 \otimes \frx_1) \cdot (f_2 \otimes \frx_2) = (f_1 \alpha(\frx_1)(f_2))\otimes \frx_2 \; + \; (f_1f_2) \otimes (\frx_1 \frx_2) \;,$$
for $f_i \otimes \frx_i \in A\otimes_L \frh$. Moreover, the
inclusion $U(\frh)\subseteq D(H,L)$ induces an $A$-linear map

\begin{numequation}\label{equ-smprod}A\smprod U(\frh)\longrightarrow A\cotimes_L
D(H,L).\end{numequation}

\end{para}

\begin{prop}\label{prop-skewextensions}
The $A$-linear maps (\ref{equ-skew}) and (\ref{equ-smprod}) are
$L$-algebra homomorphisms. The first of these maps has dense
image.
\end{prop}
\begin{proof}
The first statement follows from the identities
\begin{itemize}
    \item[(i)] $(1\cotimes\delta_g)\cdot (f\cotimes 1)=(\rho(g)(f))\cotimes \delta_g {\rm ~~~for~any~}
    g\in H, f\in A$,\vskip8pt
    \item[(ii)] $(1\cotimes\frx)\cdot (f\cotimes 1)=(\al(\frx)(f))\cotimes 1+f\cotimes\frx {\rm ~~~for~any~}
    \frx\in\frh, f\in A$
\end{itemize}
in $A\cotimes_L D(H,L)$. In turn these identities follow from
$\Delta(\delta_g)=\delta_g\cotimes\delta_g$ and
$\Delta(\frx)=\frx\cotimes 1+1\cotimes \frx$. The final statement
follows from \cite[Lem. 3.1]{ST4}.
\end{proof}

\begin{para}We assume from now on that $L=\bbQ_p$ and that the compact
locally $\bbQ_p$-analytic group $H$ is endowed with a
$p$-valuation $\omega$. We put \[r_0:=p^{-1}\]
and consider the norm
completion $D_r(H,K)$ for some arbitrary but fixed $r\in [r_0,1)$.
Let us assume for a moment that the natural map
$D(H,L)\rightarrow D_r(H,L)$ satisfies the following
hypothesis:
\[\begin{array}{ll} {\rm (\star)}
 & {\rm The~topological~} D(H,L)-{\rm module~ structure~of~}A \\
   & {\rm extends~to~a~topological~}D_r(H,L)-{\rm module~structure.} \\
\end{array}\]
If we replace in the above discussion the comultiplication
$\Delta$ by its completion $\Delta_r$ we obtain in an entirely
analogous manner a completion $A\cotimes_L D_r(H,K)$ of the
skew group ring $A\smprod H$, base changed to $K$. It satisfies
{\it mutatis mutandis} the statement of the preceding proposition.
As before we will often abbreviate it by $A\smprod D_r(H,K)$.
\end{para}

\section{Sheaves on the Bruhat-Tits building and smooth representations}

\subsection{Filtrations of stabilizer subgroups}

\begin{para}
Let $\bT$ be a maximal $L$-split torus in $\bG$. Let $X^*(\bT)$
resp. $X_*(\bT)$ be the group of algebraic characters resp.
cocharacters of $\bT$. Let $\Phi=\Phi(\bG,\bT)\subset X^*(\bT)$
denote the root system determined by the adjoint action of $\bT$
on the Lie algebra of $\bG$. Let $W$ denote the corresponding Weyl
group. For each $\al\in\Phi$ we have the unipotent root subgroup
$\bU_\al\subseteq \bG$. Since $\bG$ is split the choice of a {\it
Chevalley basis} determines a system of $L$-group isomorphisms
$$x_\al: \bbG_a\car \bU_\al$$ for each $\al\in\Phi$ (an {\it
\'epinglage}) satisfying Chevalley's commutation relations
(\cite[p. 27]{ChevalleyTohoku}). Let $X_*(\bC)$ denote the group
of $L$-algebraic cocharacters of the connected center $\bC$ of
$\bG$.

We denote by $G,T,U_\al$ the groups of $L$-rational points of
$\bG,\bT,\bU_\al (\al\in\Phi)$ respectively. Recall the normalized
$p$-adic valuation $v_L$ on $L$, i.e. $v_L(\varpi)=1$. For
$\al\in\Phi$ we denote by $(U_{\al,r})_{r\in\bbR}$ the filtration
of $U_\al$ arising from the valuation $v_L$ on $L$ via the
isomorphism $x_\al$. It is an exhaustive and separated discrete
filtration by subgroups. Put $U_{\al,\infty}:=\{1\}$.

\end{para}

\begin{para}

Let $\sB=\sB(G)$ be the semisimple Bruhat-Tits building of $G$.
The torus $\bT$ determines an apartment $A$ in $\sB$. Recall that
a point $z$ in the Coxeter complex $A$ is called {\it special} if
for any direction of wall there is a wall of $A$ actually passing
through $z$ (\cite[1.3.7]{BruhatTits72}). As in \cite{Cartier},
3.5 we choose once and for all a special vertex $x_0$ in $A$ and a
chamber $\sC\subset A$ containing it. We use the point $x_0$ to
identify the affine space $A$ with the real vector space
\[A=( X_*(\bT)/X_*(\bC))\otimes_{\bbZ} \bbR.\]
Each root $\al\in\Phi$ induces therefore a linear form $\al:
A\ra\bbR$ in an obvious way. For any nonempty subset
$\Omega\subseteq A$ we let $f_\Omega:
\Phi\rightarrow\bbR\cup\{\infty\},~~~ \al\mapsto
-\inf_{x\in\Omega}\al(x).$ It is a {\it concave} function in the
sense of \cite[6.4.1-5]{BruhatTits72}. We emphasize that the
concept of a concave function is developed in loc.cit. more
generally for functions taking values in the set
\[\tilde{\bbR}:=\bbR\cup\{r+:r\in\bbR\}\cup\{\infty\}.\]
The latter has a natural structure as totally ordered commutative
monoid extending the total order and the addition on $\bbR$. For
any $\al\in\Phi$ and $r\in\bbR$ we define
\[U_{\al,r+}:=\cup_{s\in\bbR,s>r} U_{\al,s}.\]
For any concave function $f:\Phi\rightarrow\tilde{\bbR}$ we then
have the group
\begin{numequation}\label{equ-groupconcave} U_f:={\rm
subgroup~of~}G{\rm ~generated~by~all~}U_{\al,f(\al)}~{\rm
for~}\al\in\Phi.\end{numequation}

\end{para}
\begin{para}
For each nonempty subset $\Omega\subseteq \sB$ we let
\[P_\Omega:=\{g\in G: gz=z {\rm ~for~any~}z\in\Omega\}\]
be its pointwise stabilizer in $G$. For any facet $F\subseteq\sB$
we will recall from \cite[I.2]{SchSt97} a certain decreasing
filtration of $P_F$ by open normal pro-$p$ subgroups which will be
most important for all that follows in this article. To do this we
first consider a facet $F$ in the apartment $A$. Its concave
function $f_F$ is given by
\[f_F: \Phi\longrightarrow\bbR, \alpha\mapsto -\inf_{x\in
F}\alpha(x).\] For $\al\in\Phi$ we put $f^*_F(\al):=f_F(\al)+$ if
$\al|_F$ is constant and $f^*_F(\al):=f_F(\al)$ otherwise. This
yields a concave function $f_F^*: \Phi\rightarrow\tilde{\bbR}.$
With $f^*_F$ also the functions $f^*_F+e$, for any integer $e\geq
0$, are concave. Hence there is the descending sequence of
subgroups
\[ U_{f^*_F}\supseteq U_{f^*_F+1}\supseteq
U_{f^*_F+2}\supseteq...\]

\end{para}
\begin{para}

On the other hand we let $\frT:=Spec(o_L[X^*(\bT)])$ and
\[T^{(e)}:=\ker (\frT (o_L)\lra \frT (o_L/\varpi_L^{e+1}o_L))\]
for any $e\geq 0$ (cf. \cite[pf. of Prop. I.2.6]{SchSt97})
and finally define
\[\UFe:=U_{f^*_F+e}\cdot T^{(e)}\]
for each $e\geq 0$ (loc.cit. p.21). This definition is extended to
{\it any} facet $F$ in $\sB$ by putting $U_{F}^{(e)}:=g
U_{F'}^{(e)} g^{-1}$ if $F=gF'$ with $g\in G$ and $F'$ a facet in
$A$. We thus obtain a filtration
\[ P_F\supseteq U_F^{(0)}\supseteq U_F^{(1)}\supseteq
...\] of the pointwise stabilizer $P_F$ by normal subgroups. As in
loc.cit. we definte, for any point $z\in\sB$,
\[ \Uze:=\UFe\]
where $F$ is the unique facet of $\sB$ that contains $z$. The
group $\Uze$ fixes the point $z$. By construction we have
\begin{numequation}\label{equ-conjgroups}
U_{gz}^{(e)}=g U_{z}^{(e)} g^{-1}
\end{numequation}
for any $z\in\sB$ and any $g\in G$.

\vskip8pt

Remark: We emphasize that the definition of the groups
$\{\UFe\}_{F\subset\sB, e\geq 0}$ depends on the choice of the
special vertex $x_0$ as an origin for $A$. We also remark that the
very same groups appear in the work of Moy-Prasad on unrefined
minimal types (\cite{MoyPrasad}).

\end{para}

We will make use of the following basic properties of the groups
$\UFe$. To formulate them let $$\Phi=\Phi^+\cup\Phi^-$$ be any
fixed decomposition of $\Phi$ into positive and negative roots.

\begin{prop}\label{prop-rootspace}

\begin{itemize}
    \item[(i)] Let $F\subset A$ be a facet. For any $e\geq 0$ the product map induces a
bijection
\[ (\prod_{\al\in\Phi^-}U_{f_F^*+e}\cap U_\al)\times T^{(e)}\times(\prod_{\al\in\Phi^+}U_{f_F^*+e}\cap
U_\al)\car\UFe\] whatever ordering of the factors of the left hand
side we choose. Moreover, we have
\[U_{f_F^*+e}\cap U_\al=U_{\al,f^*_F(\al)+e}\] for any
$\al\in\Phi$.

\item[(ii)] For any facet $F\subset\sB$ the $\UFe$ for $e\geq 0$
form a fundamental system of compact open neighbourhoods of $1$ in
$G$,

    \item[(iii)] $U_{F'}^{(e)}\subseteq U_F^{(e)}$ for any two facets $F,F'$
in $\sB$ such that $F'\subseteq\overline{F}$.

\end{itemize}
\end{prop}

\begin{proof}
Cf. \cite[Prop. I.2.7, Cor. I.2.9, Prop. I.2.11]{SchSt97}.
\end{proof}

\begin{para}\label{ConcreteDescription}
As an example and in view of later applications we give a more
concrete description of the groups $\{U_{x_0}^{(e)}\}_{e\geq 0}$.
The stabilizer $P_{\{x_0\}}$ in $G$ of the vertex $x_0$ is a
special, good, maximal compact open subgroup of $G$
(\cite[3.5]{Cartier}). We let $\frG$ be the connected reductive
$o_L$-group scheme with generic fibre $\bG$ associated with the
special vertex $x_0$ (\cite[3.4]{TitsCorvallis},
\cite[4.6.22]{BruhatTits84}). Its group of $o_L$-valued points
$\frG(o_L)$ can be identified with $P_{\{x_0\}}$. For $e\geq 0$ we
therefore have in $P_{\{x_0\}}$ the normal subgroup
$\frG(\varpi^{e}):=\ker\; ( \frG(o_L)\rightarrow
\frG(o_L/\varpi^{e}o_L) ).$

Now the concave function $f_{\{x_0\}}$ vanishes identically whence
$f_{\{x_0\}}^*$ has constant value $0+$. Thus,
\[ U_{\alpha,f^*_{\{x_0\}}(\alpha)+e}=\cup_{s>0} \{a\in L: v_L(a)\geq
e+s\}=\varpi^{e+1}o_L\] for any $e\geq 0$. By (\ref{prop-rootspace} (i)) and the definition of $T^{(e)}$ we
therefore have a canonical isomorphism $U_{x_0}^{(e)}\car
\frG(\varpi^{e+1})$ for any $e\geq 0$.

\end{para}

\subsection{The Schneider-Stuhler construction} We now review the
construction of a certain `localization' functor constructed by P.
Schneider and U. Stuhler in \cite[IV.1]{SchSt97}. In fact,
there will be a functor for each 'level' $e\geq 0$. Following
loc.cit., we will suppress this dependence in our notation. In
\cite{SchSt97} only complex representations are considered.
However, all results remain true over our characteristic zero
field $K$ (\cite{VignerasSheaves}).

\begin{para}

Recall that a {\it smooth} representation $V$ of $G$ is a
$K$-vector space $V$ together with a linear action of $G$ such
that the stabilizer of each vector is open in $G$. A morphism
between two such representations is simply a $K$-linear
$G$-equivariant map.

Now let us fix an integer $e\geq 0$ and let $V$ be a smooth
representation. For any subgroup $U\subseteq G$ we have the
$K$-vector space
\[V_U:={\rm maximal~quotient~of~}V {\rm~on~which~the~}U-{\rm action~is~trivial}\]
of $U$-coinvariants of $V$. For any open subset
$\Omega\subseteq\sB$ we let
\[
\tiV(\Omega):=K{\rm -vector~ space~ of~ all~ maps~}
s:\Omega\rightarrow\dot{\bigcup}_{z\in\Omega} V_{\Uze} {\rm ~such~
that}
\]
\begin{itemize}
    \item[-] $s(z)\in V_{\Uze}$ for all $z\in\Omega$,
    \item[-] there is an open covering $\Omega=\cup_{i\in I} \Omega_i$ and
    vectors $v_i\in V$ with $$s(z)={\rm class~of~}v_i~~\in V_{\Uze}$$ for any $z\in\Omega_i$ and $i\in I$.
\end{itemize}

\end{para}
We summarize some properties of this construction in the following
proposition. Recall that a sheaf on a polysimplicial space is
called {\it constructible} if its restriction to a given geometric
polysimplex is a constant sheaf (\cite[8.1]{KashiwaraSchapira}).
\begin{prop}\label{prop-smoothloc}
\begin{itemize}
\item[(i)] The correspondance $\Omega\mapsto\tiV(\Omega)$ is a
sheaf of $K$-vector spaces,
    \item[(ii)] for any $z\in\sB$ the stalk of the sheaf $\tiV$ at $z$ equals $(\tiV)_z=V_{\Uze}$,
    \item[(iiI)] $\tiV$ is a constructible sheaf whose restriction to any facet $F$ of
    $\sB$ is constant with value $V_{\UFe}$,
    \item[(iv)] the correspondance $V\mapsto\tiV$ is an exact
    functor from smooth $G$-representations to sheaves of
    $K$-vector spaces on $\sB$.
\end{itemize}
\end{prop}
\begin{proof}
(i) follows from the local nature of the preceding definition.
(ii) and (iii) is \cite[Lem. IV.1.1]{SchSt97}. (iv) follows from
(ii) because of $char(K)=0$.
\end{proof}

We recall that the smooth representation $V$ is called {\it
admissible} if the $H$-invariants $V^H$ form a finite dimensional
$K$-vector space for any compact open subgroup $H$ of $G$. In this
situation the natural projection map $V\rightarrow V_H$ induces a
linear isomorphism $V^H\car V_H$. For an admissible representation
$V$ we may therefore deduce from (\ref{prop-smoothloc} (ii))
that the stalks of $\tiV$ are finite dimensional $K$-vector
spaces.

We emphasize again that the functor $V\mapsto \tiV$ depends on the
level $e\geq 0$.

\subsection{$p$-valuations on certain stabilizer subgroups}

We keep the notations from the preceding paragraph and define
certain $p$-valuations on the groups $\UFe$. However, for the rest
of this section we {\bf assume} $L=\bbQ_p$.

\begin{lemma} Let $F$ be a facet in $\sB$ and $e,e'\geq 0$. The commutator
group $(\UFe,U_{F}^{(e')})$ satisfies
\[(\UFe,U_{F}^{(e')})\subseteq U_F^{(e+e')}.\]
\end{lemma}
\begin{proof}
Choosing a facet $F'$ in $A$ and an element $g\in G$ such that
$F'=gF$ we may assume that $F$ lies in $A$. Define a function
$h_F:\Phi\cup\{0\}\rightarrow\tilde{\bbR}$ via $h_F|_\Phi:=f^*_F$
and $h_F(0):=0+$. Then $g:=h_F+e$ and $f:=h_F+e'$ are concave
functions in the sense of \cite[Def. 6.4.3]{BruhatTits72}.
Consider the function
$h:\Phi\cup\{0\}\rightarrow\tilde{\bbR}\cup\{-\infty\}$ defined as
\[h(a):=\inf \{\sum_i f(a_i)+\sum_jg(b_j)\}\]
where the infimum is taken over the set of pairs of finite
nonempty sets $(a_i)$ and $(b_j)$ of elements in $\Phi\cup\{0\}$
such that $a=\sum_i a_i+\sum_j b_j$. Using that the functions $f$
and $g$ are concave one finds $$h_F(a)+e+e'\leq h(a)$$ for any
$a\in\Phi\cup\{0\}.$ By loc.cit., Prop. 6.4.44, the function $h$ is
therefore concave and has the property
\[(U_f,U_g)\subseteq U_h\subseteq U_{h_F+e+e'}.\]
Here, the groups involved are defined completely analogous to
(\ref{equ-groupconcave}) (cf. loc.cit., Def. 6.4.42). It remains
to observe that $U_{h_F+a}=U_{F}^{(a)}$ for any integer $a\geq 0$
(\cite[p. 21]{SchSt97}).
\end{proof}
Let $l$ be the rank of the torus $\bT$. By construction of $\frT$
any trivialization $\bT\simeq(\bbG_m)^l$ yields an identification
$\frT\simeq (\bbG_{m/o_L})^l$ which makes the structure of the
topological groups $T^{(e)}, e\geq 0$ explicit. Moreover, we {\bf
assume} in the following $e\geq 2$. For each $g\in\UFe\setminus
\{1\}$ let
$$ \omega^{(e)}_F(g):=\sup \{n\geq 0: g\in U_F^{(n)}\}.$$
The following corollary is essentially due to H. Frommer
(\cite[1.3, proof of Prop. 6]{Frommer}). For sake of completeness
we include a proof.
\begin{cor}
The function
\[\omega^{(e)}_F: \UFe\setminus\{1\}\longrightarrow (1/(p-1),\infty)\subset\bbR\]
is a $p$-valuation on $\UFe$.
\end{cor}
\begin{proof}
The first axiom (i) is obvious and (ii) follows from the lemma.
Let $g\in\UFe$ with $n:=\omega^{(e)}_F(g)$. We claim
$\omega_F^{(e)}(g^p)=n+1$. The root space decomposition
(\ref{prop-rootspace})
\[ m: (\prod_{\al\in\Phi^-}U_{\al,f^*_F(\al)+n})\times
T^{(n)}\times(\prod_{\al\in\Phi^+}U_{\al,f^*_F(\al)+n})\car
U_F^{(n)}\] is in an obvious sense compatible with variation of
the level $n$. If $g\in T^{(n)}$ the claim is immediate. The same
is true if $g\in U_{\al,f^*_F(\al)+n}$ for some $\al\in\Phi$:
indeed the filtration of $U_\al$ is induced by the $p$-adic
valuation on $\bbQ_p$ via $x_\al: \bbQ_p\simeq U_\al$. In general
let $m(h_1,...,h_d)=g$. By what we have just said there is $1\leq
i\leq d$ such that $\omega^{(e)}(h_i^p)=n+1$ and
$\omega^{(e)}(h^p_j)\geq n+1$ for all $j\neq i$. Furthermore,
$h_1^p\cdot\cdot\cdot h_d^pg'=g^p$ where $g'\in
(U_F^{(n)},U_F^{(n)})\subseteq U_F^{(2n)}$. Since $n\geq 2$ we
have $2n\geq n+2$ and hence $g^p\in U_F^{(n+1)}$. If $g^p\in
U_F^{(n+2)}$ then $h_1^p\cdot\cdot\cdot h_d^p=g^pg'^{-1}\in
U_F^{(n+2)}$ which contradicts the existence of $h_i$. Hence
$\omega^{(e)}(g^p)=n+1$ which verifies axiom (iii).
\end{proof}

\begin{para}
For a given root $\al\in\Phi$ let $u_\al$ be a topological
generator for the group $U_{\al,f^*_F(\al)+e}$. Let $t_1,...,t_l$
be topological generators for the group $T^{(e)}$. In the light of
the decomposition of (\ref{prop-rootspace} (i)) it is rather
obvious that the set
\[\{u_\al\}_{\al\in\Phi^-}\cup
\{t_i\}_{i=1,...,l}\cup\{u_\al\}_{\al\in\Phi^+}\] arranged in the
order suggested by loc.cit. is an ordered basis for the $p$-valued
group $(\UFe,\omega_F^{(e)})$. Of course, $\omega_F^{(e)}(h)=e$
for any element $h$ of this ordered basis.

\vskip8pt

For technical reasons we will work in the following with the
slightly simpler $p$-valuations

\[\mathring{\omega}_F^{(e)}:=\omega_F^{(e)}-(e-1)\]

satisfying $\mathring{\omega}_F^{(e)}(h)=1$ for any element $h$ of
the above ordered basis. If $z\in\sB$ lies in the facet
$F\subset\sB$ we write $\mathring{\omega}_z^{(e)}$ for
$\mathring{\omega}_F^{(e)}$.

\vskip8pt

Remark: The tangent map at $1\in G$ corresponding to the $p$-power
map equals multiplication by $p$ and thus, is an isomorphism. It
follows from (\ref{prop-rootspace} (ii)) that there is
$e(F)\geq 2$ such that the $p$-power map is invertible on $\UFe$
for all $e\geq e(F)$. In this case, $\UFe$ is obviously
$p$-saturated (cf. 2.2.1) and therefore a uniform pro-$p$ group
(apply remark before Lemma 4.4 in \cite{ST5} to
$\mathring{\omega}_F^{(e)}$ and use $p\neq 2$). Since any facet in
$\sB$ is conjugated to a facet in $\sC$ we deduce from
(\ref{equ-conjgroups}) that there is a number $e_{uni}\geq 2$ such
that all the groups $\UFe$ for $F\subset\sB$ are uniform pro-$p$
groups whenever $e\geq e_{uni}$. In this situation \cite[Prop. A1]{HKN},
asserts that the subgroups
$$\UFe\supset U_F^{(e+1)}\supset U_F^{(e+2)}...$$
form the lower $p$-series of the group $\UFe$.

\end{para}

We may apply the discussion of (\ref{subsec-dist}) to
$(\UFe,\comFe)$ and the above ordered basis to obtain a family of
norms $||.||_r, r\in [1/p,1)$ on $D(\UFe,K)$ with completions
$D_r(\UFe,K)$ being $K$-Banach algebras. For facets $F,F'$ in
$\sB$ such that $F'\subseteq\overline{F}$ we shall need a certain
`gluing' lemma for these algebras.
\begin{lemma}\label{lem-glue}Let $F,F'$ be two facets in $\sB$ such that $F'\subseteq\overline{F}$. The
inclusion $U_{F'}^{(e)}\subseteq\UFe$ extends to a
norm-decreasing algebra homomorphism
$$\sigma_r^{F'F}: D_r(U_{F'}^{(e)},K)\longrightarrow D_r(\UFe,K).$$
Moreover,
\begin{itemize}
    \item[(i)] $\sigma_r^{FF}=id$,\vskip8pt
    \item[(ii)] $\sigma_r^{F'F}\circ\sigma_r^{F''F'}=\sigma_r^{F''F}$ if $F''$ is a
third facet in $\sB$ with $F''\subseteq\overline{F'}$.
\end{itemize}
Finally, $\sigma_r^{F'F}$ restricted to $Lie(U_{F'}^{(e)})$ equals
the map $Lie(U_{F'}^{(e)})\simeq Lie(U_{F}^{(e)})\subset
D_r(\UFe,K)$ where the first arrow is the canonical Lie algebra
isomorphism from \cite[III \S3.8]{B-L}, .
\end{lemma}
\begin{proof}
By functoriality (\cite[1.1]{KohlhaaseI}) of $D(\cdot,K)$ we
obtain an algebra homomorphism
\[\sigma: D(U_{F'}^{(e)},K)\longrightarrow D(\UFe,K).\]
Let $h'_1,...,h'_d$ and $h_1,...,h_d$ be the ordered bases of
$\UFbe$ and $\UFe$ respectively. Let $b'_i=h'_i-1\in\bbZ[\UFbe]$
and ${\bb'}^{m}:={b'_1}^{m_1}\cdot\cdot\cdot {b'_d}^{m_d}$ for
$m\in\bbN_0^d$. Given an element
\[\lambda=\sum_{m\in\bbN_0^d}d_m{\bb'}^{m}\in
D(U_{F'}^{(e)},K)\] we have $||\lambda||_r= \sup_m
|d_m|~||b_i'||_r$. Because of
\[ ||\sigma (\lambda)||_r \leq
\sup_m \;|d_m| (||\sigma (b_1')||_r)^{m_1}\cdot\cdot\cdot
(||\sigma (b_d')||_r)^{m_d})
\]
it therefore suffices to prove $||\sigma (b_i')||_r\leq
||b_i'||_r$ for any $i$. If $h'_i$ belongs to the toral part of
the ordered basis of $\UFbe$ then $\sigma(b_i')=b_i'$ and we are
done. Let therefore $\al\in\Phi$ and consider the corresponding
elements $h'_\al$ and $h_\al$ in the ordered bases of $\UFbe$ and
$\UFe$ respectively. By the root space decomposition we have
\[U_{\al,f^*_{F'}(\al)+e}\subseteq
U_{\al,f^*_{F}(\al)+e}=(h_\al)^{\bbZ_p}.\] Let therefore
$a\in\bbZ_p$ such that $h'_\al=(h_\al)^a$. Since a change of
ordered basis does not affect the norms in question (cf. \ref{subsubsec-norms}) we may assume $a=p^s$ for some natural
number $s\geq 0$. Then
\[h'_\al-1=(h_\al+1-1)^{p^s}-1=\sum_{k=1,...,p^s}{p^s\choose
k}(h_\al-1)^k\] and therefore
\[ ||\sigma(h'_\al-1)||_r\leq\max_{k=1,...,p^s} |{p^s\choose
k}|~ ||(h_\al-1)||_r^k=\max_{k=1,...,p^s} |{p^s\choose k}|
r^{k}\leq r=||h'_\al-1||_r\] which shows the claim and the
existence of $\sigma_r^{FF'}$. The properties (i),(ii) follow from
functoriality of $D(\cdot,K)$ by passing to completions. Since
$U_{F'}^{(e)}\subseteq\UFe$ is an open immersion of Lie groups the
final statement is clear.
\end{proof}

\section{Sheaves on the flag variety and Lie algebra
representations}\label{sec-BB}

\subsection{Differential operators on the flag variety}
\begin{para}
 Let $X$ denote the variety of Borel subgroups of $\bG.$ It is a
 smooth and projective $L$-variety. Let $\cO_X$ be its structure sheaf. Let $\frg$ be the Lie algebra of
$\bG$. Differentiating the natural (left) action of $\bG$ on $X$
yields a homomorphism of Lie algebras
\[\al: \frg\longrightarrow \Gamma(X,\cT_X)\]
into the global sections of the tangent sheaf $\cT_X={\mathcal
Der}_{L}(\cO_{X})$ of $X$ (\cite[II \S4.4.4]{DemazureGabriel}). In
the following we identify an abelian group (algebra, module etc.)
with the corresponding constant sheaf on $X$. This should not
cause any confusion. Letting $$\frg^\ci:=\cO_X\otimes_L \frg$$ the
map $\al$ extends to a morphism of $\cO_X$-modules $\al^\ci:
\frg^\ci\longrightarrow \cT_X.$ Defining $[\frx,f]:=\al(\frx)(f)$
for $\frx\in\frg$ and a local section $f$ of $\cO_X$ makes
$\frg^\ci$ a sheaf of $L$-Lie algebras\footnote{Following
\cite{BB81} we call such a sheaf simply a Lie algebra over $X$ in
the sequel. This abuse of language should not cause confusion.}.
Then $\al^\ci$ is a morphism of $L$-Lie algebras. We let
$\frb^\ci:=\ker\al^\ci,$ a subalgebra of $\frg^\ci$, and
$\frn^\ci:=[\frb^\ci,\frb^\ci]$ its derived algebra. At a point
$x\in X$ with residue field $k(x)$ the reduced stalks of the
sheaves $\frb^\ci$ and $\frn^\ci$ equal the Borel subalgebra
$\frb_x$ of $k(x)\otimes_L \frg$ defined by $x$ and its nilpotent
radical $\frn_x\subset\frb_x$ respectively.

Let $\frh$ denote the abstract Cartan algebra of $\frg$
(\cite[\S2]{Milicic93}). We view the $\cO_X$-module
$\cO_X\otimes_L\frh$ as an abelian $L$-Lie algebra. By definition
of $\frh$ there is a canonical isomorphism of $\cO_X$-modules and
$L$-Lie algebras
\begin{numequation}\label{equ-cartan}
\frb^\ci/\frn^\ci\car\cO_X\otimes_L\frh.\end{numequation}

Let $U(\frg)$ be the enveloping algebra of $\frg$. The enveloping
algebra of the Lie-algebra $\frg^\ci$ has the underlying
$\cO_X$-module $\cO_X\otimes_L U(\frg)$. Its $L$-algebra of local
sections over an open affine $V\subseteq X$ is the skew enveloping
algebra $\cO_X(V)\smprod U(\frg)$ relative to $\al: \frg\ra {\rm
Der}_L (\cO_X(V))$ (in the sense of sec \ref{sec-skew}). To
emphasize this skew multiplication we follow \cite[3.1.3]{BMR08} and denote the enveloping algebra of $\frg^\ci$ by
\[\cO_X\smprod U(\frg).\]

\end{para}
\begin{para}

To bring in the torus $\bT$ we choose a Borel subgroup $\bB\subset
\bG$ defined over $L$ containing $\bT$. Let $\bN\subset\bB$ be the
unipotent radical of $\bB$ and let $\bN^{-}$ be the unipotent
radical of the Borel subgroup opposite to $\bB$. We denote by
$$q:\bG\longrightarrow\bG/\bB=X$$
the canonical projection.

\vskip8pt

Let $\frb$ be the Lie algebra of $\bB$ and $\frn\subset\frb$ its
nilpotent radical. If $\frt$ denotes the Lie algebra of $\bT$ the
map $\frt\subset\frb\rightarrow\frb/\frn\simeq\frh$ induces an
isomorphism $\frt\simeq\frh$ of $L$-Lie algebras. We will once and
for all identify these two Lie algebras via this isomorphism.
Consequently, (\ref{equ-cartan}) yields a morphism of
$\cO_X$-modules and $L$-Lie algebras
\[\frb^\ci\lra\frb^\ci/\frn^\ci\car\cO_X\otimes_L\frt.\]
Given a linear form $\lambda\in\frt^*$ it extends $\cO_X$-linearly
to the target of this morphism and may then be pulled-back to
$\frb^\ci$. This gives a $\cO_X$-linear morphism $\lambda^\ci:
\frb^\ci\longrightarrow\cO_X.$

\end{para}

\begin{para}Let $\rho:=\frac{1}{2}\sum_{\al\in\Phi^+}\al.$ Given $\chi\in\frt^*$
we put $\lambda:=\chi-\rho.$ Denote by $\cI_\chi$ the right ideal
sheaf of $\cO_X\smprod U(\frg)$ generated by $\ker\lambda^\ci$,
i.e. by the expressions
\[\xi - \lambda^\ci(\xi)\] with $\xi$ a local section of
$\frb^\ci\subset\frg^\ci\subset \cO_X\smprod U(\frg)$. It is a
two-sided ideal and we let
\[\cD_\chi:=(\cO_X\smprod U(\frg))/\cI_\chi\]
be the quotient sheaf. This is a sheaf of noncommutative
$L$-algebras on $X$ endowed with a natural algebra morphism
$U(\frg)\rightarrow\Gamma(X,\cD_\chi)$ induced by $\frx\mapsto
1\otimes\frx$ for $\frx\in U(\frg)$. On the other hand $\cD_\chi$
is an $\cO_X$-module through the (injective) $L$-algebra morphism
$\cO_X\rightarrow\cD_\chi$ induced by $f\mapsto f\otimes 1$. This
allows to define the full subcategory $\cM_{qc}(\cD_\chi)$ of the
(left) $\cD_\chi$-modules consisting of modules which are
quasi-coherent as $\cO_X$-modules. It is abelian.

\end{para}

\begin{para}
For future reference we briefly discuss a refinement of the above
construction of the sheaf $\cD_{\chi}$. The right ideal of
$\cO_X\smprod U(\frg)$ generated by $\frn^\ci$ is a two-sided
ideal and, following \cite[\S3]{Milicic93} we let
$$\cD_{\frt}:=(\cO_X\smprod
U(\frg))/\frn^\ci(\cO_X\smprod U(\frg))$$ be the quotient sheaf.
We have the open subscheme $U_1:=q(\bN^{-})$ of $X$. Choose a
representative $\dot{w}\in G$ for every $w\in W$ with $\dot{1}=1$.
The translates $U_w:=\dot{w}U_1$ for all $w\in W$ form a Zariski
covering of $X$. Let $\frn^{-}$ be the Lie algebra of $\bN^{-}$
and put $\frn^{-,w}:={\rm Ad}(\dot{w})(\frn^{-})$ for any $w\in
W$.

\vskip8pt

For any $w\in W$ there are obvious canonical maps from
$\cO_X(U_w), U(\frn^{-,w})$ and $U(\frt)$ to $\cO_X(U_w)\smprod
U(\frg)$ and therefore to $\cD_{\frt} (U_w)$. According to
\cite[Lem. C.1.3]{Milicic93Preprint} they induce a $K$-algebra
isomorphism
\begin{numequation}\label{equ-trivialization} (\cO_X(U_w)\smprod U(\frn^{-,w})
)\otimes_L
U(\frt)\car\cD_\frt(U_w).\end{numequation} Note here that
$\bN^{-}=\bbA_L^{|\Phi^-|}$ implies that the skew enveloping
algebra $\cO_X(U_w)\smprod U(\frn^{-,w})$ is equal to the usual
algebra of differential operators $\cD_X(U_w)$ on the translated
affine space $U_w=\dot{w}U_1$.

\vskip8pt The above discussion implies that the canonical
homomorphism $$U(\frt)\mapsto\cO_X\smprod U(\frg), \frx\mapsto
1\otimes\frx$$ induces a central embedding $U(\frt)\hookrightarrow
\cD_\frt.$ In particular, the sheaf $(\ker\lambda)\cD_\frt$ is a
two-sided ideal in $\cD_\frt$. According to the discussion before Thm. 3.2 in \cite{Milicic93}, p. 138, the canonical map
$\cD_\frt\rightarrow\cD_\chi$ coming from $\frn^{\ci}\subset
\ker\lambda^\ci$ induces
$$\cD_\frt\otimes_{U(\frt)}
L_\lambda=\cD_{\frt}/(\ker\lambda)\cD_\frt \car \cD_\chi,$$ an
isomorphism of sheaves of $K$-algebras.

\vskip8pt

Remark: According to the above we may view the formation of the
sheaf $\cD_\chi$ as a two-step process. In a first step on
constructs the sheaf $\cD_\frt$ whose sections over the Weyl
translates of the big cell $U_1$ are explicitly computable.
Secondly, one performs a central reduction
$\cD_\frt\otimes_{U(\frt)}L_\lambda$ via the chosen character
$\lambda=\chi-\rho$. This point of view will be useful in later
investigations.

\end{para}

\vskip8pt

\subsection{The Beilinson-Bernstein localization theorem}

\begin{para}\label{para-HC}

We recall some notions related to the classical {\it
Harish-Chandra isomorphism}. To begin with let $S(\frt)$ be the
symmetric algebra of $\frt$ and let $S(\frt)^W$ be the subalgebra
of Weyl invariants. 
Let $Z(\frg)$ be the center of the universal enveloping algebra
$U(\frg)$ of $\frg$. The classical Harish-Chandra map is an
algebra isomorphism $Z(\frg)\car S(\frt)^W$ relating central
characters and highest weights of irreducible highest weight
$\frg$-modules in a meaningful way (\cite[7.4]{Dixmier}). Given a
linear form $\chi\in\frt^*$ we let
$$\sigma(\chi): Z(\frg)\rightarrow L$$ denote the central character
associated with $\chi$ via the Harish-Chandra map. Recall that
$\chi\in\frt^*$ is called {\it dominant} if
$\chi(\check{\al})\notin\{-1,-2,...\}$ for any coroot
$\check{\al}$ with $\al\in\Phi^+$. It is called {\it regular} if
$w(\chi)\neq \chi$ for any $w\in W$ with $w\neq 1$.

\end{para}
Let $\theta:=\sigma(\chi)$ and put
$U(\frg)_\theta:=U(\frg)\otimes_{Z(\frg),\theta} L$ for the
corresponding central reduction.
\begin{thm}(Beilinson/Bernstein)\label{thm-BB}
\begin{itemize}
    \item[(i)] The algebra morphism $U(\frg)\rightarrow
    \Gamma(X,\cD_\chi)$ induces an isomorphism $U(\frg)_\theta\simeq
    \Gamma(X,\cD_\chi)$.
    \item[(ii)] If $\chi$ is dominant and regular the functor
    $M\mapsto \cD_\chi\otimes_{U(\frg)_\theta} M$ is an
    equivalence of categories between the (left)
    $U(\frg)_\theta$-modules and $\cM_{qc}(\cD_\chi)$.
     \item[(iii)] Let $M$ be a $U(\frg)_\theta$-module. The reduced stalk of the sheaf $\cD_\chi\otimes_{U(\frg)_\theta}
 M$ at a point $x\in X$ equals the $\lambda$-coinvariants of the $\frh$-module
 $(k(x)\otimes_L M)/\frn_x(k(x)\otimes_L M)$.

\end{itemize}
\end{thm}
\begin{proof}
This is the main theorem of \cite{BB81}.
\end{proof}

Remarks:\begin{itemize}
    \item[(i)] In \cite{BB81} the theorem is proved
under the assumption that the base field is algebraically closed.
However, all proofs of loc. cit. go through over an arbitrary
characteristic zero field in the case where the $\frg$ is split
over the base field. In the following, this is the only case we
shall require.
    \item[(ii)] If $\lambda:=\chi-\rho\in X^*(\bT)\subset\frt^*$ and if
$\cO(\lambda)$ denotes the associated invertible sheaf on $X$ then
$\cD_\chi$ can be identified with the sheaf of differential
endomorphisms of $\cO(\lambda)$ (\cite[p. 138]{Milicic93}). It is
therefore a {\it twisted sheaf of differential operators} on $X$
in the sense of \cite[\S1]{BB81}. In particular, if $\chi=\rho$
the map $\al^\ci$ induces an isomorphism $\cD_\rho\car\cD_X$ with
the usual sheaf of differential operators on $X$ (\cite[\S16.8]{EGA_IV_4}). In this case, $\cM_{qc}(\cD_\chi)$ equals therefore the
usual category of algebraic $D$-modules on $X$ in the sense of
\cite{BorelDMod}.

\end{itemize}

\section{Berkovich analytifications}\label{sec-Berkovich}
\subsection{Differential operators on the analytic flag variety}
\begin{para}
For the theory of Berkovich analytic spaces we refer to
\cite{BerkovichBook}, \cite{BerkovichEtale}. We keep the notations
introduced in the preceding section. In particular, $X$ denotes
the variety of Borel subgroups of $\bG$. Being a scheme of finite
type over $L$ we have an associated Berkovich analytic space
$\Xan$ over $L$ (\cite[Thm. 3.4.1]{BerkovichBook}). In the
preceding section we recalled a part of the algebraic
Beilinson-Bernstein localization theory over $X$. It admits the
following `analytification' over $\Xan$.

\vskip8pt

 By construction $\Xan$
comes equipped with a canonical morphism
\[\pi: \Xan\rightarrow X\] of locally ringed spaces. Let $\pi^*$
be the associated inverse image functor from $\cO_X$-modules to
$\cO_{\Xan}$-modules. Here $\cO_{\Xan}$ denotes the structure
sheaf of the locally ringed space $\Xan$. As with any morphism of
locally ringed spaces we have the sheaf
\[\cT_{\Xan}:={\mathcal Der}_{L}(\cO_{\Xan})\]
of $L$-derivations of $\cO_{\Xan}$ (\cite[16.5.4]{EGA_IV_4}). By
definition $\Gamma(\Xan,\cT_{\Xan})={\rm Der}_L(\cO_{\Xan})$.
Since $\Xan$ is smooth over $L$ the results of
\cite[3.3/3.5]{BerkovichEtale} imply that the stalk of this sheaf
at a point $x\in\Xan$ equals $\cT_{\Xan,x}={\rm
Der}_L(\cO_{\Xan,x}).$

\vskip8pt

Let $\Gan$ denote the analytic space associated to the variety
$\bG$ and let $\pi_{\bG}:\Gan\ra\bG$ be the canonical morphism.
The space $\bG$ is a group object in the category of $L$-analytic
spaces (a {\it $L$-analytic group} in the terminology of
\cite[5.1]{BerkovichBook}). The unit sections of $\bG$ and $\Gan$
correspond via $\pi_\bG$ which allows us to canonically identify
the Lie algebra of $\Gan$ with $\frg$ (loc.cit., Thm. 3.4.1 (ii)).
By functoriality the group $\Gan$ acts on $\Xan$. The following
result is proved as in the scheme case.
\begin{lemma}\label{lem-analytification}
The group action induces a Lie algebra homomorphism
$$\frg\ra\Gamma(\Xan,\cT_{\Xan}).$$
\end{lemma}
We define
$$\frg^{\ci,an}:=\cO_\Xan\otimes_L\frg=\rho^*(\frg^\ci).$$ The preceding lemma allows on the one hand, to
define a structure of $L$-Lie algebra on $\frg^{\ci,an}$. Its
enveloping algebra will be denoted by $\cO_{\Xan}\smprod U(\frg)$.
On the other hand, the map from the lemma extends to a
$\cO_\Xan$-linear morphism of $L$-Lie algebras
\begin{numequation}\label{equ-alpha-analytic}\al^{\ci,an}:
\frg^{\ci,an}\longrightarrow\cT_{\Xan}.\end{numequation}

As in the algebraic case we put $\frb^{\ci,an}:=\ker\al^{\ci,an}$
and $\frn^{\ci,an}:=[\frb^{\ci,an},\frb^{\ci,an}]$. Again, we
obtain a morphism
$\frb^{\ci,an}\rightarrow\cO_{\Xan}\otimes_L\frt$. Given
$\chi\in\frt^*$ and $\lambda:=\chi-\rho$ we denote by $\cI^{an}$
resp. $\cI^{an}_\chi$ the right ideal sheaf of $\cO_{\Xan}\smprod
U(\frg)$ generated by $\frn^{\ci,an}$ resp. $\ker\lambda^{\ci,an}$
where $\lambda^{\ci,an}$ equals the $\cO_\Xan$-linear form of
$\frb^{\ci,an}$ induced by $\lambda$. These are two-sided ideals.
We let
\[\cD^{an}_\frt:=(\cO_\Xan\smprod U(\frg))/\cI^{an}~~~~~~{\rm and}~~~~~~\cD^{an}_\chi:=(\cO_\Xan\smprod U(\frg))/\cI^{an}_\chi\]
be the quotient sheaves. We view $\cD^{an}_\chi$ as a sheaf of
twisted differential operators on $\Xan$.

\vskip8pt

All these constructions are compatible with their algebraic
counterparts via the functor $\pi^*$. For example, using the fact
that $\pi^*(\cT_X)=\cT_{\Xan}$ it follows from the above proof
that $\al^{\ci,an}=\pi^*(\al^\ci)$. Moreover, all that has been
said in sec. 5 on the relation between the
    sheaves $\cD_\frt$ and $\cD_\chi$ remains true for its
    analytifications. In particular, $\cD^{an}_\chi$ is a central
    reduction of $\cD^{an}_\frt$ via the character $\lambda:
    U(\frt)\rightarrow L$:
    \begin{numequation}\label{equ-centred}\cD^{an}_{\frt}/(\ker\lambda)\cD^{an}_{\frt}\car\cD^{an}_\chi.\end{numequation}

\end{para}

\subsection{The Berkovich embedding and analytic stalks}
Recall our chosen Borel subgroup $\bB\subset\bG$ containing $\bT$
and the quotient morphism $q:\bG\rightarrow \bG/\bB=X.$ We will
make heavy use of the following result of V. Berkovich which was
taken up and generalized in a conceptual way in
\cite{RemyThuillierWerner10}. Let $\eta\in X$ be the generic point
of $X$.

\begin{thm}(Berkovich,Remy/Thuillier/Werner)\label{thm-RTW}
There exists a $G$-equivariant injective map
\[\vartheta_\bB:\sB\longrightarrow\Xan\]
which is a homeomorphism onto its image. The latter is a locally
closed subspace of $\Xan$ contained in the preimage
$\pi^{-1}(\eta)$ of the generic point of $X$.
\end{thm}
\begin{proof}
This is \cite[5.5.1]{BerkovichBook}. We sketch the
construction in the language of \cite{RemyThuillierWerner10}. The
map is constructed in three steps. First one attaches to any point
$z\in\sB$ an $L$-affinoid subgroup $\bG_z$ of $\bG^{an}$ whose
rational points coincide with the stabilizer of $z$ in $G$. In a
second step one attaches to $\bG_z$ the unique point in its Shilov
boundary (the {\it sup-norm} on $\bG_z$) which defines a map
$\vartheta:\sB\rightarrow\bG^{an}$. In a final step one composes
this map with the analytification of the orbit map $\bG\rightarrow
X, g\mapsto g.\bB$. The last assertion follows from the next lemma.
\end{proof}
\begin{lemma}\label{lem-stalksarefields} Let $z\in\sB$. The local
rings $\cO_{\Xan,\vartheta_{\bB}(z)}$ and
$\cO_{X,\pi(\vartheta_{\bB}(z))}$ are fields. In particular,
$\pi(\vartheta_{\bB}(z)))=\eta$, the generic point of $X$.
\end{lemma}
\begin{proof}
The surjective orbit map $\bG^{an}\rightarrow\Xan$ restricted to
the affinoid subdomain $\bG_z\subset\bG^{an}$ induces an injective
local homomorphism on local rings. However, the point $z\in\bG_z$
corresponds to a norm on the affinoid algebra of $\bG_z$ whence
the local ring $\cO_{\bG_z,\vartheta(z)}$ is a field. Hence,
$\cO_{\Xan,\vartheta_{\bB}(z)}$ is a field. Since $\pi$ induces a
faithfully flat (and hence injective) local homomorphism on local
rings, $\cO_{X,\pi(\vartheta_{\bB}(z))}$ has to be a field, too.

\end{proof}
Since $\Xan$ is compact the closure of the image of
$\vartheta_\bB$ in $\Xan$ is a compactification of $\sB$
(loc.cit., Remark 3.31). It is called the {\it Berkovich
compactification} of $\sB$ of type $\emptyset$ (loc.cit., Def.
3.30).
We will in the following often identify $\sB$ with its image under
$\vartheta_\bB$ and hence, view $\sB$ as a locally closed subspace
of $\Xan$.

\begin{para}
By \cite[1.5]{BerkovichEtale} the space $\Xan$ is a {\it good}
analytic space (in the sense of loc.cit., Rem. 1.2.16) which means
any point of $\Xan$ lies in the topological interior of an
affinoid domain. In particular, given $x\in\Xan$ the stalk
$\cO_{\Xan,x}$ may be written as
\[\cO_{\Xan,x}=\varinjlim_{x\in V} \cA_V\]
where the inductive limit ranges over the affinoid neighbourhoods
$V$ of $x$ and where $\cA_V$ denotes the associated affinoid
algebra. As usual a subset of neighbourhoods of $x$ will be called
{\it cofinal} if it is cofinal in the directed partially ordered
set of all neighbourhoods of $x$.
If $V$ is an affinoid neighbourhood of $x$, the corresponding
affinoid algebra $\cA_V$ has its Banach topology. We endow
$\cO_{\Xan,x}$ with the locally convex final topology (\cite[\S5.E]{NFA}) arising from the above inductive limit. This topology makes
$\cO_{\Xan,x}$ a topological $L$-algebra. We need another, rather
technical, property of this topology.\end{para}
\begin{lemma}\label{lem-neighbour}
Let $x\in\Xan$. There is a sequence $V_1 \supset V_2 \supset V_3
...$ of irreducible reduced strictly affinoid neighbourhoods of
$x$ which is cofinal and has the property: the homomorphism of
affinoid algebras $\cA_{V_{i}}\ra\cA_{V_{i+1}}$ associated with
the inclusion $V_{i+1}\subset V_{i}$ is flat and an injective
compact linear map between Banach spaces. In partiuclar, the stalk
$\cO_{\Xan,x}$ is a vector space of compact type.
\end{lemma}
\begin{proof}
Being an analytification the analytic space $\Xan$ is closed (in
the sense of \cite[p. 49]{BerkovichBook}). Since the valuation on
$L$ is nontrivial it is therefore strictly $k$-analytic (loc.cit.,
Prop. 3.1.2) . Let $V$ be a strictly affinoid neighbourhood of $x$
in $\Xan$ so that $x$ lies in the topological interior of $V$. In
the following we will use basic results on the relative interior
${\rm Int}(Y/Z)$ of an analytic morphism $Y\ra Z$ (loc.cit.,
2.5/3.1). As usual we write ${\rm Int}(Y)$ in case of the
structure morphism $Y\ra\sM(L)$. Since $\Xan$ is closed we have by
definition ${\rm Int}(\Xan)=\Xan$. Moreover, loc.cit., Prop. 3.1.3
(ii) implies ${\rm Int}(V)={\rm Int}(V/\Xan).$ By part (i) of the
same proposition the topological interior of $V$ is equal to ${\rm
Int}(V/\Xan)$ and, thus, $x\in {\rm Int}(V)$. Now the residue
field of $L$ being finite there is a countable basis
$\{W_n\}_{n\in\bbN}$ of neighbourhoods of $x$ consisting of
strictly affinoid subdomains (even Laurent domains) of $V$
(\cite[Prop. 3.2.9]{BerkovichBook}). By smoothness of $\Xan$ the
local ring $\cO_{\Xan,x}$ is noetherian regular and hence an
integral domain. We may therefore assume that all $W_n$ are
reduced and irreducible (loc.cit., last sentence of 2.3). Consider
$V_1:=W_{n_1}$ for some $n_1\in\bbN$. As we have just seen $x\in
{\rm Int}(V_1)$. Since ${\rm Int(V_1)}$ is an open neighbourhood
of $x$ there is $n_2>n_1$ such that $W_{n_2}\subseteq {\rm
Int}(V_1)$. We put $V_2:=W_{n_2}$ and repeat the above argument
with $V_1$ replaced by $V_2$. In this way we find a cofinal
sequence $V_1 \supset V_2 \supset V_3 ...$ of strictly irreducible
reduced affinoid neighbourhoods of $x$ with the property ${\rm
Int}(V_i)\supseteq V_{i+1}$ for all $i\geq 1$. According to
loc.cit., Prop. 2.5.9 the bounded homomorphism of $L$-Banach
algebras $\cA_{V_{i}}\ra\cA_{V_{i+1}}$ associated with the
inclusion $V_{i+1}\subset V_{i}$ is inner with respect to $L$ (in
the sense of loc.cit., Def. 2.5.1). The arguments in
\cite[Prop. 2.1.16]{EmertonA} now show that
$\cA_{V_{i}}\ra\cA_{V_{i+1}}$ is a compact linear map between
Banach spaces. Finally, this latter map is injective because $V_i$
is irreducible and $V_{i+1}$ contains a nonempty open subset of
$V_i$. It is also flat since, by construction, $V_{i+1}$ is an
affinoid subdomain of $V_i$ (\cite[Prop. 2.2.4
(ii)]{BerkovichBook}).
\end{proof}

For the rest of this section we will {\bf assume} $L=\bbQ_p.$ Let
$z\in\sB\subset\Xan$ be a point. For any $e\geq 0$ the group
$\Uze\subseteq P_{\{z\}}$ fixes the point $z$. Suppose $V$ is an
affinoid neighbourhood of $z$ which is $\Uze$-stable. Since the
$\Uze$-action on $\Xan$ comes from the algebraic action of $\bG$
on the variety $X$ the induced action
\[ \Uze\lra {\rm GL}(\cA_V)\] is a locally analytic
$\Uze$-representation. By section \ref{sec-skew} the completed
skew group ring $\cA_V\cotimes_L D(\Uze,K)$ exists. We shall need
a more refined version of this fact.
\begin{lemma}\label{lem-stableneighbour}
There exists a number $e_0\geq 0$ with the following property. For
any point $z\in\sB$, viewed as a point in $\Xan$, the
distinguished sequence of affinoid neighbourhoods $\{V_n\}_n$ of
$z$ described in the preceding lemma may be chosen in such a way
that each $V_n$ is $U_z^{(e_0)}$-stable.
\end{lemma}
\begin{proof}
Recall our fundamental chamber $\sC$ in the Coxeter complex $A$.
Let $F$ be a facet contained in $\overline{\sC}$ and $z\in F$ a
point. It suffices to establish the following assertion: there
exists a number $e(F)\geq 0$ and a sequence of affinoid
neighbourhoods $\{V_n\}_n$ of $z$ as in the preceding lemma such
that each $V_n$ is $U_z^{(e)}$-stable. Indeed, since any facet in
$\sB$ is a $G$-conjugate of some $F\subseteq\overline{\sC}$ and
since we have (\ref{equ-conjgroups}) we may take
$e_0:=\max_{F\subseteq\overline{\sC}} e(F)$. So let us prove the
assertion. By construction of the sequence $\{V_n\}_n$ in the
preceding lemma it suffices to show that there is $e(F)\geq 0$ and
a cofinal system of affinoid neighbourhoods of $z$ which are
$\Uze$-stable for all $e\geq e(F)$. To do this let $V$ be an
arbitrary affinoid neighbourhood of $z$. Since $X$ is a compact
strictly $L$-analytic space there is formal scheme $\frX$ locally
finitely presented over $o_L$ with generic fibre $X$ and special
fibre $\frX_s$ and an open affine subscheme $\frY\subset\frX_s$
such that $V=sp^{-1}(\frY)$. 
Here $sp: X\rightarrow\frX_s$ is the specialization map associated
to $\frX$. Consider the connected smooth $o_L$-group scheme $\frG$
associated to the (hyper-)special point $x_0$ (cf. \ref{ConcreteDescription})
and a Borel subgroup scheme $\frB\subset\frG$ with generic fibre $\bB$.
Making $V$ smaller if necessary, we may assume that $\frX$ is a
formal admissible blow-up of the formal completion $\frX_0$ of the
$o_L$-scheme $\frG/\frB$ along its special fibre. It is now easy
to see that the algebraic $\frG$-action on $\frG/\frB$ induces,
for $n$ sufficiently large, an action of $\frG(\varpi^n)$ on the
formal scheme $\frX$ stabilizing $\frX_s$ pointwise. The map $sp$
is equivariant with respect to this action whence $\frG(\varpi^n)$
stabilizes $V$. We now let $e(F)$ be any number such that
$U_z^{(e(F))}\subseteq P_{\{x_0\}}$ (Prop. \ref{prop-rootspace} (ii)). Let $t_1=1,...,t_m$ be a system of representatives in
$\Uze$ for the finite group $\Uze/(\Uze\cap\frG(\varpi^n))$. Since
$\Xan$ is separated $W:=\cap_i t_iV$ is an affinoid $\Uze$-stable
neighbourhood of $z$ contained in $V$.
\end{proof}

\subsection{A structure sheaf on the building}\label{subsec-structuresheaves}

\begin{para}
To be able to compare the localization of Schneider-Stuhler and
Beilinson-Bernstein we equip the topological space $\sB$ with a
sheaf of commutative and topological $L$-algebras. Recall that
$\cO_{\Xan}$ is naturally a sheaf of locally convex algebras. We
consider the exact functor $\sF$ from abelian sheaves on $\Xan$ to
abelian sheaves on $\sB$ given by restriction along
$\vartheta_\bB:\sB\hookrightarrow\Xan$. Let
$$\cO_{\sB}:=\sF(\cO_{\Xan}).$$ Given an open set
$\Omega\subseteq\sB$ we have the canonical map $\varinjlim_{U}
\cO_{\Xan}(U)\ra \cO_\sB(\Omega)$ where $U$ runs through the open
subsets of $\Xan$ containing $\Omega$. The source has the locally
convex inductive limit topology and the target receives the final
locally convex topology. We point out that the stalk
$\cO_{\sB,z}=\cO_{\Xan,z}$ for any point $z\in\sB$ is in fact a field (Lem. \ref{lem-stalksarefields}).

\vskip8pt

We have the following properties of $\sF$:

\begin{itemize}
    \item[(1)] $\sF$ preserves (commutative) rings, $L$-algebras, $L$-Lie
    algebras and $G$-equivariance,\vskip8pt
    \item[(2)] $\sF$ maps $\cO_{\Xan}$-modules into
    $\cO_\sB$-modules,\vskip8pt
    \item[(3)] $\sF$ induces a Lie algebra homomorphism ${\rm
    Der}_L(\cO_{\Xan})\ra {\rm Der}_L(\cO_\sB)$, \vskip8pt
    \item[(4)] $\cO_{\sB}$ is a sheaf of {\it locally convex}
    $L$-algebras and the stalk $\cO_{\sB,z}$ is of compact type with a defining system of $U_z^{(e_0)}$-stable
    Banach algebras for all $z\in\sB$.

\end{itemize}

\vskip8pt

Composing the map $\frg\ra {\rm Der}_L(\cO_{\Xan})$ from (Lem. \ref{lem-analytification}) with (4) yields a Lie algebra
homomorphism $\frg\ra {\rm Der}_L(\cO_\sB)$ and the associated
skew enveloping algebra $\cO_\sB\smprod U(\frg).$ By (1),(2) we
have the $L$-Lie algebras and $\cO_\sB$-modules
$\frn_\sB^{\ci,an}:= \sF(\frn^{\ci,an})$ and
$\frb_\sB^{\ci,an}:=\sF(\frb^{\ci,an})$. Similarly,
$$\lambda_\sB^{\ci,an}:=\sF(\lambda^{\ci,an}): \frb_\sB^{\ci,an}\lra
\cO_\sB$$ is a morphism of $L$-Lie algebras and $\cO_\sB$-modules.
Let $\cI^{an}_{\sB,\frt}$ resp. $\cI^{an}_{\sB,\chi}$ be the right
ideal sheaf of $\cO_{\sB}\smprod U(\frg)$ generated by
$\frn_\sB^{\ci,an}$ resp. $\ker (\lambda_\sB^{\ci,an})$. One
checks that these are two-sided ideals. We let
\[\begin{array}{lcr}
 \cD^{an}_{\sB,\frt}:=(\cO_\sB\smprod U(\frg))/\cI^{an}_{\sB,\frt}  & {\rm and} &\cD^{an}_{\sB,\chi}:=(\cO_\sB\smprod
U(\frg))/\cI^{an}_{\sB,\chi}.  \\
\end{array}\]

Note that by exactness of $\sF$ we have \[\begin{array}{lcr}
 \cD^{an}_{\sB,\frt}=\sF(\cD^{an}_{\frt}) & {\rm and} &\cD^{an}_{\sB,\chi}=\sF(\cD^{an}_{\chi}).  \\
\end{array}\]

\end{para}

\begin{para}

The sheaf $\cD^{an}_{\sB,\chi}$ of twisted differential operators
on $\sB$ is formed with respect to the Lie algebra action of
$\frg$ on the ambient space $\sB\subset\Xan$. In an attempt to
keep track of the whole analytic $G$-action on $\Xan$ we will
produce in the following a natural injective morphism of sheaves
of algebras
\[\cD^{an}_{\sB,\chi}\lra \sD_{r,\chi}\] with target a sheaf of what we tentatively call {\it twisted
distribution operators} on $\sB$. Actually, there will be one such
sheaf for each `radius' $r\in [r_0,1)$ in $p^{\bbQ}$ and each
sufficiently large 'level' $e>0$. Again following \cite{SchSt97}
we suppress the dependence on the level in our the notation.



\subsection{Mahler series and completed skew group rings}

\begin{para}

Suppose for a moment that $\cA$ is an arbitrary $L$-Banach
algebra. Since $\bbQ_p\subset\cA$ the completely valued
$\bbZ_p$-module $(\cA,|.|)$ is saturated in the sense of
\cite[I.2.2.10]{Lazard65}. Consequently, we have the theory of
Mahler expansions over $\cA$ at our disposal (loc.cit., III.1.2.4
and III.1.3.9). In this situation we prove a version of the
well-known relation between decay of Mahler coefficients and
overconvergence.
\end{para}
\begin{prop}\label{prop-overconvergence}
Let $f=\sum_{\al\in\bbN_0^d} a_\alpha x^\alpha$ be a $d$-variable
power series over $\cA$ converging on the disc $|x_i|\leq R$ for
some $R>1$. Let $c>0$ be a constant such that $|a_\al|\leq
cR^{-|\alpha|}$ for all $\alpha$. Let
\[f(\cdot)=\sum_{\al\in\bbN_0^d}c_{\al} {{\cdot}\choose {\al}},\]
$c_{\al}\in\cA$, be the Mahler series expansion of $f$. Then
$|c_\alpha|\leq c s^{|\alpha|}$ for all $\alpha$ where
$s=r_1R^{-1}$ with $r_1=p^{\frac{-1}{p-1}}$.
\end{prop}

\begin{proof}
We prove the lemma in case $d=1$. The general case follows along
the same lines but with more notation. We define the following
series of polynomials over $\bbZ$
\[ (x)_0=1,...,(x)_k=x(x-1)\cdot\cdot\cdot (x-k+1) \]
for $k\geq 1$. The $\bbZ$-module $\bbZ[x]$ has the $\bbZ$-bases
$\{x^k\}_{k\geq 0}$ and $\{(x)_k\}_{k\geq 0}$ and the transition
matrices are unipotent upper triangular. We may therefore write
\begin{numequation}\label{stirling}x^n=\sum_{k=0,...,n}s(n,k)(x)_k\end{numequation} with $s(n,k)\in\bbZ$. Put
$b_k:=c_k/k!$. Then $$\sum_{k\geq 0} c_k{{x}\choose
{k}}=\sum_{k\geq 0} b_k(x)_k$$ is a uniform limit of continuous
functions (even polynomials) on $\bbZ_p$ (cf. \cite[Thm.
VI.4.7]{Robert}). We now proceed as in (the proof of)
\cite[Prop. 5.8]{WashingtonBook}. Fix $i\geq 1$ and write
\[ \sum_{n\leq i} a_nx^n=\sum_{k\leq i} b_{k,i}(x)_k\]
as polynomials over $\cA$ with some elements $b_{k,i}\in\cA$.
Inserting (\ref{stirling}) and comparing coefficients yields
$b_{k,i}=\sum_{k\leq n \leq i} a_n s(n,k)$ and consequently,
\[ |b_{k,i}|\leq\max_{k \leq n \leq i} |a_n| \leq \max_{k \leq
n\leq i} (cR^{-n})\leq cR^{-k}\] since $R^{-1}<1$. We easily
deduce from this that $\{b_{k,i}\}_{i\geq 0}$ is a Cauchy sequence
in the Banach space $\cA$. Let $\tilde{b}_k$ be its limit.
Clearly, $|\tilde{b}_k|\leq cR^{-k}$. Put $\tilde{c}_k:=k!\cdot
\tilde{b}_k$. Since $|k!|\leq (r_1)^k$ we obtain $|\tilde{c}_k|
\leq c (r_1R^{-1})^k=cs^k$ for all $k$. By definition of
$\tilde{b}_k$ the series of polynomials
\[\sum_{k\geq 0}\tilde{c}_k{{x}\choose {k}}=\sum_{k\geq 0}
\tilde{b}_k(x)_k\] converges pointwise to the limit
\[\lim_{i\rightarrow\infty}\sum_{k\leq i} b_{k,i}(x)_k=\lim_{i\rightarrow\infty}\sum_{n\leq i}a_nx^n=f(x).\] By \cite[IV.2.3, p. 173]{Robert} this convergence is uniform and so uniqueness of
Mahler expansions implies $\tilde{c}_k=c_k$ for all $k$. This
proves the lemma.
\end{proof}
We let $e_1$ be the smallest integer strictly bigger than
$\frac{e(L/\bbQ_p)}{ p-1}.$ Recall that the usual logarithm series
$\log(1+X)=X-\frac{X^2}{2}+\frac{X^3}{3}...$ provides an
isomorphism of topological groups $1+(\varpi_L o_L)^{e}\car
(\varpi_L o_L)^{e}$ whenever $e\geq e_1$ (\cite[Prop.
5.5]{NeukirchI}).

\vskip8pt

Recall that a {\it special} analytic subdomain $V\subset\Xan$ is a
finite union over affinoid subdomains and has an associated Banach
algebra $\cA_V$ (\cite[Cor. 2.2.6]{BerkovichBook}).

\begin{cor}\label{cor-extendmodule}
Let $L=\bbQ_p$ and $z\in\sB$. Consider a $U_z^{(e-1)}$-stable
special analytic domain $V$ of $\Xan$ for $e>e_1$. The locally
analytic representation $\rho: \Uze\ra {\rm GL}(\cA_{V})$
satisfies the assumption {\rm $(\star)$} of section \ref{sec-skew}
for any $r\in [r_0,1)$.
\end{cor}
\begin{proof}
Let $\delta\in D_r(\Uze,L)$ and $f\in\cA_V$. Let $\rho^{(e)}_f:
g\mapsto g.f$ be the orbit map of $f$. In particular,
$(g.f)(x)=f(g^{-1}x)$ for any $x\in V$. The map $\rho^{(e)}_f$ is
a $\cA_V$-valued locally analytic function on $\Uze$. The ordered
basis $(h_1,...,h_d)$ of the $p-$valuation on $\Uze$ determines a
global chart $\Uze\simeq (\bbZ_p)^d$. We have the theory of Mahler
expansions over $\cA_V$ at our disposal. The function
$\rho^{(e)}_f$, viewed as a function on $(\bbZ_p)^d$, is locally
analytic and hence, its Mahler expansion
\[\rho_f^{(e)} (\cdot)=\sum_{\al\in\bbN_0^d}c_{f,\al} {{\cdot}\choose
{\al}},\] $c_{f,\al}\in\cA_V$ has the property (*):
$|c_{f,\al}|_V\leq c s^{|\al|}$ for some constant $c>0$ and some
real number $0<s<1$.

Now $V$ is also stable under the larger group $U_z^{(e-1)}$ and we
may apply the above discussion to the triple $(f,U_z^{(e-1)},V)$.
The function $\rho_f^{(e)}$ equals the restriction of
$\rho_f^{(e-1)}$ to $U_z^{(e)}\subseteq U_z^{(e-1)}$. We see that
the Mahler series of $\rho_f^{(e)}$ can be expanded in a power
series converging on the disc of radius $p$. The preceding lemma
applied to $R=p$ implies that the Mahler coefficients $c_{f,\al}$
of $\rho_f^{(e)}$ satisfy the property (*) for
\[s=r_1p^{-1}=p^{\frac{-1}{p-1}-1}=p^{\frac{-p}{p-1}}<p^{-1}=r_0.\]
Let us fix such a number $s$. We may expand $\delta$ into a series
$\delta=\sum_{\al\in\bbN_0^d}d_\al\bb^\al$ with $d_\al\in L$ such
that $|d_\al|r^{|\al|}\rightarrow 0$. Since $s\leq r$ the sum
\begin{numequation}\label{sum}
\delta.f:=\delta(\rho_f^{(e)}):=\sum_{\al\in\bbN_0^d}d_\al
c_{f,\al}\end{numequation} converges in the Banach space $\cA_V$
according to $(*)$. The map $(\delta,f)\mapsto \delta.f$ makes
$\cA_V$ a topological module over $D_r(\Uze,L)$ in a way
compatible with the map $D(\Uze)\ra D_r(\Uze)$.
\end{proof}
Until the end of this section we will {\bf assume} $L=\bbQ_p,
e>\max (e_0,e_1)$ and $r\in [r_0,1)$.
\begin{prop}\label{prop-skewexist}
Let $z\in\sB$. Then the stalk $\cO_{\sB,z}$ may be written as an
inductive limit over $\Uze$-stable Banach algebras $\cA_V$ with
the property: the completed skew group rings $\cA_V\smprod
D_r(\Uze,K)$ and $\sta\smprod D_r(\Uze,K)$ exist and the natural
map
\[ \varinjlim_V\;(\cA_V\smprod D_r(\Uze,K))\car \sta\smprod
D_r(\Uze,K)\] is an isomorphism of topological $K$-algebras.
\end{prop}
\begin{proof}
This follows from section \ref{sec-skew} together with the above
results.
\end{proof}
The following corollary is immediate and recorded only for future
reference.
\begin{cor}\label{cor-homgerm}
Keep the assumptions and notations. Let $\iota_z$ be the natural
map $\cA_V\rightarrow \sta$ sending a function to its germ at $z$.
The map $\iota_z\cotimes_L {id}$ is an algebra homomorphism
\[\cA_V\smprod D_r(\Uze,K)\lra \sta\smprod D_r(\Uze,K).\]
\end{cor}
\begin{cor}
Keep the assumptions and notations. The inclusions
$L[\Uze]\subseteq D_r(\Uze,K)$ and $U(\frg)_K\subseteq
D_r(\Uze,K)$ induce algebra homomorphisms
\begin{itemize}
    \item[(i)] $\cA_V\smprod\Uze=\cA_V\otimes_L
    L[\Uze]\longrightarrow\cA_V\smprod
D_r(\Uze,K),$\vskip8pt
    \item[(ii)]$\cA_V\smprod U(\frg)_K\longrightarrow\cA_V\smprod
D_r(\Uze,K).$
\end{itemize}
Varying $V$ these maps assemble to algebra homomorphisms
\begin{itemize}
    \item[(i)] $\sta\smprod\Uze=\sta\otimes_L
    L[\Uze]\longrightarrow\sta\smprod
D_r(\Uze,K),$\vskip8pt
    \item[(ii)]$\sta\smprod U(\frg)_K\longrightarrow\sta\smprod
D_r(\Uze,K).$
\end{itemize}
\end{cor}
\begin{proof}
Consider the case of $\cA_V$. The map (i) follows from (\ref{prop-skewextensions}). The same is true for the map (ii) once
we convince ourselves that there is a commutative diagram of
algebra homomorphisms

\[\xymatrix{
U(\frg) \ar[d]^\subseteq \ar[r]^>>>>>>{\al^{\ci,an}} &  \End_L(\cA_V) \ar[d]^{Id}\\
D_r(\Uze,L) \ar[r] & \End_L(\cA_V) }
\] where the upper horizontal arrow is derived from (\ref{equ-alpha-analytic}) and
the lower horizontal arrow describes the $D_r(\Uze,L)$-module
structure of $\cA_V$ as given by (Cor. \ref{cor-extendmodule}).
Restricting the lower horizontal arrow to $\frg$ amounts to
differentiate the locally analytic $\Uze$-action on $\sta$. This
action comes from the algebraic action of $\bG$ on $X$. The
diagram commutes by the remark following (Lem.
\ref{lem-analytification}). Having settled the case $\cA_V$ the
case of $\sta$ now follows by passage to the inductive limit.
\end{proof}

As a result of this discussion we have associated to each point
$z\in\sB\subset\Xan$ a (noncommutative) topological $K$-algebra,
namely the completed skew group ring $\sta\smprod D_r(\Uze,K)$. As
we have seen it comes together with an injective algebra
homomorphism
\begin{numequation}\label{equ-alghom} \cO_{\sB,z}\smprod
U(\frg)_K\hookrightarrow \sta\smprod D_r(\Uze,K).\end{numequation}

In the next section we will sheafify this situation and obtain a
sheaf of noncommutative $K$-algebras $\cO_\sB\smprod D_r$ on $\sB$
together with an injective morphism of sheaves of algebras
\[\cO_{\sB}\smprod U(\frg)_K\hookrightarrow \cO_\sB\smprod D_r\]
inducing the map  (\ref{equ-alghom}) at all points $z\in\sB$.

\vskip8pt

To do this we shall need a simple `gluing property' of the
algebras $\sta\smprod D_r(\Uze,K)$.
\begin{lemma}\label{lem-homglue}
Let $F,F'$ be facets in $\sB$ such that $F'\subseteq\overline{F}$
and let $\sigma_r^{F'F}: D_r(U_{F'}^{(e)},K)\rightarrow
    D_r(\UFe,K)$ be the corresponding algebra homomorphism.
    Suppose $V$ and $V'$ are two special analytic domains
    stable under $\UFe$
    and $U^{(e)}_{F'}$ respectively. If $V\subseteq V'$ the map $res^{V'}_V\cotimes\sigma_r^{F'F}$ is an algebra homomorphism
    \[res^{V'}_V\cotimes\sigma_r^{F'F}: \cA_{V'}\smprod D_r(U_{F'}^{(e)},K)\lra
    \cA_{V}\smprod D_r(\UFe,K).\]
\end{lemma}
\begin{proof}
Since the map $\sigma_r^{F'F}$ is induced from the inclusion
$U_{F'}^{(e)}\subseteq U_F^{(e)}$ there is a commutative diagram

\[\xymatrix{
D_r(U_{F'}^{(e)},K) \ar[d]^{\sigma_r^{F'F}\times{res}} \times \cA_{V'} \ar[r] & \cA_{V'}  \ar[d]^{res}\\
D_r(\UFe,K)\times \cA_{V} \ar[r] & \cA_V }.
\]
where the horizontal arrows describe the module structures of
$\cA_{V'}$ and $\cA_V$ over $D_r(U_{F'}^{(e)},K)$ and
$D_r(\UFe,K)$ respectively (Lem. \ref{cor-extendmodule}). The
assertion follows now from the construction of the skew
multiplication of the source and target of
$res^{V'}_V\cotimes\sigma_r^{F'F}$ (cf. sec. \ref{sec-skew}).
\end{proof}

\section{A Sheaf of 'distribution operators' on the building}
In this section we keep our assumptions, i.e. we {\bf assume}
$L=\bbQ_p, e>\max (e_0,e_1)$ and $r\in [r_0,1)$. Moreover, we will
assume that $e\geq e_{uni}$ so that all groups $\Uze$ will be
uniform pro-$p$ groups (4.3.3). We will work from now on
exclusively over the coefficient field $K$. To ease {\bf notation}
we will therefore drop this coefficient field from the notation
when working with distribution algebras. We thus write
$D(G)=D(G,K), D_r(\UFe)=D_r(\UFe,K)$ etc.

\vskip8pt

Recall that the sheaf of (twisted) differential operators
$\cD_\chi$ on $X$ may be constructed from the skew tensor product
$\cO_X\smprod U(\frg)$ (sec. 5.1). In a similar way we are going
to construct a sheaf of 'distribution operators' on $\sB$ starting
from a twisted tensor product $\shfo$. Here, $\tiD$ replaces the
constant sheaf $\underline{U(\frg)}$ and equals a sheaf of
distribution algebras on $\sB$. It will be a constructible sheaf
with respect to the usual polysimplicial structure of $\sB$.
Recall from 4.2. that a sheaf on a polysimplicial space is called
{\it constructible} if its restriction to a given geometric
polysimplex is a constant sheaf.

\subsection{A constructible sheaf of distribution algebras}
Recall that the {\it star} of a facet $F'$ in $\sB$ is the subset
of $\sB$ defined by
\[ St(F'):={\rm
~union~of~all~facets~}F\subseteq\sB {\rm~such~that~}
F'\subseteq\overline{F}.\] These stars form a locally finite open
covering of $\sB$. Given an open set $\Omega\subset\sB$ we have
the natural map
\[\iota_z: \cO_{\sB}(\Omega)\longrightarrow\sta,~~~~f\mapsto {\rm
germ~of~}f{\rm~at~}z\] for any $z\in\Omega$.

\begin{dfn}\label{def-sheaf}
For an open subset $\Omega\subseteq\sB$ let
\[
\tiD(\Omega):=K{\rm -vector~ space~ of~ all~ maps~}
s:\Omega\rightarrow\dot{\bigcup}_{z\in\Omega} D_r(\Uze) {\rm ~s.
t.}
\]
\begin{itemize}
    \item[(1)] $s(z)\in
D_r(\Uze)$ for all $z\in\Omega$,\vskip8pt
    \item[(2)] for each facet $F\subseteq\sB$ there exists a finite open covering $\Omega\cap St(F)=\cup_{i\in
    I}\Omega_i$ with the property: for each $i$ with $\Omega_i\cap
    F\neq\emptyset$ there is an element $s_i\in
    D_r(\UFe)$ such that

    \vskip8pt

    \begin{itemize}
    \item[(2a)] ~$s(z)=s_i
    {\rm~for~any~} z\in\Omega_i\cap F,$

    \vskip8pt
    \item[(2b)]
    ~$s(z')=\sigma_r^{FF'}(s_i){\rm~for~any~} z'\in
    \Omega_i$. Here, $F'$ is the unique facet in $St(F)$ that contains $z'$.
\end{itemize}\end{itemize}
\end{dfn}
>From (2a) it is almost obvious that the restriction of $\tiD$ to a
facet $F$ is the constant sheaf with value $D_r(\UFe)$. Hence
$\tiD$ is constructible. Furthermore, if $\Omega'\subseteq\Omega$
is an open subset there is the obvious restriction map
$\tiD(\Omega)\rightarrow\tiD(\Omega')$. The proof of the following
result is implicitly contained in the proofs of (Lem.
\ref{lem-sheafI}) and (Lem. \ref{lem-stalkI}) below.
\begin{lemma}
With pointwise multiplication $\tiD$ is a sheaf of $K$-algebras.
For $z\in\sB$ one has $(\tiD)_z=D_r(\Uze)$.
\end{lemma}

\subsection{Sheaves of completed skew group rings}

\begin{dfn}\label{def-sheaf}
For an open subset $\Omega\subseteq\sB$ let
\[
\shf(\Omega):=K{\rm -vector~ space~ of~ all~ maps~}
s:\Omega\rightarrow\dot{\bigcup}_{z\in\Omega} \sta\smprod
D_r(\Uze) {\rm ~such~that}
\]
\begin{itemize}
    \item[(1)] $s(z)\in\sta\smprod
D_r(\Uze)$ for all $z\in\Omega$,\vskip8pt
    \item[(2)] for each facet $F\subseteq\sB$ there exists an open covering $\Omega\cap St(F)=\cup_{i\in
    I}\Omega_i$ with the property: for each $i$ with $\Omega_i\cap
    F\neq\emptyset$ there is an
element \[s_i\in
    \cO_{\sB}(\Omega_i)\cotimes_L D_r(\UFe)\] such that

    \begin{itemize}
    \item[(2a)] ~$s(z)=(\iota_z\cotimes {\rm id})(s_i)
    {\rm~for~any~} z\in\Omega_i\cap F,$\vskip8pt

    \item[(2b)]
    ~$s(z')=(\iota_{z'}\cotimes\sigma_r^{FF'})(s_i){\rm~for~any~} z'\in
    \Omega_i$. Here, $F'$ is the unique facet in $St(F)$ that contains $z'$.
\end{itemize}\end{itemize}
\end{dfn}

Consider a function $s: \Omega\rightarrow\dot{\cup}_{z\in\Omega}
\sta\smprod D_r(\Uze)$ satisfying $(1)$. It will be convenient to
call an open covering $\Omega\cap St(F)=\cup_{i\in I} \Omega_i$
together with the elements $s_i$ such that $(2a)$ and $(2b)$ hold
a {\it datum} for $s$ with respect to the facet $F$. Any open
covering of $\Omega\cap St(F)$ which is a refinement of the
covering $\{\Omega_i\}_{i\in I}$, together with the same set of
elements $s_i$ is again a datum for $s$ with respect to $F$.

\begin{lemma}\label{lem-refine}
Let $\Omega\subseteq \sB$ be an open set and let
$s\in\shf(\Omega).$ Let $F$ be a facet in $\sB$. There is a datum
$\{\Omega_i\}_{i\in I}$ for $s$ with respect to $F$ such that, in
case $\Omega_i\cap F\neq\emptyset$, the set $\Omega_i$ is
$\UFe$-stable and the skew group ring $\cO_\sB(\Omega_i)\smprod
D_r(\UFe)$ exists.
\end{lemma}

\begin{proof} Let $\{\Omega_i\}_{i\in I}$ be a datum for
$s$ with respect to $F$. For any $z\in \Omega\cap St(F)$ choose
$i(z)\in I$ such that $z\in\Omega_{i(z)}$. In case
$\Omega_{i(z)}\cap F=\emptyset$ we let $W_{i(z)}:=\Omega_{i(z)}$.
In case $\Omega_{i(z)}\cap F\neq\emptyset$ we replace
$\Omega_{i(z)}$ by an open $\Uze$-stable neighbourhood $W_{i(z)}$
of $z$. Varying the point $z$ we obtain in this way an open
covering
$$\Omega\cap St(F)=\cup_{z\in\Omega\cap St(F)}W_{i(z)}$$ which
is a refinement of the covering $\{\Omega_i\}_{i\in I}$. In case
$W_{i(z)}\cap F\neq\emptyset$ we let
$t_{i(z)}\in\cO_{\sB}(W_{i(z)})\cotimes_L D_r(\UFe)$ be the image
of $s_{i(z)}$ under the map $res\cotimes {\rm id}$ where $res:
\cO_\sB(\Omega_{i(z)})\ra  \cO_\sB(W_{i(z)})$ is the restriction
map. The verification of $(2a)$ and $(2b)$ for these new elements
is straightforward. We thus have a new datum for $s$ with respect
to the facet $F$ that has the desired property.
\end{proof}

\vskip8pt

Suppose $\Omega'\subseteq\Omega$ is an open subset and let
$s\in\shf(\Omega)$. Let $F\subseteq\sB$ be a facet. Given a
corresponding datum $\{\Omega_i\}_{i\in I}$ for $s$ put
$\Omega_i':=\Omega'\cap\Omega_i$. Together with the elements $s_i$
, in case $\Omega'_i\cap F\neq\emptyset$, we obtain a datum for
the function $s|_{\Omega'}$. It follows that $\shf$ is a presheaf
of $K$-vector spaces on $\sB$.

\vskip8pt

In the following it will be convenient to define $\cF(\Omega)$ as
the $K$-vector space of all maps
\[s: \Omega\rightarrow\dot{\bigcup}_{z\in\Omega} \sta\smprod_L
D_r(\Uze)\] satisfying condition $(1)$ in the above definition. It
is clear that pointwise multiplication makes $\cF$ a sheaf of
$K$-algebras on $\sB$ such that $\shf$ is a subpresheaf of
$K$-vector spaces.

\begin{lemma}\label{lem-sheafI}
The induced multiplication makes $\shf\subseteq\cF$ an inclusion
of sheaves of $K$-algebras.
\end{lemma}
\begin{proof}
Consider an open subset $\Omega\subseteq\sB$. Let us first show
that for $s,s'\in\shf(\Omega)$ we have $ss'\in\shf(\Omega)$, i.e.
that $\shf(\Omega)$ is a subalgebra of $\cF(\Omega)$.

 To do this
let $F\subseteq\sB$ be a facet. Let $\{\Omega_i\}_{i\in I}$ and
$\{\Omega'_j\}_{j\in J}$ be corresponding data for $s$ and $s'$
respectively. Passing to $\{\Omega_{ij}\}_{ij}$ with
$\Omega_{ij}=\Omega_i\cap\Omega'_j$ we may assume: there exists
one datum $\{\Omega_i\}_{i\in I}$ for both $s,s'$. By (Lem.
\ref{lem-refine}) we may also assume that each set $\Omega_i$ is
$\UFe$-stable such that the skew group ring
$\cO_{\sB}(\Omega_{i})\smprod D_r(\UFe)$ exists. The element
$$s_is'_i\in \cO_{\sB}(\Omega_{i})\smprod D_r(\UFe)$$ is therefore well
defined. We will show that $\{\Omega_i\}_{i\in I}$ together with
the elements $s_is'_i$, in case $\Omega_i\cap F\neq\emptyset$, is
a datum for the function $ss'$. Let us check $(2a)$. Given $z\in
\Omega_i\cap F$ we compute
\[ (ss')(z)=s(z)s'(z)=(\iota_z\cotimes {id})(s_i)\cdot(\iota_z\cotimes
{id})(s'_i)=(\iota_z\cotimes {id})(s_is'_i)\] according to (Cor. \ref{cor-homgerm}). To check $(2b)$ we let $z'\in\Omega_i$ and
compute
\[ (s\cdot s')(z')=s(z')s'(z')=(\iota_z\cotimes\sigma_r^{F'F})(s_i)\cdot(\iota_z\cotimes\sigma_r^{F'F})(s'_i)=
(\iota_z\cotimes\sigma_r^{F'F})(s_i\cdot s'_i)\] according to (Lem. \ref{lem-homglue}). This shows $(2b)$. Consequently,
$ss'\in\shf(\Omega)$ and hence, $\shf(\Omega)$ is a subalgebra of
$\cF(\Omega)$. If $\Omega'\subseteq\Omega$ is an open subset the
restriction map $\shf(\Omega)\rightarrow\shf(\Omega')$ is
obviously multiplicative. Thus, $\shf$ is a presheaf of
$K$-algebras.

Let us show that $\shf$ is in fact a sheaf. Since
$\shf\subseteq\cF$ is a subpresheaf and $\cF$ is a sheaf it
suffices to prove the following: if $$\Omega=\bigcup_{j\in J}
U_j$$ is an open covering of an open subset $\Omega\subseteq\sB$
and if $s_j\in\shf(U_j)$ are local sections with $s_j|_{U_j\cap
U_i}=s_i|_{U_i\cap U_j}$ for all $i,j\in J$ then the unique
section $s\in\cF(\Omega)$ with $s|_{U_j}=s_j$ for all $j\in J$
lies in $\shf(\Omega)$. To do this let $F\subseteq\sB$ be a facet.
Consider for each $j\in J$ a datum $\{U_{ji}\}_{i\in I}$ for
$s_j$. In particular, $U_j\cap St(F)=\cup_{i\in I} U_{ji}$ and
there are distinguished elements
$$s_{ji}\in\cO_{\sB}(U_{ji})\cotimes_L D_r(\UFe)$$ whenever
$U_{ji}\cap F\neq\emptyset$. Then $\Omega\cap St(F)=\cup_{ji}
U_{ji}$ (together with the elements $s_{ji}$ whenever $U_{ji}\cap
F\neq\emptyset$) is a datum for $s$. Indeed, given $z\in
U_{ji}\cap F$ one has $s(z)=s_j(z)=(\iota_z\cotimes id)(s_{ji})$
which shows $(2a)$. Moreover, if $z'\in U_{ji}$ one has
$s(z')=s_j(z')=(\iota_z\cotimes\sigma_r^{F'F})(s_{ji})$ which
shows $(2b)$. Together this means $s\in\shf(\Omega)$.
\end{proof}
The next proposition shows that the stalks of the sheaf $\shf$ are
as expected.
\begin{lemma}\label{lem-stalkI}
One has $\shf_z=\sta\smprod D_r(\Uze)$ for any $z\in\sB$.
\end{lemma}
\begin{proof}
There is the $K$-algebra homomorphism
\[ \shf_z\lra\sta\smprod
D_r(\Uze),~~~{\rm~germ~of~}s{\rm~at~}z\mapsto s(z).\] Let us show
that this map is injective. Let $[s]$ be the germ of a local
section $s\in\shf(\Omega)$ over some open subset
$\Omega\subseteq\sB$ with the property $s(z)=0$. Let $F$ be the
unique facet of $\sB$ that contains $z$ and let
$\{\Omega_i\}_{i\in I}$ be a corresponding datum for $s$. We may
assume, in case $\Omega_i\cap F\neq\emptyset$, that the
topological space $\Omega_i$ is irreducible. In this case the map
$\iota_z: \cO_{\sB}(\Omega_i)\rightarrow\sta$ is injective for
every point $z\in\Omega_i$. According to \cite[Cor.
1.1.27]{EmertonA} each map
$$\iota_z\cotimes {id}: \cO(\Omega_i)\cotimes_L
D_r(\UFe)\rightarrow \sta\cotimes_L D_r(\UFe)$$ remains injective.
If $i_0\in I$ is such that $z\in\Omega_{i_0}\cap F$ we therefore
conclude from  $0=s(z)=(\iota_z\cotimes {id}) (s_{i_0})$ that
$s_{i_0}=0$. Given $z'\in\Omega_{i_0}$ let $F'$ be the unique
facet of $St(F)$ containing $z'$. Then
$s(z')=(\iota_z\cotimes\sigma_r^{F'F})(s_{i_0})=0$ according to
$(2b)$ and, consequently, $s|_{\Omega_{i_0}}=0$. Since
$\Omega_{i_0}$ is an open neighbourhood of $z$ this shows $[s]=0$
and proves injectivity.

Let us now show that our map in question is surjective. Let
$t\in\sta\smprod D_r(\Uze)$ be an element in the target. Since the
stalk $\cO_{\sB,z}$ is a compact limit there is an open
$\Uze$-stable neighbourhood $\Omega'$ of $z$ and an element
$\tilde{s} \in\cO_{\sB}(\Omega')\smprod D_r(\Uze)$ such that
$(\iota_z\cotimes {id})(\tilde{s})=t$. Let $F\subseteq\sB$ be a
facet containing $z$ and define
$$\Omega:=\Omega'\cap St(F),~~~~~s:=(res^{W}_\Omega\cotimes
{\rm id}) (\tilde{s})\in\cO_{\sB}(\Omega)\smprod D_r(\Uze).$$
Since $St(F)$ is an open neighbourhood of $z$ and contains only
finitely many facets of $\sB$ we may pass to a smaller $\Omega'$
(and hence $\Omega$) and therefore assume: any $F'\in
(\sB\setminus St(F))$ satisfies $F'\cap\Omega=\emptyset$. For any
$z'\in\Omega$ let $s(z'):=(\iota_{z'}\cotimes\sigma_r^{FF'})(s)$
where $F'$ denotes the facet in $St(F)$ containing $z'$. This
defines a function
\[s: \Omega\rightarrow\dot{\bigcup}_{z'\in\Omega}
\sta\smprod D_r(U_{z'}^{(e)})\] satisfying condition (1) of
definition (\ref{def-sheaf}). According to (Lem. \ref{lem-glue}) one
has $\sigma_r^{FF}={\rm id}$ whence
\[s(z)=(\iota_z\cotimes {\rm id})(s)=(\iota_z\cotimes {\rm
id})(\tilde{s})=t.\] Thus, the germ of $s$ at $z$ will be a
preimage of $t$ once we have shown that $s\in\shf(\Omega)$.

To do this consider an arbitrary facet $F'\subseteq\sB$ together
with the covering of $\Omega\cap St(F')$ consisting of the single
element
\[\Omega_0:=\Omega\cap St(F').\] Suppose $\Omega_0\cap
F'\neq\emptyset$. We have to exhibit an element
$s_0\in\cO_{\sB}(\Omega_0)\smprod D_r(U_{F'}^{(e)})$ satisfying
$(2a)$ and $(2b)$. Since $F'\in St(F)$ we may define $s_0:=({\rm
id}\cotimes\sigma_r^{FF'})(s).$ For any $z'\in\Omega\cap F'$ we
compute
\[ s(z')=(\iota_{z'}\cotimes\sigma_r^{FF'})(s)=(\iota_{z'}\cotimes{\rm id}) ({\rm id}\cotimes\sigma_r^{FF'})(s)
=(\iota_{z'}\cotimes {\rm id})(s_0)\] which shows $(2a)$.
Moreover, for any $z'\in\Omega_0$ we compute
\[ s(z')=(\iota_{z'}\cotimes\sigma_r^{FF''})(s)=
(\iota_{z'}\cotimes\sigma_r^{F'F''})\;(\iota_{z'}\cotimes\sigma_r^{FF'})(s)=(\iota_{z'}\cotimes\sigma_r^{F'F''})(s_0)\]
by (Lem. \ref{lem-glue}). Here $F''$ denote the facet of $St(F')$
that contains $z'$. This shows $(2b)$ and completes the proof.
\end{proof}
\begin{cor} The $\cO_{\sB,z}$-module structure on $\shf_z$ for any
$z\in\sB$ sheafifies to a $\cO_\sB$-module structure on $\shf$
(compatible with scalar multiplication by $L$).
\end{cor}
\begin{proof}
As with any sheaf (\cite[II.1.2]{Godement}) we may regard $\shf$
as the sheaf of {\it continuous} sections of its \'etale space

\[\xymatrix{
\dot{\bigcup}_{z\in\sB} \shf_z\ar[d]\\
\sB}.
\]
and the same applies to the sheaf $\cO_\sB$. By the preceding
proposition we have $\shf_z=\sta\smprod D_r(\Uze)$ for any
$z\in\sB$. Let $\Omega\subseteq\sB$ be an open subset,
$s\in\shf,f\in\cO_\sB(\Omega)$. For $z\in\Omega$ put $(f\cdot
s)(z):=f(z)\cdot s(z)$. This visibly defines an element $f\cdot
s\in\cF(\Omega)$. The $'\cO_\sB$-linearity' in the conditions
$(2a)$ and $(2b)$ proves $f\cdot s\in \shf$. It follows that
$\shf$ is an $\cO_\sB$-module in the prescribed way.
\end{proof}
\begin{prop}\label{prop-morconst}
The natural map
\[D_r(\Uze)\rightarrow (\shfo)_z,~~~~
\delta\mapsto 1\cotimes\delta\] sheafifies to a morphism of
sheaves of $K$-algebras $\tiD\longrightarrow\shfo.$
\end{prop}
\begin{proof}
This is almost obvious.
\end{proof}

Recall (\ref{equ-alghom}) that we have for any $z\in\sB$ a
canonical $K$-algebra homomorphism
\[
\cO_{\sB,z}\smprod U(\frg)_K\ra \sta\smprod D_r(\Uze).\]
\begin{prop}\label{prop-BBvsdistr}
The homomorphisms (\ref{equ-alghom}) sheafify into a morphism
\[\cO_{\sB}\smprod U(\frg)_K\lra \shfo\]
of sheaves of $K$-algebras. This morphism is $\cO_\sB$-linear.
\end{prop}
\begin{proof}
We view $\cO_{\sB}\smprod U(\frg)_K$ as the sheaf of continuous
sections of its \'etale space

\[\xymatrix{
\dot{\bigcup}_{z\in\sB} \sta\smprod_L U(\frg)_K\ar[d]\\
\sB.}
\]
Composing such a section with (\ref{equ-alghom}) defines a
morphism $i: \cO_{\sB}\smprod U(\frg)_K\lra \cF$ of sheaves of
$K$-algebras and we will prove that its image lies in the subsheaf
$\shf$. To do this let $\Omega\subseteq\sB$ be an open subset and
$s\in\shfu(\Omega)$ a local section.

Let $F\subseteq\sB$ be a facet. Consider the covering of
$\Omega\cap St(F)$ consisting of the single element
$\Omega_0:=\Omega\cap St(F)$. In case $\Omega_0\cap
F\neq\emptyset$ let $s_0$ be the image of $\tilde{s}$ under the
map
\[\cO_{\sB}(\Omega)\otimes_L U(\frg)_K\lra
\cO_{\sB}(\Omega)\cotimes_L D_r(\UFe)\] induced by
$U(\frg)_K\subseteq D_r(\UFe)$. For any $z\in\Omega_0\cap F$ we
obviously have $i(s)(z)=(\iota_z\cotimes {\rm id})(s_0)$ which
shows $(2a)$. For any $z\in \Omega_0$ we find
$i(s)(z)=(\iota_z\cotimes {\rm
id})(s_0)=(\iota_z\cotimes\sigma_r^{FF'})(s_0)$ by the last
statement of (Lem. \ref{lem-glue}). Here $F'$ denotes the facet
containing $z$. This shows $(2b)$. In the light of the definitions
it is clear that the resulting morphism $\cO_{\sB}\smprod
U(\frg)_K\ra \shfo$ is $\cO_\sB$-linear.
\end{proof}

\subsection{Infinitesimal characters.}
In the following we write $\frg_K:=\frg \otimes_{\bbQ_p} K,
\frt_K:=\frt\otimes_{\bbQ_p}K$ etc.

\begin{para} According to \cite[Prop. 3.7]{ST4} the ring $Z(\frg_K)$ lies in the center of the ring $D(G)$.
In the following we fix a central character $$\theta:
Z(\frg_K)\longrightarrow K$$ and we let
$$D(G)_\theta:=D(G)\otimes_{Z(\frg_K),\theta} K$$ be the
corresponding central reduction of $D(G)$. A (left)
$D(G)_\theta$-module $M$ is called {\it coadmissible} if it is
coadmissible as $D(G)$-module via the natural map $D(G)\rightarrow
D(G)_\theta, \delta\mapsto \delta\cotimes 1.$ In the following we
are going to study the abelian category of coadmissible
$D(G)_\theta$-modules. As explained in the beginning this category
is anti-equivalent to the category of admissible locally analytic
$G$-representations over $K$ which have infinitesimal character
$\theta$.

\vskip8pt

Example: Let $\lambda_{0}:D(T)\longrightarrow K$ denote the
character of $D(T)$ induced by the augmentation map $K[T]
\rightarrow K$. The restriction of $\lambda_0$ to the Lie algebra
${\frt_K}\subset D(T)$ vanishes identically whence $\chi=\rho$.
Let $\theta_0: Z(\frg_K)\longrightarrow K$ be the infinitesimal
character associated to $\rho$ via the Harish-Chandra
homomorphism. Then $\ker\theta_0=Z(\frg_K)\cap U(\frg_K)\frg_K$.

\vskip8pt

Remark: In \cite[\S8]{AW} K. Ardakov and S. Wadsley establish a version of
Quillen's lemma for $p$-adically completed universal enveloping
algebras. It implies that any topologically irreducible admissible
locally analytic $G$-representation admits, up to a finite
extension of $K$, a central character and an infinitesimal
character (\cite{DospSchraen}).

\vskip8pt

\begin{para}To investigate the local situation let $F$ be
a facet in $\sB$. We have $$Z(\frg_K)\subseteq D(\UFe)\cap
Z(D(G))\subseteq Z(D(\UFe)),$$ again according to \cite[Prop. 3.7]{ST4}. We let
$D_r(\UFe)_\theta:=D_r(\UFe)\otimes_{Z(\frg)_K,\theta}K$ be the
corresponding central reduction of $D_r(\UFe)$.

\vskip8pt

Let $F,F'$ be two facets in $\sB$ such that
$F'\subseteq\overline{F}$ and consider the homomorphism
$\sigma_r^{F'F}$. According to the last statement of
(Lem. \ref{lem-glue}) it factors by continuity into a homomorphism
\[\sigma_r^{F'F}: D_r(U_{F'}^{(e)})_\theta\longrightarrow
D_r(\UFe)_\theta.\] We may therefore define a sheaf of
$K$-algebras $\tiDt$ in a completely analogous way as the sheaf
$\tiD$ by replacing each $D_r(\Uze)$ and each $D_r(\UFe)$ by their
central reductions. In particular, $(\tiDt)_z=D_r(\Uze)_\theta$
for any $z\in\sB$ and there is an obvious quotient morphism
$$\tiD\longrightarrow \tiDt.$$\end{para}

\subsection{Twisting}
 We now bring in a toral character $$\chi: \frt_K\longrightarrow
K$$ such that $\sigma(\chi)=\theta$. We consider the two-sided
ideals $\cI^{an}_{\sB,\frt}$ and $\cI^{an}_{\sB,\chi}$ of
$\cO_{\sB}\smprod U(\frg_K)$. Denote the right ideal in $\shfo$
generated by the image of the first resp. second under the
morphism
\[\cO_{\sB}\smprod U(\frg_K)\lra \shfo\]
by $\sI^{an}_{\frt}$ resp. $\sI^{an}_{\chi}$.

\begin{lemma}\label{lem-twosided}
The right ideals $\sI^{an}_{\frt}$ and $\sI^{an}_{\chi}$ are
two-sided ideals.
\end{lemma}
\begin{proof}
According to section \ref{sect-equiv} the sheaves $\cO_{\sB}\smprod
U(\frg_K)$ and $\shfo$ have a natural $G$-equivariant structure
such that the morphism $\cO_{\sB}\smprod
U(\frg_K)\rightarrow\shfo$ is equivariant. Moreover, the ideals
$\cI^{an}_{\sB,\frt}$ and $\cI^{an}_{\sB,\chi}$ are $G$-stable.
Hence, so are the right ideals $\sI^{an}_{\frt}$ and
$\sI^{an}_{\chi}$. That these ideals are two-sided can be checked
stalkwise. We give the argument in the case $\sI^{an}_{\frt}$. The
other case follows in the same way. By \cite[Lem. 3.1]{ST4} it
suffices to prove that the product $\delta_g\cdot \partial\in
\cO_{\sB,z}\smprod D_r(\Uze)$ lies in the subspace
$\sI^{an}_{\frt,z}$ for $g\in\Uze$ and
$\partial\in\sI^{an}_{\frt,z}$. Using the power series expansions
for elements of completed distribution algebras (2.2.5) we may
write $\partial$ as an infinite convergent sum
$\partial=\sum_{\alpha\in\bbN_0^d}
f_\alpha\hat{\otimes}\bb^\alpha$ with $f_\alpha\in\cO_{\sB,z}$. By
definition of the skew multiplication
 (\ref{equ-skew}) we have
$$\delta_g\cdot \partial=\sum_{\alpha\in\bbN_0^d}
(g.f_\alpha)\hat{\otimes}\delta_g\bb^\alpha=\sum_{\alpha\in\bbN_0^d}
(g.f_\alpha)\hat{\otimes}(\delta_g\bb^\alpha\delta_g^{-1})\delta_g=\sum_{\alpha\in\bbN_0^d}
g^*(f_\alpha\hat{\otimes}\bb^\alpha)\delta_g=g^*(\partial)\delta_g\in\sI^{an}_{\frt,z}$$
where $g^*:\sI^{an}_{\frt,z}\car\sI^{an}_{\frt,z}$ is induced by
the equivariant structure on the sheaf $\sI^{an}_{\frt}$ (note
that $\Uze$ stabilizes the point $z$).
\end{proof}

By the preceding lemma we may form the quotient sheaves
\[\begin{array}{lcr}
 \sD_{r,\frt}:=\shf/\sI^{an}_{\frt},  & & \sD_{r,\chi}:=\shf/\sI^{an}_{\chi}. \\
\end{array}\]
These are sheaves of (noncommutative) $K$-algebras on $\sB$ and,
at the same time, $\cO_{\sB}$-modules. We have a commutative
diagram of morphisms

\begin{numequation}\label{equ-extension}
\xymatrix{
\cD^{an}_{\sB,\frt}\ar[d] \ar[r] &  \cD^{an}_{\sB,\chi} \ar[d] \\
 \sD_{r,\frt} \ar[r] & \sD_{r,\chi}}
\end{numequation}
with surjective horizontal arrows. Moreover, it follows from
(\ref{equ-centred}) and the preceding lemma that the lower
horizontal arrow induces an isomorphism
\begin{numequation}\label{equ-centred2}\sD_{r,\frt}/(\ker\lambda)\sD_{r,\frt}\car\sD_{r,\chi}.\end{numequation}
We have the following extension of the property 2. in \cite[\S2, Lemme]{BB81}.
\begin{prop}\label{prop-central}
The morphism $\tiD\rightarrow\shfo\rightarrow\sDc$ factors through
$\tiD\rightarrow \tiDt$.
\end{prop}
\begin{proof}
Letting $\sK$ be the kernel of the morphism $\tiD\rightarrow\tiDt$
the claim amounts to $$\sK\subseteq\ker (\tiD\rightarrow\sDc).$$
This can be checked stalkwise, i.e. we are reduced to show that,
for each $z\in\sB$ the natural map $D_r(\Uze)\ra \sta\smprod
D_r(\Uze)/(\sI_\chi)_z$ factores through $D_r(\Uze)_\theta$. The
kernel of the map $D_r(\Uze)\ra D_r(\Uze)_\theta$ is generated by
$$I_\theta:=\ker (U(\frg_K)\ra U(\frg_K)_\theta)$$ and the ideal
$(\sI_\chi)_z$ is generated by the image of
$\cI^{an}_{\sB,\chi,z}$. It therefore suffices to show that the
natural map $U(\frg_K)\ra \sta\smprod U(\frg_K)$ maps $I_\theta$
into $\cI^{an}_{\sB,\chi,z}$. This follows from loc.cit.
\end{proof}
\end{para}
\vskip8pt

\begin{para}
Let us finally make the structure of the stalks of the sheaves
$\sD_{r,\frt}$ and $\sD_{r,\chi}$ at a point $z\in\sB$ more
explicit. According to (Lem. \ref{lem-stalksarefields}) the local
ring $\cO_{\sB,z}$ is a field. For simplicity we put
$\kappa(z):=\cO_{\sB,z}$ and view this as a topological field of
compact type. Note that the Berkovich point $z\in\sB\subset\Xan$
canonically induces a norm topology on $\kappa(z)$ which is weaker
than our topology. We shall not make use of this norm topology in
the following.

\vskip8pt

By loc.cit. together with (5.1.1) we furthermore have
$(\frn^\circ)_{\pi(z)}=\frn_{\pi(z)}$ and
$(\frb^\circ)_{\pi(z)}=\frb_{\pi(z)}$ for the stalks of the
sheaves $\frn^\circ$ and $\frb^\circ$ at $\pi(z)=\eta$ (the
generic point of $X$). Since passage to the stalk is exact, this
proves the following lemma.
\begin{lemma}\label{lem-stalksdis}
Let $z\in\sB$. There is a canonical isomorphism
$$\sD_{r,\frt,z}\car (\kappa(z)\hat{\otimes}_L D_r(\Uze))/
\frn_{\pi(z)}(\kappa(z)\hat{\otimes}_L D_r(\Uze)).$$ This
isomorphism induces a canonical isomorphism between
$\sD_{r,\chi,z}$ and the $\lambda$-coinvariants of the
$\frt_K$-module $(\kappa(z)\hat{\otimes}_L D_r(\Uze))/
\frn_{\pi(z)}(\kappa(z)\hat{\otimes}_L D_r(\Uze))$. In particular,
$$\sD_{r,\rho,z}\car (\kappa(z)\hat{\otimes}_L D_r(\Uze))/
\frb_{\pi(z)}(\kappa(z)\hat{\otimes}_L D_r(\Uze)).$$

\end{lemma}

\end{para}

\section{From representations to sheaves}
We keep all our assumptions from the previous section, i.e. we
assume that $L = \bbQ_{p}$ and that $e> \max(e_{uni},e_{0},
e_{1})$. Our proposed 'localization functor' from representations
to sheaves associated to the pair $\sigma(\chi)=\theta$ will be a
functor
\[\mathscr{L}_{r,\chi}: M\mapsto \sDc\otimes_{\tiDt}\tiM\] from
(coadmissible) left $D(G)_\theta$-modules $M$ to left
$\sDc$-modules satisfying additional properties. Here $\tiM$ is a
constructible sheaf replacing the constant sheaf $\underline{M}$
appearing in the Beilinson-Bernstein construction (Thm.
\ref{thm-BB}). It is a modest generalization of the sheaf $\tiD$
as follows.

\subsection{A constructible sheaf of modules}
Suppose we are given any (left) $D(G)$-module $M$. Let
$F\subseteq\sB$ be a facet. We may regard $M$ as a
$D(\UFe)$-module via the natural map $D(\UFe)\ra D(G)$. We put
$$M_r(\UFe):=D_r(\UFe)\otimes_{D(\UFe)} M,$$
a (left) $D_r(\UFe)$-module. If $F'\subseteq\sB$ is another facet
such that $F'\subset\overline{F}$ the map
$$\sigma_r^{F'F}\otimes {\rm id}: M_r(U_{F'}^{(e)})\lra M_r(\UFe),~~~~\delta\otimes m\mapsto \sigma_r^{F'F}(\delta)\otimes m$$
is a module homomorphism relative to $\sigma_r^{F'F}$ and inherits
the homomorphic properties from $\sigma_r^{F'F}$ (cf. Lem.
\ref{lem-glue}). Again, we may define a sheaf of $K$-vector spaces
$\tiM$ on $\sB$ in a completely analogous way as the sheaf $\tiD$
by replacing each $D_r(\UFe)$ and each $D_r(\Uze)$ by $M_r(\UFe)$
and $M_r(\Uze)$ respectively. In particular, $\tiM$ restricted to
a facet $F$ is the constant sheaf with value $M_r(\UFe)$ and
therefore $\tiM$ is a constructible sheaf. If $s\in D_r(\Uze),
m\in M_r(\Uze)$ the `pointwise multiplication' $(s\cdot
m)(z):=s(z)m(z)$ makes $\tiM$ a $\tiD$-module.

\vskip8pt


\begin{lemma}
If $M$ is a $D(G)_\theta$-module then $\tiM$ is a $\tiDt$-module
via the morphism $\tiD\ra\tiDt$.
\end{lemma}
\begin{proof}
This is almost obvious.
\end{proof}

\subsection{A localization functor}
As usual $\sDc\otimes_{\tiDt}\tiM$ denotes the sheaf associated to
the presheaf $V\mapsto \sDc(V)\otimes_{\tiDt(V)}\tiM(V)$ on $\sB$.
The construction $M\mapsto \tiM$ is functorial in $M$ and commutes
with arbitrary direct sums. Thus the correspondance
 \[\mathscr{L}_{r,\chi}: M\mapsto \sDc\otimes_{\tiDt}\tiM\] is a covariant functor from (left) $D(G)_\theta$-modules to (left)
 $\sDc$-modules. It commutes with arbitary direct sums.
We call it tentatively a {\it localization functor} associated to
$\chi$.
\end{para}
We emphasize that the functor $\mathscr{L}_{r,\chi}$ depends on
the choice of the level $e$. As we did before we suppress this
dependence in the notation. As a second remark, let $\mathcal{M}$
be an arbitrary $\sDc$-module and $f:
\mathscr{L}_{r,\chi}(M)\rightarrow\mathcal{M}$ a morphism. The
composite
$$M\rightarrow\Gamma(\sB,\tiM)\rightarrow
\Gamma(\sB,\mathscr{L}_{r,\chi}(M))\stackrel{f}{\rightarrow}\Gamma(\sB,\mathcal{M})$$
is a $K$-linear map. We therefore have a natural transformation of
functors
$${\rm Hom}_{\sDc}(\mathscr{L}_{r,\chi}(\cdot),.)\longrightarrow
{\rm Hom}_K(\cdot,\Gamma(\sB,\cdot)).$$ Generally, it is far from
being an equivalence.

\vskip8pt

We compute the stalks of the localization
$\mathscr{L}_{r,\chi}(M)$ for a {\it coadmissible} module $M$. In
this case $(\tiM)_z$ is finitely generated over the Banach algebra
$D_r(\Uze)_\theta$ and therefore has a unique structure as a
Banach module over $D_r(\Uze)_\theta$. Let $z\in\sB\subset \Xan$
with residue field $\kappa(z)$. Recall that $\pi(z)$ equals the generic point of $X$.
\begin{prop}\label{prop-stalks}
Let $M$ be a coadmissible left $D(G)_\theta$-module and let
$z\in\sB$. The morphism $\tiM\rightarrow\mathscr{L}_{r,\chi}(M)$
induces an isomorphism between the $\lambda$-coinvariants of the
$\frt_K$-module
$$(\kappa(z)\hat{\otimes}_L (\tiM)_z)/
\frn_{\pi(z)}(\kappa(z)\hat{\otimes}_L (\tiM)_z)$$ and the stalk
$\mathscr{L}_{r,\chi}(M)_z$. In particular, if $\theta=\theta_0$
we have
$$
(\kappa(z)\hat{\otimes}_L (\tiM)_z)/
\frb_{\pi(z)}(\kappa(z)\hat{\otimes}_L
(\tiM)_z)\car\mathscr{L}_{r,\rho}(M)_z.$$
\end{prop}
\begin{proof}
Let $N$ be an arbitrary finitely generated $D_r(\Uze)$-module.
According to (Lem. \ref{lem-stalksdis}) the space
$\sD_{r,\frt,z}\otimes_{D_r(\Uze)}N$ may be written as
$$(((\kappa(z)\hat{\otimes}_L
D_r(\Uze))/\frn_{\pi(z)}(\kappa(z))\hat{\otimes}_L
D_r(\Uze))\otimes_{\kappa(z)\hat{\otimes}_L D_r(\Uze)}
\kappa(z)\hat{\otimes}_L D_r(\Uze))\otimes_{D_r(\Uze)} N.$$ Since
$N$ is a complete Banach module this may be identified with
$(\kappa(z)\hat{\otimes}_L
N)/\frn_{\pi(z)}(\kappa(z)\hat{\otimes}_L N)$ by associativity of
the completed tensor product. The resulting isomorphism
$$(\kappa(z)\hat{\otimes}_L
N)/\frn_{\pi(z)}(\kappa(z)\hat{\otimes}_L
N)\car\sD_{r,\frt,z}\otimes_{D_r(\Uze)}N$$ is functorial in $N$.
According to the second part of loc.cit. we obtain a functorial
homomorphism
$$(\lambda-{\rm coinvariants~of~}(\kappa(z)\hat{\otimes}_L
N)/\frn_{\pi(z)}(\kappa(z)\hat{\otimes}_L
N))\longrightarrow\sD_{r,\chi,z}\otimes_{D_r(\Uze)}N$$ which is an
isomorphism in the case $N=D_r(\Uze)$. Since source and target are
right exact functors in $N$ commuting with finite direct sums we
may use a finite free presentation of $N$ to obtain that it is an
isomorphism in general. The assertion of the proposition follows
by taking $N=(\tiM)_z$.
\end{proof}
\begin{cor}
Let $\chi$ be dominant and regular. The functor
$\mathscr{L}_{r,\chi}$, restricted to coadmissible modules, is
exact.
\end{cor}
\begin{proof}
Exactness can be checked at a point $z\in\sB$ where the functor in
question equals the composite of three functors. The first functor
equals $N\mapsto D_r(\Uze)_\theta \otimes_{D(\Uze)_\theta} N$ on
the category of coadmissible $D(\Uze)_\theta$-modules. It is exact
by \cite[Rem. 3.2]{ST5}. The second functor equals $N\mapsto
\kappa(z)\hat{\otimes}_L N$ on the category of finitely generated
$D_r(\Uze)_\theta$-modules. It is exact by \cite[Prop.
1.1.26]{EmertonA}. The last functor equals the Beilinson-Bernstein stalk
functor at $z$ according to (Thm. \ref{thm-BB} (iii)). By loc.cit., part (ii), this functor is exact if $\chi$ is dominant and regular. \end{proof}

\begin{lemma}\label{lem-zerodim}
Let $z\in\sB$. If $N$ is a finitely generated
$D_r(\Uze)_\theta$-module which is finite dimensional over $K$,
then the natural map
$$\cD^{an}_{\sB,\chi,z}\otimes_{U(\frg_K)_\theta} N\car\sDcz\otimes_{D_r(\Uze)_\theta} N
$$ is an
isomorphism which is functorial in such $N$.
\end{lemma}
\begin{proof}
We adopt the notation of (Prop. \ref{prop-skewexist}) and write \[
\varinjlim_V\;(\cA_V\smprod D_r(\Uze))\car \sta\smprod
D_r(\Uze),\] an isomorphism of topological $K$-algebras. By
\cite[Prop. 2.1]{ST5} the finitely generated module $(\cA_V\smprod
D_r(\Uze))\otimes_{D_r(\Uze)} N$ has a unique Banach topology. We
therefore have canonical $\cA_V$-linear isomorphisms
$$(\cA_V\hat{\otimes}_L D_r(\Uze))\otimes_{D_r(\Uze)}
N\simeq \cA_V\hat{\otimes}_L N=\cA_V\otimes_L N.$$ Passage to the
inductive limit yields the $\cO_{\sB,z}$-linear map
$$(\cO_{\sB,z}\hat{\otimes}D_r(\Uze))\otimes_{D_r(\Uze)}
N\simeq \cO_{\sB,z}\otimes_L N=(\cO_{\sB,z}\smprod
U(\frg_K))\otimes_{U(\frg_K)} N.$$ The target maps canonically to
$\cD^{an}_{\sB,\chi,z}\otimes_{U(\frg_K)_\theta} N$ and the
composed map annihilates all elements of the form
$\xi\hat{\otimes}n$ with $n\in N$ and
$\xi\in\cI^{an}_{\sB,\chi,z}$. Since such $\xi$ generate
$\sI^{an}_{\chi,z}$ the composed map factores therefore into a map
$$\sD_{r,\chi,z}\otimes_{D_r(\Uze)_\theta} N\rightarrow
\cD^{an}_{\sB,\chi,z}\otimes_{U(\frg_K)_\theta}N.$$ This gives the
required inverse map.
\end{proof}

\begin{cor}\label{cor-zerodim} Let $M$ be a left $D(G)_\theta$-module such that
${\rm dim}_K M_r(\Uze)<\infty$ for all $z\in\sB$. The natural
morphism of sheaves
$$\cD^{an}_{\sB,\chi}\otimes_{U(\frg_K)_\theta}\tiM\car
\sDc\otimes_{\tiDt}\tiM=\mathscr{L}_{r,\chi}(M)$$ induced from
(\ref{equ-extension}) is an isomorphism.
\end{cor}
\begin{proof}
Let $z\in\sB$. Applying the preceding lemma to $N:=M_r(\Uze)$ we
see that the morphism is an isomorphism at the point $z$. This
proves the claim.
\end{proof}

\subsection{Equivariance}\label{sect-equiv}

In this section, we keep the same assumptions as before, i.e. we
assume $L = \bbQ_{p}$, $e >  \max(e_{uni},e_{0}, e_{1})$ and $r
\in
[r_{0}, 1)$. 

\begin{para}Consider for a moment an arbitrary ringed space
$(Y,\mathcal{A})$ where $\mathcal{A}$ is a sheaf of (not
necessarily commutative) $K$-algebras on $Y$. Let $\Gamma$ be an
abstract group acting (from the right) on $(Y,\mathcal{A})$. In
other words, for every $g,h\in\Gamma$ and every open subset
$U\subseteq Y$ there is an isomorphism of $K$-algebras $g^*:
\mathcal{A}(U)\car\mathcal{A}(g^{-1}U)$ compatible in an obvious
sense with restriction maps and satisfying $(gh)^*=h^* g^*$.

A {\it $\Gamma$-equivariant} $\mathcal{A}$-module (cf.
\cite[II.F.5]{Jantzen}) is a (left) $\cA$-module $\cM$ equipped,
for any open subset $U\subseteq Y$ and for $g\in G$, with
$K$-linear isomorphisms $g^*:\cM(U)\car\cM(g^{-1}U)$  compatible
with restriction maps and such that $g^*(am)=g^*(a)g^*(m)$ for
$a\in\cA(U),m\in\cM(U)$. If $g,h\in G$ we require $(gh)^*=h^*g^*$.
An obvious example is $\cM=\cA$. If $\cM$ is equivariant we have a
$K$-linear isomorphism $\cM_z\car \cM_{g^{-1}z}$ between the
stalks of the sheaf $\cM$ at $z$ and $g^{-1}z$ for any $g\in G$.

Finally, a morphism of equivariant modules is a $\cA$-linear map
compatible with the $\Gamma$-actions. The equivariant modules form
an abelian category.\end{para}

\vskip8pt

\begin{para}
After these preliminaries we go back to the situation
discussed in the previous section. We keep all the assumptions
from this section.  The group $G$ naturally acts on the ringed
space $(\Xan,\cO_\Xan)$. Moreover, $G$ acts on $\frg$ and
$U(\frg)$ via the adjoint action as usual. It follows from the
classical argument (\cite[\S3]{Milicic93}) that the sheaves
$$\cO_\Xan\smprod U(\frg), \cI^{an}_\chi {\rm ~and~}
\cD^{an}_\chi:=(\cO_\Xan\smprod U(\frg))/\cI^{an}_\chi$$ (as
defined in section \ref{sec-Berkovich}) are equivariant
$\cO_\Xan$-modules. Of course, here $g^*:
\cD^{an}_\chi(U)\car\cD^{an}_\chi(g^{-1}U)$ is even a $K$-algebra
isomorphism for all $g\in G$ and open subsets $U\subseteq \Xan$.

\vskip8pt

On the other hand, the group $G$ acts on the ringed space
$(\sB,\cO_\sB)$ and the natural map
$\vartheta_\bB:\sB\longrightarrow\Xan$ is $G$-equivariant (Thm. \ref{thm-RTW}). Since our functor $\sF$ preserves $G$-equivariance
the $\cO_\sB$-modules
$$\cO_\sB\smprod U(\frg_K),\cI^{an}_{\sB,\chi} \hskip14pt {\rm and} \hskip14pt
\cD^{an}_{\sB,\chi}=(\cO_\sB\smprod
U(\frg_K))/\cI^{an}_{\sB,\chi}$$ are $G$-equivariant. Again, here
$g^*: \cD^{an}_{\sB,\chi}(U)\car\cD^{an}_{\sB,\chi}(g^{-1}U)$ is a
$K$-algebra isomorphism for all $g\in G$ and open subsets
$U\subseteq\sB$. Recall (\ref{def-sheaf}) the sheaf of
$K$-algebras $\shfo$.
\begin{prop}
The $\cO_\sB$-module $\shfo$ is $G$-equivariant. For $g\in G$ the
map $g^*$ is a $K$-algebra isomorphism.
\end{prop}
\begin{proof}
Given $g\in G$ and $z\in\sB$ we have the group isomorphism
$g^{-1}(\cdot)g: \Uze\car U_{g^{-1}z}^{(e)}$ by
(\ref{equ-conjgroups}). Since it is compatible with variation of
the level $e$ it is compatible with the $p$-valuations
$\mathring{\omega}_z$ and $\mathring{\omega}_{g^{-1}z}$. It
induces therefore an isometric isomorphism of Banach algebras
\begin{numequation}\label{equ-equiDr}g^{-1}(\cdot)g: D_r(\Uze)\car D_r(U_{g^{-1}z}^{(e)}).\end{numequation}The
induced map
$$\cO_{\sB,z}\hat{\otimes}_L D_r(\Uze)\car
\cO_{\sB,z}\hat{\otimes}_L D_r(U_{g^{-1}z}^{(e)})$$ is
multiplicative with respect to the skew multiplication and we
obtain an isomorphism of topological $K$-algebras
\begin{numequation}\label{equ-equivar}g^*: \shf_z\car \shf_{g^{-1}z}\end{numequation} according to (Lem. \ref{lem-stalkI}).
Since we have the identity $g\frx g^{-1}={\rm Ad}(g)(\frx)$ in
$D(G)$ this isomorphism fits into the commutative diagram
\[\xymatrix{
(\cO_{\sB}\smprod U(\frg))_z\ar[d] \ar[r]^>>>>>\simeq &
(\cO_{\sB}\smprod
U(\frg))_{g^{-1}z} \ar[d] \\
 \shf_z \ar[r]^>>>>>>\simeq & \shf_{g^{-1}z}}
\]
where the vertical arrows are the inclusions from
(\ref{equ-alghom}). Recall the sheaf $\cF$ appearing in (Lem.
\ref{lem-sheafI}). Let $\Omega\subseteq\sB$ be an open subset. The
isomorphisms (\ref{equ-equivar}) for $z\in\Omega$ assemble to a
$K$-algebra isomorphism
$$g^*:\cF(\Omega)\car\cF(g^{-1}\Omega), s\mapsto [z\mapsto (g^*)^{-1}(s(gz))]$$ compatible with restriction
maps and satisfying $(gh)^*=h^*g^*$ for $g,h\in G$. It now
suffices to see that $g^*$ maps the subspace $\shf(\Omega)$ into
$\shf(g^{-1}\Omega)$. Let $s\in\shf(\Omega)$. If $F$ is a facet in
$\sB$ we let $\Omega=\cup_{i\in I} \Omega_i$ be a datum for $s$
with respect to $F$. If $F\cap\Omega_i\neq\emptyset$ we consider
$g^{-1}V_i$ and $(g^*)^{-1}(s_i)$ and obtain a datum
$g^{-1}\Omega=\cup_{i\in I} g^{-1}\Omega_i$ for the section
$(g^*)^{-1}sg\in \cF(g^{-1}\Omega)$ with respect to the facet
$g^{-1}F$. Indeed, the axiom {\rm (2a)} for the
section$(g^*)^{-1}sg$ follows from the commutativity of the
diagram

\[\xymatrix{
\cO_\sB(U)\ar[d]^{(g^*)^{-1}} \ar[r]^>>>>>{\iota_z} &  \cO_{\sB,z}\ar[d]^{(g^*)^{-1}} \\
 \cO_\sB(gU) \ar[r]^>>>>>{\iota_{gz}} & \cO_{\sB,gz} }
\]
valid for any open subset $U\subseteq\sB$ containing $z$.
Moreover, we have a commutative diagram

\[\xymatrix{
D_r(U_{F'}^{(e)}) \ar[d]^{(g^*)^{-1}} \ar[r]^>>>>{\sigma_r^{F'F}} &  D_r(\UFe)\ar[d]^{(g^*)^{-1}} \\
 D_r(U_{gF'}^{(e)}) \ar[r]^>>>>{\sigma_r^{gF'gF}} & D_r(U_{gF}^{(e)}) }
\]
whenever $F',F$ are two facets in $\sB$ with
$F'\subseteq\overline{F}$. From this the axiom {\rm (2b)} for the
section $(g^*)^{-1}sg$ follows easily.
\end{proof}
It follows from the preceding proof that the morphism
$\cO_{\sB}\smprod U(\frg)\rightarrow \shfo$ from (Prop.
\ref{prop-BBvsdistr}) is equivariant. The equivariant structure of
$\cI^{an}_{\sB,\chi}$ therefore implies that the ideal sheaf
$\sI^{an}_{\chi}$ of $\shfo$ is naturally equivariant. This yields
the following corollary.
\begin{cor} The $\cO_\sB$-module $\sDc$ is
equivariant. The map $g^*$ is a $K$-algebra isomorphism for any
$g\in G$. The morphism $\cD^{an}_{\sB,\chi}\ra\sD_{r,\chi}$ from
(\ref{equ-extension}) is equivariant.
\end{cor}
\end{para}

\vskip8pt

The above discussion shows that there is a natural right action of
$G$ on the ringed space $(\sB,\sDc)$. We let ${\rm Mod}_G(\sDc)$
be the abelian category of $G$-equivariant (left) $\sDc$-modules.

\begin{para}Using very similar arguments we may use the isomorphisms
(\ref{equ-equiDr}) appearing in the above proof to define an
equivariant structure on the sheaves $\tiD$ and $\tiDt$. As before
we suppose $\sigma(\chi)=\theta$. If $M$ is a $D(G)$-module (resp.
$D(G)_\theta$-module) with $m\in M$ and $g\in G$ we put
$g.m:=\delta_{g^{-1}}m$. This defines a $K$-linear isomorphism
$$g^*:M_r(\Uze)\car M_r(U_{g^{-1}z}^{(e)})$$
via $g^*(\delta\otimes m):=g^*(\delta)\otimes gm$ for any
$\delta\in D_r(\Uze)$. As in the case of $\tiD$ these isomorphisms
lift to an equivariant structure on the sheaf $\tiM$. Since these
isomorphisms are compatible with the isomorphisms
(\ref{equ-equiDr}) we obtain that $\tiM$ is an equivariant
$\tiD$-module (resp. $\tiDt$-module). We now define
$g^*(\partial\otimes m):=g^*(\partial)\otimes g^*(m)$ for local
sections $\partial$ and $m$ of $\sDc$ and $\tiM$ respectively.
Since the morphism $\tiDt\rightarrow\sDc$ induced by (Prop.
\ref{prop-morconst}) is equivariant this yields an equivariant
structure on $\func(M)$. If $M\rightarrow N$ is a
$D(G)_\theta$-linear map the resulting morphism
$\func(M)\rightarrow\func(N)$ is easily seen to be equivariant.
This shows
\end{para}
\begin{cor} The functor $\func$ takes values in ${\rm
Mod}_G(\sDc)$.
\end{cor}

\section{Comparison with the Schneider-Stuhler
construction}\label{compSS}

In this section, we keep the same assumptions as before, i.e. we
assume $L = \bbQ_{p}$, $e >  \max(e_{uni},e_{0}, e_{1})$ and $r
\in [r_{0}, 1)$. We will work in this section with the trivial
infinitesimal character, i.e. $\lambda:=\lambda_0$ with
$r(\lambda_0)=r_0$ and $\theta:=\theta_0$. 
\subsection{Preliminaries on smooth distributions}

\begin{para}
Let $M$ be a co-admissible $D(G)$-module such the the associated
locally analytic representation $V=M'_b$ is smooth. In the
previous section, we have associated to $M$ a sheaf $\tiM$ on the
Bruhat-Tits building $\sB$. On the other hand, we also have the
sheaf $\tiV$ on $\sB$ constructed by Schneider and Stuhler (cf.
4.6). We now show that for $r<p^{\frac{-1}{p-1}}$, the two sheaves
$\check{\tiV}$ and $\tiM$ are canonically isomorphic. Here,
$\check{V}$ denotes the smooth dual. We remark straightaway that
$V$ is admissible-smooth (\cite[Thm. 6.6]{ST5}) and hence, so is
$\check{V}$ (\cite[1.5 (c)]{Cartier}).

\vskip8pt

Suppose $H$ is a uniform locally $\bbQ_p$-analytic group with
$\bbQ_p$-Lie algebra $\frh$. Let $D^\infty(H)$ denote the quotient
of $D(H)$ by the ideal generated by $\frh$. Let $\cC^\infty_H$
denote the category of coadmissible $D^\infty(H)$-modules. If
$U_r(\frh)$ denotes the closure of $U(\frh)$ inside $D_r(H)$ we
put
$$H_{(r)}:=H\cap U_r(\frh).$$\end{para}

\begin{lemma}
The set $H_{(r)}$ is an open normal subgroup of $H$ constituting,
for $r\uparrow 1$, a neighbourhood basis of $1\in H$.
\end{lemma}
\begin{proof}
As the norm $||.||_r$ on $D_r(H)$ does not depend on the choice of
ordered basis the inversion map $h\mapsto h^{-1}$ induces an
automorphism of $D_r(H)$. It induces an automorphism of
$U_r(\frh)$ which implies that $H_{(r)}$ is a subgroup of $H$. A
similar argument with the conjugation automorphism $h\mapsto
ghg^{-1}$ for a $g\in H$ implies that this subgroup is normal in
$H$. For the remaining assertions we choose $m\geq 0$ such that
$r_m=\sqrt[p^m]{r_0} \geq r$ and consider $D(P_{m+1}(H))$. The
inclusion $D(P_{m+1}(H))\subseteq D(H)$ gives rise to an isometric
embedding $$D_{r_0}(P_{m+1}(H))\hookrightarrow D_{r_m}(H)$$ (final
remark in 2.2.3). Since $U(\frh)$ is norm-dense inside
$D_{r_0}(P_{m+1}(H))$ it follows that $$P_{m+1}(H)\subset
U_{r_m}(\frh)\subseteq U_r(\frh)$$ which implies
$P_{m+1}(H)\subseteq H_{(r)}$ and therefore $H_{(r)}$ is open.
Finally, if $r\uparrow 1$ then $r_m\uparrow 1$ whence
$m\uparrow\infty$. Since the lower $p$-series $\{P_{m}(H)\}_m$
constitutes a neighbourhood basis of $1\in H$ the last assertion
of the lemma follows.
\end{proof}

The lemma implies (cf. \cite[pf. of Thm. 6.6]{ST5}) a canonical
$K$-algebra isomorphism $D^\infty(H)\simeq\varprojlim_r
K[H/H_{(r)}]$ coming from restricting distributions to the
subspace of $K$-valued locally constant functions on $H$.
\begin{prop}\label{prop-smoothcoinvariants}
(i) We have $D_r(H)\otimes_{D(H)} D^\infty(H)\simeq K[H/H_{(r)}]$
as right $D^\infty(H)$-modules;

(ii) If $M\in\cC^\infty_H$ and $V=M'_b$ denotes the corresponding
smooth representation then $D_r(H)\otimes_{D(H)}M\simeq
(\check{V})_{H_{(r)}}$ as $K$-vector spaces. Here,
$(\cdot)_{H_{(r)}}$ denotes $H_{(r)}$-coinvariants and
$\check{(\cdot)}$ denotes the smooth dual.
\end{prop}
\begin{proof}
The first statement follows from $D_r(H)=\oplus_{h\in H/H_{(r)}}
\;\delta_h U_r(\frh)$ as right $U_r(\frh)$-modules by passing to
quotients modulo the ideals generated by $\frh$. The second
statement follows from (i) by observing the general identities
$K[H/N]\otimes_{D^\infty(H)}M=\Hom_K(V^N,K)=(\check{V})_N$ valid
for any normal open subgroup $N$ of $H$.
\end{proof}


\begin{cor}
If $M\in\cC^\infty_H$ and $r_0\leq r<p^{-1/p-1}$ then
$D_r(H)\otimes_{D(H)}M\simeq (\check{V})_{H}.$
\end{cor}
\begin{proof}
We have $U_r(\frh)=D_r(H)$ for such an $r$ and therefore
$H_{(r)}=H$.
\end{proof}

\subsection{The comparison isomorphism}
\begin{para} Let us return back to our sheaf $M\mapsto\tiM$. We {\bf
assume} in the following $$r_0\leq r<p^{-1/p-1}.$$ Let $F$ be a
facet in $X$. If we apply the above corollary to the uniform group
$\UFe$ we obtain a canonical linear isomorphism $$f_r^F:
M(U_{F}^{(e)})=D_r(\UFe)\otimes_{D(\UFe)}M \car
(\check{V})_{\UFe}.$$ If $F\subseteq\overline{F'}$ for two facets
$F,F'$ in $X$ it follows that
\begin{equation}\label{compatible} f_r^{F'}\circ\sigma_r^{FF'}=pr^{FF'}\circ f_r^F\end{equation} where
$pr^{FF'}:(\check{V})_{\UFe}\rightarrow
(\check{V})_{U_{F'}^{(e)}}$ denotes the natural
projection.\end{para}

\begin{prop}\label{prop-ScSt1}Given an open subset $\Omega\subseteq X$ the collection of
maps $f_r^z$ for $z\in\Omega$ induces a $K$-linear isomorphism
$\tiM(\Omega)\simeq \check{\tiV}(\Omega)$ compatible with
restriction maps whence a canonical isomorphism of sheaves
$$\tiM\car \check{\tiV}$$
which is natural in admissible $V$.
\end{prop}
\begin{proof}
Given $z\in\sB$ we have the isomorphism
$f_r^z:M_r(\Uze)\car(\check{V})_{\Uze}$ as explained above. These
maps assemble to a $K$-linear isomorphism, say $f_r^\Omega$,
between the space of maps
$$s:\Omega\rightarrow\dot{\bigcup}_{z\in\Omega} M_r(\Uze)$$ such
that $s(z)\in M_r(\Uze)$ for all $z\in\sB$ and the space of maps
$$s:\Omega\rightarrow\dot{\bigcup}_{z\in\Omega} (\check{V})_{\Uze}$$
such that $s(z)\in(\check{V})_{\Uze}$ for all $z\in\sB$. It is
clearly compatible with restriction. It therefore suffices to show
that it descends to an isomorphism between the subspaces
$\tiM(\Omega)$ and $\check{\tiV}(\Omega)$ respectively. Since
$\tiM$ and $\check{\tiV}$ are sheaves it suffices to verify this
over the open sets $\Omega\cap St(F)$ for facets $F\subset\sB$. We
may therefore fix a facet $F\subset\sB$ and assume that
$\Omega\subseteq St(F)$. Restricting to members $\Omega_i$ with
$\Omega_i\cap F\neq\emptyset$ of a datum for $s$ with respect to
$F$ and using the sheaf property a second time we may assume that
the covering $\{\Omega\}$ of $\Omega=\Omega\cap St(F)$ is a datum
for $s$ with respect to $F$ satisfying $\Omega\cap
F\neq\emptyset$. Let $s\in M_r(\UFe)$ be the corresponding element
of the datum. We let $\check{v}$ be any preimage in $\check{V}$ of
$f_r^F(s)\in(\check{V})_{\UFe}$. The value of the function
$f_r^\Omega(s)$ at $z\in\Omega$ is then given by
$$f_r^z(s(z))=f_r^{F'}(\sigma_r^{FF'}(s))\stackrel{(\ref{compatible})}{=}pr^{FF'}(f_r^F(s))={\rm
class~ of~}\check{v}\in (\check{V})_{U_{F'}^{(e)}}$$ where $F'\in
St(F)$ is the unique open facet containing $z$. This means
$f_r^\Omega(s)\in \check{\tiV}(\Omega)$.

Conversely, let $\check{s}\in\check{\tiV}(\Omega)$ and consider
$s:=(f_r^\Omega)^{-1}(\check{s})$. Let $F\subset\sB$ be a facet.
Any defining open covering $\Omega=\cup_{i\in I}\Omega_i$ with
vectors $\check{v}_i\in\check{V}$ for the section $\check{s}$
induces an open covering $\Omega\cap St(F)=\cup_{i\in
I}\Omega_{i,F}$ where $\Omega_{i,F}:=\Omega_i\cap St(F)$. If
$F\cap\Omega_{i,F}\neq\emptyset$ we let $s_i\in M_r(\UFe)$ be the
inverse image of the class~of~$\check{v}_i$ under $(f_r^F)^{-1}$.
We claim that this gives a datum for $s$ with respect to $F$.
Indeed, for any $z\in \Omega_{i,F}\cap F$ we compute
$$ s(z)=(f_r^z)^{-1}(\check{s}(z))=(f_r^z)^{-1}({\rm
class~of~}\check{v}_i)=s_i$$ which settles the axiom $(2a)$ for
$s$. Similarly, for any $z'\in\Omega_{i,F}$ the value of $s(z')$
equals

\[\begin{array}{rrl}
 (f_r^{z'})^{-1}(\check{s}(z'))=(f_r^{F'})^{-1}({\rm
class~of~}\check{v}_i)= & (f_r^{F'})^{-1}( pr^{FF'}(\check{v}_i))
\stackrel{(\ref{compatible})}{=}&\sigma_r^{FF'}(
(f_r^F)^{-1}(\check{v}_i))   \\
  &&  \\
 & =&\sigma_r^{FF'}(s_i) \\
\end{array}\]
where $F'$ denotes the unique open facet of $St(F)$ containing
$z'$. This proves (2b) for $s$. All in all $s\in\tiM(\Omega)$.
This proves the proposition.
\end{proof}

\vskip8pt

\begin{lemma}
Let $M$ be a coadmissible $D^{\infty}(G)$-module. Then $M$ is a
$D(G)_{\theta_{0}}$-module.
\end{lemma}
\begin{proof}
We have to show that the canonical map $D(G) \rightarrow
D^{\infty}(G)$ factors through $D(G)_{\theta_{0}}$. The kernel of
$D(G) \rightarrow D^{\infty}(G)$ is the two sided ideal generated
by $\frg$. The intersection of this latter ideal with $Z(\frg_K)$
equals $\ker\theta_0$ (cf. example 8.1.1). It follows that the map
$Z(\frg_K) \rightarrow D^{\infty}(G)$ factors through
$\theta_{0}$.
\end{proof}

\begin{thm}\label{thm-compSchSt}Let as above $r_0\leq r<p^{-1/p-1}.$
Suppose $M$ is a coadmissible $D^{\infty}(G)$-module. Then there
is a canonical isomorphism of $\cO_{\sB}$-modules
$$
C^{SS}: \cO_{\sB} \otimes_L \check{\tiV} \car \mathscr{L}_{r,
\rho}(M)
$$
which is natural in such $M$. Here, as above, $V = M'_{b}$.
\end{thm}
\begin{proof}
Since $\frg M=0$ there is a canonical isomorphism
$$\cO_{\sB}\otimes_L\tiM\car\cD^{an}_{\sB,\chi}\otimes_{U(\frg_K)_\theta}\tiM.$$
Hence the assertion is a combination of (Lem. \ref{lem-zerodim}) and
(Prop. \ref{prop-ScSt1}).
\end{proof}


\section{Compatibility with the Beilinson-Bernstein
localization}\label{compBB}

In this section, we invoke our usual assumptions, i.e. we assume
$L = \bbQ_{p}$, $e >  \max(e_{uni},e_{0}, e_{1})$ and $r \in
[r_{0}, 1)$.

\vskip8pt

Let $V$ denote a finite dimensional algebraic representation of
$\bG$. Then $V$ gives rise to a $U(\frg)$-module. Let $M = V'$
denote the dual of $V$. It is a coadmissible $D(G)$-module.
Suppose the $U(\frg_K)$-module underlying $M$ is a
$U(\frg_K)_{\theta}$-module.

\vskip8pt

Recall that to any $U(\frg_K)_{\theta}$-module $M$,
Beilinson-Bernstein associate a $\cD_{\chi}$-module which will be
denoted $\Delta(M)$ (cf. \S5). We can pull this back under the
natural map $\pi: X^{an} \rightarrow X$ to get a
$\cD_{\chi}^{an}$-module $\Delta(M)^{an}$. Finally, we may apply
the functor $\sF$ to this module. Denote the latter
$\cO_{\sB}$-module by $\Delta(M)_{\sB}^{an}$. One has the
following description of $\Delta(M)^{an}$ and
$\Delta(M)_{\sB}^{an}$:

\vskip8pt

\begin{center}

$\Delta(M)^{an}  = \cD_{\chi}^{an} \otimes_{U(\frg_K)_{\theta}} M  $  \\

\vskip8pt

$\Delta(M)^{an}_{\sB} = \cD_{\sB, \chi}^{an} \otimes_{U(\frg_K)_{\theta}} M. $\\

\end{center}
Here, the second identity follows from the compatibility between
tensor products with restriction functors
(\cite[Prop. 2.3.5]{KashiwaraSchapira}).

On the other hand, any finite dimensional algebraic representation
$V$ gives rise to a $D(G)$-module $M$, where $M = V'$. If $V$ is a
$U(\frg_K)_{\theta}$-module, then $M$ is a $D(G)_{\theta}$-module.
In particular, the results of section 8 allow us to associate to
$M$ the $\sDc$-module $\mathscr{L}_{r,\chi}(M)$. Recall, this
module is given by:
$$
\mathscr{L}_{r, \chi}(M) =  \sDc\otimes_{\tiDt} \tiM
$$
Now the canonical morphism $\cD_{\sB, \chi}^{an} \rightarrow \sDc$
induces a morphism $$C^{BB} : \cD_{\sB, \chi}^{an}
\otimes_{U(\frg_K)_{\theta}} M \longrightarrow \sDc\otimes_{\tiDt}
\tiM.$$

\setcounter{para}{0}
\begin{thm}\label{thm-compBB}
There is $r(M)\in [r_0,1)$ such that for all $r\geq r(M)$ the
canonical morphism
$$C^{BB}: \Delta(M)^{an}_{\sB}\car\mathscr{L}_{r, \chi}(M)$$ is an
isomorphism of $\cD_{\sB,\chi}^{an}$-modules.
\end{thm}
\begin{proof}
Let $F$ be a facet in $\sB$ such that $F\subseteq \overline{\sC}$.
By \cite[Prop. 4.2.10]{EmertonA} the $D(\UFe)$-module $M$
decomposes into a {\it finite} direct sum of irreducible
$D(\UFe)$-modules $M_i$. Since all $M_i$ are coadmissible
$D(\UFe)$-modules there exists $r(F)\in [r_0,1)$ such that
$$M_{i,r}:=D_r(\UFe)\otimes_{D(\UFe)}M_i \neq 0$$ for all $r\geq
r(F)$ and all $i$. By Theorem A (\cite[\S3]{ST4}) the
$D(\UFe)$-equivariant map $M_i\rightarrow M_{i,r}, m\mapsto
1\otimes m$ has dense image and is therefore surjective. Since
$M_i$ is irreducible the map is therefore bijective whenever
$r\geq r(F)$. It follows $M\car M_r(\UFe)$ for $r\geq r(F)$. Given
$g\in G$ we can use the $G$-equivariance of the sheaf $\tiM$ to
express the canonical map $M\rightarrow M_r(U_{g^{-1}F}^{(e)})$ as
the composite
$$M\stackrel{g\cdot}{\longrightarrow} M\car
M_r(\UFe)\stackrel{g^*}{\longrightarrow}M_r(U_{g^{-1}F}^{(e)}).$$
It is therefore bijective. Put
$r(M):=\max_{F\subseteq\overline{\sC}} r(F)$. Then $M\car
M_r(\UFe)$ for all $F\subset\sB$ and all $r\geq r(M)$. Identifying
$M$ with its constant sheaf on $\sB$ the natural morphism
$M\car\tiM$ is therefore an isomorphism for all $r\geq r(M)$. On
the other hand, (Lem. \ref{lem-zerodim}) gives a canonical
isomorphism
$$\cD_{\sB, \chi}^{an} \otimes_{U(\frg_K)_{\theta}} \tiM \car
\sDc\otimes_{\tiDt} \tiM.$$
\end{proof}

\section{A class of examples}

In this section, we invoke our usual assumptions, i.e. we assume
$L = \bbQ_{p}$, $e >  \max(e_{uni},e_{0}, e_{1})$ and $r \in
[r_{0}, 1)$.

\begin{para}

Let $\cO$ be the classical BGG-category for the reductive Lie
algebra $\frg_K$ relative to the choice of Borel subalgebra
$\frb_K$ (\cite{BGG2}). Since this category was originally defined
for complex semisimple Lie algebras only we briefly repeat what we
mean by it here. The category $\cO$ equals the full subcategory of
all (left) $U(\frg_K)$-modules consisting of modules $M$ such that
\begin{itemize}
    \item[(i)] $M$ is finitely generated as $U(\frg_K)$-module;
    \item[(ii)] the action of $\frt_K$ on $M$ is semisimple and locally finite;
    \item[(iii)] the action of $\frn_K$ on $M$ is locally finite.
\end{itemize}
Recall here that $\frt_K$ acts locally finite on some module $M$
if $U(\frt_K).m$ is finite dimensional for all $m\in M$ (similar
for $\frn_K$).

\vskip8pt



Let $\cO_{\rm alg}$ be the full abelian subcategory of $\cO$
consisting of those $U(\frg_K)$-modules whose $\frt_K$-weights are
integral, i.e. are contained in the lattice
$X^*(\bT)\subset\frt_K^*$. 

\end{para}

\begin{para}

In \cite{OrlikStrauchJH} the authors study an exact functor
$$M\mapsto \cF^G_B(M)$$ from $\cO_{\rm alg}$ to admissible locally
analytic $G$-representations. It maps irreducible modules to
(topologically) irreducible representations. The image of
$\cF^G_B$ comprises a wide class of interesting representations
containing all principal series representations and 'essentially'
all representations arising from homogeneous vector bundles on
$p$-adic symmetric spaces. In this final section we wish to study
the localizations of representations in this class. 
We restrict our attention to modules $M\in\cO_{\rm alg,\theta}$
having fixed central character $\theta$. Let $\chi\in\frt^*_K$ be
such that $\sigma(\chi)=\theta$.

\end{para}

\begin{para}

To start with let $U(\frg_K,B)$ be the smallest subring of $D(G)$
containing $U(\frg_K)$ and $D(B)$. The $\frb$-action on any
$M\in\cO_{\rm alg}$ integrates to an algebraic, and hence, locally
analytic $B$-action on $M$ and one has a canonical $D(G)$-module
isomorphism
$$\cF^G_B(M)'_b\car D(G)\otimes_{U(\frg_K,B)}M=:N$$ (loc.cit., Prop.
3.6). Of course, $N$ is a $D(G)_\theta$-module. We may therefore
consider its localization $\mathscr{L}_{r,\chi}(N)$ on $\sB$. We
recall that the stalk $\mathscr{L}_{r,\chi}(N)_z$ at a point $z$
is a quotient of $\kappa(z)\hat{\otimes} (\underline{N}_{r})_z$
(Prop. \ref{prop-stalks}) and therefore has its quotient topology. 
We finally say a morphism of sheaves to
$\mathscr{L}_{r,\chi}(N)$ has {\it dense image} if this holds
stalkwise at all points.

\vskip8pt

On the other hand, we may form $$GM:=K[G]\otimes_{K[B]} M.$$ It
may be viewed an $U(\frg_K)$-module via $x.(g\otimes m):= g\otimes
{\rm Ad}(g^{-1})(x).m$ for $g\in G, m\in M, x\in \frg_K$. Since
$K[G]$ is a free right $K[B]$-module, $GM$ equals the direct sum
of $U(\frg_K)$-submodules $gM:=g\otimes M$ indexed by a system of
coset representatives $g$ for $G/B$. Since the group $\bG$ is
connected, the adjoint action of $G=\bG(L)$ fixes the center
$Z(\frg_K)\subset U(\frg_K)$ (\cite[II \S6.1.5]{DemazureGabriel})
and therefore $GM$ still has central character $\theta$. Let us
consider its Beilinson-Bernstein module $\Delta(GM)$ over $X$. The
linear map $gM\car M, g\otimes m\mapsto m$ is an isomorphism and
equivariant with respect to the automorphism ${\rm Ad}(g^{-1})$ of
$U(\frg_K)$. It follows that, given an open subset $V\subseteq X$,
we have a linear isomorphism $\Delta(gM)(V)\car\Delta(M)(g^{-1}V)$
given by $\delta\otimes (g\otimes m)\mapsto g^*(\delta)\otimes m$
for a local section $\delta$ of $\cD_\chi$ and $m\in M$. Here
$g^*$ refers to the $G$-equivariant structure on $\cD_\chi$
(8.3.2). The same argument works for the analytifications
$\Delta^{an}(gM)$ and $\Delta^{an}(M)$. In particular, the stalks
$\Delta^{an}(gM)_z\simeq \Delta^{an}(M)_{g^{-1}z}$ are isomorphic
vector spaces for any $z\in\sB$ and any $g\in G$.
\begin{lemma}
We have
\[\Delta^{an}(M)|_{\sB}=0\Longleftrightarrow\Delta^{an}(GM)|_{\sB}=0.\]
\end{lemma}
\begin{proof}
Suppose $\Delta^{an}(M)|_{\sB}=0$. Let $g\in G$. For any $z\in\sB$
we compute $\Delta^{an}(gM)_z\simeq \Delta^{an}(M)_{g^{-1}z}=0$
whence $\Delta^{an}(gM)|_{\sB}=0$. This yields
$\Delta^{an}(GM)|_{\sB}=0,$ since $\Delta^{an}(\cdot)_{\sB}$
commutes with arbitrary direct sums. The converse is clear.
\end{proof}
\begin{lemma}
There is a canonical morphism of $\sD_{r,\chi}$-modules
$$\sD_{r,\chi}\otimes_{\cD^{an}_{\sB,\chi}}\Delta(GM)^{an}_{\sB}\longrightarrow\mathscr{L}_{r,\chi}(\cF_B^G(M)')$$
functorial in $M$ and with dense image.
\end{lemma}
\begin{proof}
The morphism is induced from the functorial map $$P: GM\rightarrow
D(G)\otimes_{U(\frg_K,B)}M=N$$ via the inclusions $K[B]\subset
D(B)$ and $K[G]\subset D(G)$. Let us show that the morphism has
dense image. We claim first that the map $P$ has dense image with
respect to the canonical topology on the coadmissible module $N$.
Let $G_0$ be the (hyper-)special maximal compact open subgroup of
$G$ equal to the stabilizer of the origin $x_0\in A$. Let
$B_0:=B\cap G_0$. The Iwasawa decomposition $G=G_0\cdot B$ implies
$K[G]=K[G_0]\otimes_{K[B_0]} K[B]$ and similarly for distributions
$D(\cdot)$. Let $G_0M:=K[G_0]\otimes_{K[B_0]} M$ and
$N_0:=D(G_0)\otimes_{U(\frg_K,B_0)}M$. Then $G_0M\simeq GM$ as
$K[G_0]$-modules and $N\simeq N_0$ as $D(G_0)$-modules via the
obvious maps. Write $D(G_0)=\varprojlim_r D_r(G_0)$ with some
Banach algebra completions $D_r(G_0)$. The map $P$ induces maps
$P_r: G_0M\rightarrow D_r(G_0)\otimes_{U(\frg_K,B_0)} M$. Since
$K[G_0]\subset D_r(G_0)$ is dense, the definition of the Banach
topology on the target implies that $P_r$ has dense image. Passing
to the limit over $r$ shows that $P$ has dense image. Let
$z\in\sB$. Then the map $P$ composed with the map $N\rightarrow
\underline{N}_{r,z}$ has dense image (\cite[\S3 Thm. A]{ST5}).
Now we are done: the map
$$\sD_{r,\chi,z}\otimes_{\cD^{an}_{\sB,\chi,z}}\Delta(GM)^{an}_{\sB,z}\rightarrow \mathscr{L}_{r,\chi}(N)_z,$$ pulled back to
$\Delta(GM)^{an}_{\sB,z}$, may be written as
$$((\kappa(z)\hat{\otimes}_L GM)/\frn_{\pi(z)}(\kappa(z)\hat{\otimes}_L
GM))_{\lambda-coinv.}\longrightarrow ((\kappa(z)\hat{\otimes}_L
\underline{N}_{r,z})/\frn_{\pi(z)}(\kappa(z)\hat{\otimes}_L
\underline{N}_{r,z}))_{\lambda-coinv.}$$ by (Thm. \ref{thm-BB}) and (Prop. \ref{prop-stalks}). Consequently, it has dense image.
\end{proof}
\end{para}
\begin{para}
We now look closer at the case $\theta=\theta_0$ and $\chi=\rho$.
Let $V:=ind_B^G(1)$ be the smooth induction of the trivial
character of $B$. Let us assume additionally that $e$ is large
enough so that the Schneider-Stuhler sheaf $\cVsim$ of its smooth
dual does not vanish (\cite[Thm. IV.4.1]{SchSt97}). The finitely
many irreducible modules in $\cO_{{\rm alg},\theta_0}$ are given
by the irreducible quotients $L_w$ of the Verma modules $M_w$ of
highest weight $-w(\rho)-\rho$ for $w\in W$. The cardinality of
the latter set of weights is $|W|$. As usual, $w_0$ denotes the
longest element in $W$. Let $w\in W$. Let $\cM_w$ and $\cL_w$ be
the Beilinson-Bernstein localizations over $X$ of $M_w$ and $L_w$
respectively. Let $\iota_w: X_w\hookrightarrow X$ be the inclusion
of the Bruhat cell $\bB w\bB/\bB$ into $X$ and let $\cO_{X_w}$ be
its structure sheaf with its natural (left) $D_{X_w}$-module
structure. Let $\cN_w=\iota_{w*} \cO_w$ be its $D$-module
push-forward to $X$. Since $\cO_{X_w}$ is a holonomic module and
$\iota$ is an affine morphism, $\cN_w$ may be viewed as an
$D_X$-module (rather than just a complex of such), cf.
\cite[3.4]{Hotta}.
\begin{prop}
Let $w\in W$ and $\cL_w^{an}$ be the analytification of $\cL_w$.
Then $\cL_w|_{\sB}\neq 0$ if and only if $w=w_0$.
\end{prop}
\begin{proof}
By loc.cit., Lem. 12.3.1 the sheaf $\cN_w$ has support contained
in $X_w$. By loc.cit., Prop. 12.3.2 (i) the module $\cL_w$ injects
into $\cN_w$. Now let $w\neq w_0$. Let $\eta\in X$ be the generic
point of $X$ and $\Xan_{\eta}$ the fibre of $\pi: \Xan\rightarrow
X$ over $\eta$. Since $\eta\notin X_w$ one has $(\cN_w)_\eta=0$
and therefore $\cN_w^{an}|_{\Xan_{\eta}}=0$. (Lem. \ref{lem-stalksarefields}) states that $\sB\subset\Xan_\eta$ whence
$\cL_w^{an}|_{\sB}=0$. The converse is clear: the module $L_{w_0}$
equals the trivial one-dimensional $U(\frg)$-module having
localization $\cL_{w_0}=\cO_X$ (e.g. by the Borel-Weil theorem).
Hence, $\cL^{an}_{w_0}|_{\sB}=\cO_{\sB}$.
\end{proof}
\begin{cor}
Let $w\in W$. Then $\mathscr{L}_{r,\rho}(\cF^G_B(L_w)')\neq 0$ if
and only if $w=w_0$.
\end{cor}
\begin{proof}
Let $w\neq w_0$. The preceding proposition together with the first
lemma above yields $\Delta^{an}(GL_w)|_{\sB}=0$. The second lemma
then yields $\mathscr{L}_{r,\rho}(\cF^G_B(L_w)')=0.$ Conversely,
let $w=w_0$. We have $\cF^G_B(L_{w_0})=ind_B^G(1)=V$, the smooth
induction of the trivial $B$-representation
(\cite{OrlikStrauchJH}). By choice of $e$ we have $\cVsim\neq 0$.
Let $z\in\sB$ be a point such that $\check{V}_{\Uze}\neq 0$. With
$N:=V'$ and $(\Uze)_{(r)}:=\Uze\cap U_r(\Uze)$, the (Prop.
\ref{prop-smoothcoinvariants}) yields a surjection
$$(\underline{N}_{r})_z=D_r(\Uze)\otimes_{D(\Uze)}
N=\check{V}_{(\Uze)_{(r)}}\longrightarrow \check{V}_{\Uze}$$
between the two spaces of coinvariants which implies
$(\underline{N}_{r})_z\neq 0$. It follows
$\mathscr{L}_{r,\rho}(N)_z=\kappa(z)\otimes_L(\underline{N}_{r})_z\neq
0$ (Prop. \ref{prop-stalks}) which means
$\mathscr{L}_{r,\rho}(N)|_{\sB}\neq 0$.
\end{proof}
Recall that any $U(\frg_K)$-module $M\in\cO$ is of finite length.
\begin{prop}
Let $M\in\cO_{{\rm alg},\theta_0}$. Let $n\geq 0$ be the
Jordan-H\"older multiplicity of the trivial representation in the
module $M$ and let $V=ind_B^G(1)$. There is a (noncanonical)
isomorphism of $\cO_{\sB}$-modules
$$\mathscr{L}_{\rho,r}(\cF^G_B(M)')\car
\mathscr{L}_{\rho,r}(V'^{\oplus n})$$ with both sides equal to
zero in case $n=0$.
\end{prop}
\begin{proof}
Let $\cVrsim$ be the constructible sheaf of $K$-vector spaces on
$\sB$ which is constructed in the same way as $\cVsim$ but using
the groups $(\UFe)_{(r)}$ instead of $\UFe$ for all facets $F$.
The very same arguments as in the case $r=r_0$ (Thm.
\ref{thm-compSchSt}) show that the $\cO_{\sB}$-module
$\mathscr{L}_{\rho,r}(V')$ is isomorphic to the module
$\cO_{\sB}\otimes_L\cVrsim$. In particular, it is a free
$\cO_{\sB}$-module.

We now prove the claim of the proposition by induction on $n$. Let
$n=0$. By exactness of the functors $\cF^G_B(.)'$ and
$\mathscr{L}_{\rho,r}$ a Jordan-H\"older filtration of $M$ induces
a filtration of $\mathscr{L}_{\rho,r}(\cF^G_B(M)')$ whose graded
pieces vanish by the preceding corollary. Thus
$\mathscr{L}_{\rho,r}(\cF^G_B(M)')=0.$ Let $n=1$. Using a
Jordan-H\"older filtration of $M$ and the case $n=0$ we may assume
that the trivial representation sits in the top graded piece of
$M$. Applying the case $n=0$ a second time gives the claim. Assume
now $n\geq 2$. Using again a Jordan-H\"older filtration of $M$ we
have an exact sequence
$$0\rightarrow M_1\rightarrow M\rightarrow M_2\rightarrow 0$$
in $\cO_{{\rm alg},\theta_0}$ where $M_i$ has multiplicity
$n_i\geq 1$. Applying the induction hypothesis to $M_1$ and $M_2$
yields an exact sequence of $\cO_{\sB}$-modules
$$0\rightarrow \mathscr{L}_{\rho,r}(V'^{\oplus n_1})\rightarrow \mathscr{L}_{\rho,r}(\cF^G_B(M)')
\rightarrow \mathscr{L}_{\rho,r}(V'^{\oplus n_2})\rightarrow 0.$$
By our first remark this sequence is (noncanonically) split. Since
$\mathscr{L}_{\rho,r}$ commutes with direct sums, this completes
the induction.
\end{proof}

\end{para}

\bibliographystyle{plain}
\bibliography{local_120919}

\end{document}